\documentclass{amsart}
\usepackage{amssymb,latexsym,mathrsfs,amsmath,tikz,tikz-cd}
\usepackage{amsthm}
\usepackage{graphicx}
\usepackage{pstricks}
\usepackage{graphics}
\usepackage{psfrag}
\usepackage{pstool}
\usepackage{tikz}
\usepackage{caption}
\usepackage{subcaption}
\usepackage{etex}

\usepackage[all,2cell]{xy}
\SelectTips{eu}{12}

\newcommand{\CKh}{\mathrm{CKh}}
\newcommand{\sgn}{\mathrm{sgn}}
\newcommand{\Kh}{\mathrm{Kh}}

\newcommand{\comp}{\mathrm{Comp}}
\newcommand{\oc}{\mathfrak{o}}
\newcommand{\wt}{\widetilde}
\newcommand{\Z}{\mathbb{Z}}
\newcommand{\bs}{\mathbf{s}}

\newcommand{\cC}{\mathscr{C}}
\newcommand{\cS}{\mathscr{S}}
\newcommand{\cL}{\mathscr{L}}
\newcommand{\cD}{\mathscr{D}}
\newcommand{\cF}{\mathscr{F}}
\newcommand{\cw}{\mathcal{C}}
\newcommand{\M}{\mathcal{M}}
\newcommand{\X}{\mathcal{X}}

\newcommand{\R}{\mathbb{R}}
\newcommand{\C}{\mathbb{C}}
\newcommand{\E}{\mathbb{E}}
\newcommand{\br}{\mathbf{r}}

\newcommand{\CP}{\mathbb{C}\mathbf{P}}

\newcommand{\SO}{\mathrm{SO}}
\newcommand{\UU}{\mathrm{U}}
\newcommand{\SL}{\mathfrak{sl}}

\newcommand{\Cat}{\mathscr{C}}
\newcommand{\gr}[1]{|{#1}|}
\newcommand{\id}{\mathrm{id}}
\newcommand{\ind}{\mathrm{ind}}
\newcommand{\Hom}{\mathrm{Hom}}
\newcommand{\KhSpace}{\mathcal{X}^\mathit{Kh}}

\newcommand{\F}{\mathbf{F}}

\newcommand{\RP}{\mathbb{R}\mathbf{P}}
\newcommand{\PP}{{}^+_+}
\newcommand{\PM}{{}^-_+}
\newcommand{\MP}{{}^+_-}
\newcommand{\MM}{{}^-_-}

\DeclareMathOperator{\Sq}{Sq}
\DeclareMathOperator{\St}{St}
\DeclareMathOperator{\sq}{sq}
\DeclareMathOperator{\Ob}{Ob}

\newtheorem{lemma}{Lemma}[section]

\newtheorem{proposition}[lemma]{Proposition}
\newtheorem{theorem}[lemma]{Theorem}

\theoremstyle{definition}
\newtheorem{remark}[lemma]{Remark}
\newtheorem{definition}[lemma]{Definition}
\newtheorem{example}[lemma]{Example}

\begin{document}
\parindent0em
\setlength\parskip{.1cm}

\title[$\SL_n$ homotopy type for matched diagrams]{An $\SL_n$ stable homotopy type for matched diagrams}

\thanks{AL and DS were both supported by EPSRC grant EP/M000389/1, DJ was supported by an EPSRC graduate studentship.}

\author[Dan Jones]{Dan Jones}
\address{Department of Mathematical Sciences\\ Durham University}
\email{daniel.jones@durham.ac.uk}

\author[Andrew Lobb]{Andrew Lobb}
\address{Department of Mathematical Sciences\\ Durham University}
\email{andrew.lobb@durham.ac.uk}

\author[Dirk Sch\"utz]{Dirk Sch\"utz}
\address{Department of Mathematical Sciences\\ Durham University}
\email{dirk.schuetz@durham.ac.uk}

\begin{abstract}
There exists a simplified Bar-Natan Khovanov complex for open 2-braids.  The Khovanov cohomology of a knot diagram made by gluing tangles of this type is therefore often amenable to calculation.  We lift this idea to the level of the Lipshitz-Sarkar stable homotopy type and use it to make new computations.

Similarly, there exists a simplified Khovanov-Rozansky $\SL_n$ complex for open 2-braids with oppositely oriented strands and an even number of crossings.  Diagrams made by gluing tangles of this type are called matched diagrams, and knots admitting matched diagrams are called bipartite knots.  To a pair consisting of a matched diagram and a choice of integer $n \geq 2$, we associate a stable homotopy type.  In the case $n=2$ this agrees with the Lipshitz-Sarkar stable homotopy type of the underlying knot.  In the case $n \geq 3$ the cohomology of the stable homotopy type agrees with the $\SL_n$ Khovanov-Rozansky cohomology of the underlying knot.

We make some consistency checks of this $\SL_n$ stable homotopy type and show that it exhibits interesting behaviour.  For example we find a $\CP^2$ in the $\SL_3$ type for some diagram, and show that the $\SL_4$ type can be interesting for a diagram for which the Lipshitz-Sarkar type is a wedge of Moore spaces.
\end{abstract}

\maketitle

\tableofcontents

\section{Introduction}
\label{sec:introduction}
In \cite{LipSarKhov}, Lipshitz-Sarkar construct a stable homotopy type associated to a knot.  They show that the cohomology of this object recovers Khovanov cohomology\footnote{We use \emph{Khovanov cohomology} here rather than \emph{Khovanov homology} for reasons discussed in Subsection \ref{subsec:conventions}}.  The construction of this stable homotopy type proceeds along lines laid down by Cohen-Jones-Segal \cite{CJS}.  Allowing ourselves a good measure of imprecision, the idea can be summarized as follows.

To a knot diagram $D$ of a knot $K$ Khovanov associated a combinatorial bigraded cochain complex $\CKh(D)$ \cite{kh1}.  The cohomology of this complex (Khovanov cohomology) is an invariant $\Kh(K)$ of $K$, exhibiting as its graded Euler characteristic a knot polynomial - the Jones polynomial - and is projectively functorial for knot cobordisms.  These properties (as well as host of spectral sequences starting from Khovanov cohomology and abutting to various Floer-theoretic invariants of the knot $K$), invite one to think of Khovanov cohomology as a type of Floer homology.

Following this mental yoga further, one might think of the standard generators of the Khovanov cochain complex $\CKh(D)$ as being akin to critical points of a Floer functional (or, more simply, of a Morse function).  Then the differential describes $0$-dimensional moduli spaces of flowlines between those critical points.  If one can make a good guess as to what the higher-dimensional spaces of flowlines might be, then  one can follow the recipe of Cohen-Jones-Segal \cite{CJS} and associate to a knot diagram a stable homotopy type whose cohomology recovers Khovanov cohomology.  Finally one hopes that what one has constructed is invariant under the Reidemeister moves.

There is of course much difficulty to be overcome in making the previous two paragraphs yield up an honest stable homotopy type invariant.  In particular, the input to the Cohen-Jones-Segal machine is more complicated than we have made out, in fact it takes the form of a \emph{framed embedded flow category}, and constructing such a thing takes the majority of the paper \cite{LipSarKhov}.  Nevertheless, one should think of framed flow categories as being closely related to Morse theory.

Our motivation for writing this paper was to attempt to extend this construction to the case of $\SL_n$ Khovanov-Rozansky cohomology (for $n \geq 2$) in a way that might enable us and others to make computations.  The Lipshitz-Sarkar stable homotopy type should appear as the case $n=2$ of this extension.  Since there is a notion of stabilization of these cohomologies as $n \rightarrow \infty$, something similar should be true for the stable homotopy types.

Given the input of a special type of knot diagram (a \emph{matched} diagram) we show how to create a stable homotopy type for each $n \geq 2$ whose cohomology recovers the $\SL_n$ Khovanov-Rozansky cohomology of the underlying knot.  In the case $n=2$ we recover the Lipshitz-Sarkar stable homotopy type.  As well as obtaining the `correct' thing when $n=2$, we give some consistency checks suggesting that the space we construct is the right space.

The stable homotopy type we define is eminently computable and indeed we give some interesting computations at the end of this paper.  In fact, the construction even makes the Lipshitz-Sarkar stable homotopy type more computable for certain knots (for example for pretzel knots) since it reduces the number of objects in the corresponding flow category.  We do not show in this paper that the stable homotopy type for $n > 2$ is independent of the choice of diagram, but this is something that we shall return to in future work.

\subsection{Statement of results}
\label{subsec:results}
In Figure \ref{2twist} we define an \emph{elementary tangle diagram} of index $r$.  Given a link diagram $D$, there is of course always a decomposition of $D$ into such elementary pieces - for example one piece of index $1$ for each crossing of $D$.  When a link diagram $D$ comes together with such a decomposition we shall write it as $D_\br$ where $\br = (r_1, \ldots, r_m) \in \Z^m$ and the decomposition is into $m$ elementary tangles where the $i$th tangle has index $r_i$.

\begin{definition}
\label{def:glued_d}
We shall refer to the data of a link diagram $D$ together with such a choice of decomposition into elementary tangles as a \emph{glued diagram} $D_\br$.
\end{definition}

\begin{figure}
	\centerline{
		{
			\psfrag{ldots}{$\ldots$}
			\psfrag{L}{$L$}
			\psfrag{R}{$R$}
			\includegraphics[height=0.6in,width=3.5in]{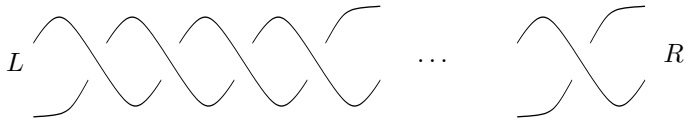}
		}}
		\caption{We show what we mean by an \emph{elementary tangle diagram} of index $r > 0$.  In this figure $r$ is the number of crossings, and $L$ and $R$ denote a choice of left and right endpoints.  The mirror image of this tangle would be a diagram for negative $r$.}
		\label{2twist}
\end{figure}

From now on we shall consider all diagrams to be oriented.  Given a choice of integer $n \geq 2$ and a glued diagram $D_\br$ (we shall write $D$ when we intend to ignore the decomposition), we shall define a framed embedded flow category $\cL^n(D_\br)$.  The objects of the flow category are cohomologically graded, and the category splits as a disjoint union of flow categories along an quantum integer grading.

There is a space $\X^n(D_\br)$ associated to $\cL^n(D_\br)$ by the Cohen-Jones-Segal construction - this takes the form of a stable homotopy type.  The cohomology of $\X^n(D_\br)$ is bigraded by the underlying cohomological degree and quantum grading of the objects of $\cL^n(D_\br)$.

Following suitable degree choices we have

\begin{theorem}
	\label{thm:Lip-Sar_equivalence}
	The space $\X^2(D_\br)$ agrees with the space associated to the diagram $D$ by the Lipshitz-Sarkar construction.
\end{theorem}

In the case $n > 2$ we restrict our attention to \emph{matched diagrams}.

\begin{definition}
	\label{defn:matched_diagram}
	A glued diagram $D_\br$ is called \emph{matched} if each coordinate of $\br$ is even and each elementary tangle in the decomposition of $D$ has oppositely oriented strands.  A link admitting such a diagram is called \emph{bipartite}.
\end{definition}

\begin{remark}
Note that if $D_\br$ is a knot diagram then the evenness condition implies the orientation condition.
\end{remark}

Khovanov cohomology is a categorification of the Jones polynomial, which arises from the fundamental representation of $\SL_2$ via the Reshetikhin-Turaev construction.  The categorification of the corresponding polynomial for $\SL_n$ is Khovanov-Rozansky cohomology \cite{khr1}.  Again, this is a bigraded link invariant and we shall write it as $H_{\SL_n}(D;G)$ for $G$ an abelian group.  That one can have general $G$-coefficients (as opposed to just complex coefficients) was noted first in the case $n=3$ by Khovanov in his foam categorification.  The case of $n>3$ was only recently shown by Queffelec-Rose \cite{QuefRose} to have a foam interpretation valid with arbitrary coefficents.

With suitable grading shifts we have

\begin{theorem}
	\label{thm:cohom_of_Xn}
	If $D_\br$ is a matched link diagram then the cohomology with complex coefficents of $\X^n(D_\br)$ is isomorphic to $H_{\SL_n}(D;\C)$ as a bigraded complex vector space.
\end{theorem}

The reason for restricting this theorem to complex coefficients is not so serious.  If one wanted to prove the theorem for arbitrary coefficients one would have to verify Krasner's theorem \cite{Krasner} on simplified Khovanov-Rozansky cochain complexes for matched diagrams over the integers (Krasner's result is over the complex numbers).  There is ample computational evidence that indeed Krasner's theorem will hold in this case, and such an extension using Queffelec-Rose's foamy definition of Khovanov-Rozansky cohomology over the integers should not be difficult, although we do not undertake it here.

Although the flow category $\cL^n(D_\br)$ depends heavily on the choice of decomposition of $D$ into elementary tangles we show that

\begin{proposition}
	\label{prop:matched_independence}
	If $D_\br$ is a matched link diagram and $n \geq 2$, then $\X^n(D_\br)$ is independent of the decomposition of $D$ into elementary tangles, and so can be written $\X^n(D)$.
\end{proposition}

\begin{figure}
	\centerline{
		{
			\psfrag{D}{$D$}
			\psfrag{D'}{$D'$}
			\includegraphics[height=1.5in,width=2.2in]{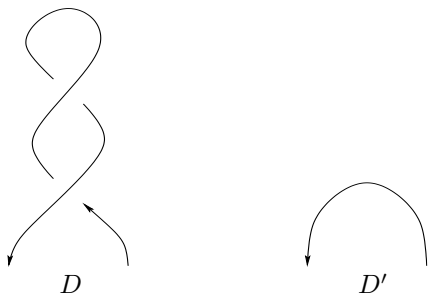}
		}}
		\caption{We illustrate an extended version of Reidemeister move I for matched diagrams.}
		\label{fig:sln_R1}
	\end{figure}

\begin{figure}
	\centerline{
		{
			\psfrag{D}{$D$}
			\psfrag{D'}{$D'$}
			\includegraphics[height=2.2in,width=2.2in]{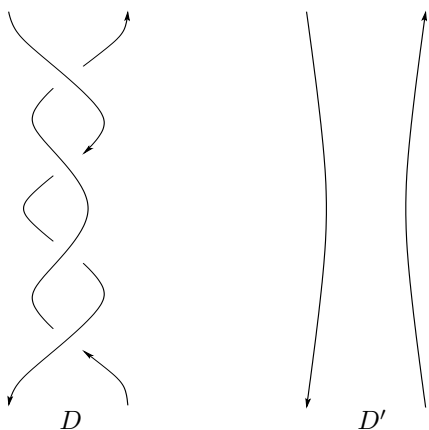}
		}}
		\caption{We illustrate an extended version of Reidemeister move II for matched diagrams.}
		\label{fig:sln_R2}
\end{figure}

Furthermore, we derive relatives of Reidemeister moves I and II for matched diagrams.

\begin{proposition}
	\label{prop:sln_R1+R2}
In Figures \ref{fig:sln_R1} and \ref{fig:sln_R2} we give matched diagrams $D$ and $D'$ that differ locally.  For such diagrams we have that $\X^n(D) \simeq \X^n(D')$.  The corresponding relations for the mirror images also hold.
\end{proposition}

Hence we have a multitude of evidence that our definition for this restricted class of diagrams is giving us the `correct' $\SL_n$ stable homotopy type: the construction gives the Lipshitz-Sarkar space for $n=2$ (see also Remark \ref{rem:lens}); the cohomology of the space is the correct thing (for general $n$ using complex coefficients, and for $n=3$ verified over the integers for a host of examples); the space constructed is independent of the decomposition of the matched diagram into elementary tangles; and the space is invariant under extended Reidemeister moves I and II.

Unfortunately, there does not exist a genuine Reidemeister calculus to move between different matched diagrams of the same bipartite knot.  It is probably not the case, for example, that just the moves of Figures \ref{fig:sln_R1} and \ref{fig:sln_R2} suffice!  Furthermore it would be preferable to have a construction of an $\SL_n$ stable homotopy type for all links (not just bipartite links) which takes as input a link diagram and is invariant under the Reidemeister moves.  We shall return to this question in a later paper.

We include at the end of this paper a section of computations, undertaken both by hand and by computer.  Of note is the detection of a $\CP^2$ summand in the $\SL_3$ stable homotopy type for the pretzel link $P(-2,2,2)$.  A $\CP^2$ summand has yet to be detected in the Lipshitz-Sarkar stable homotopy type, although this may be due only to lack of computations.

Furthermore, it has long been known that the Khovanov cohomology of a knot may be thin while the HOMFLYPT cohomology is thick.  At the level of spaces one should ask then if there is an example of a knot whose $\SL_2$ stable homotopy type is a wedge of Moore spaces while its $\SL_n$ stable homotopy type is more interesting for some $n > 2$.  Such an example is provided by the pretzel knot $P(2,-3,5)$ whose $\SL_4$ stable homotopy type induces non-trivial second Steenrod squares.

Finally, work by Baues and by Baues-Hennes \cite{baues, bauhen} reveals ways to detect combinatorially more stable homotopy types than Moore spaces, $\CP^2$, $\RP^5/ \RP^2$, $\RP^4/\RP^1$, and $\RP^2 \wedge \RP^2$ (and wedge sums of these spaces).  This is applied in the case of the Lipshitz-Sarkar stable homotopy type to the torus knots $T(4,5)$ and $T(4,7)$ after finding glued diagrams of each formed of a relatively low number of elementary tangles.

\subsection{Conventions}
\label{subsec:conventions}
In this paper we try to follow accepted conventions as far as we can.  One way in which we shall differ from usual (although not universal) practice is to refer to the invariants due to Khovanov and to Khovanov-Rozansky as \emph{cohomology} theories rather than \emph{homology} theories.  Since the differential in the complexes constructed does indeed increase the $i$-grading this makes sense.  It is also consistent with the spaces defined by Lipshitz-Sarkar and by us since the associated cohomology theories are the usual singular cohomologies of the spaces themselves.

In Figure \ref{fig:posneg} we fix what we mean by a positive or negative oriented crossing.

\begin{figure}
\centerline{
	{
\psfrag{ldots}{$\ldots$}
\psfrag{+}{$+$}
\psfrag{-}{$-$}
\includegraphics[height=1in,width=2.4in]{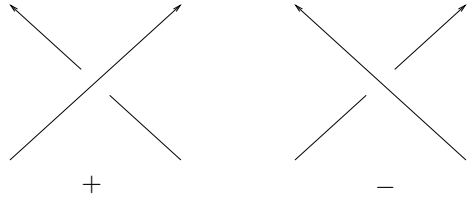}
}}
\caption{We show what we mean by positive and negative oriented crossings.}
\label{fig:posneg}
\end{figure}

In both Khovanov and in Khovanov-Rozansky cohomology, the cohomology of the positive trefoil is supported in non-negative cohomological degrees.  However, in Khovanov cohomology the invariant is supported in positive quantum degrees while in $sl_n$ Khovanov-Rozansky cohomology (even in the $sl_2$ case) the invariant is supported in negative quantum degrees.  This is a consequence of the grading assigned to $x$ in the underlying Frobenius algebra $\C[x]/x^2$, which is taken to be $+2$ for Khovanov-Rozansky cohomology and $-2$ for Khovanov cohomology.  This is unfortunate but we stick to this convention in the current paper.

We shall sometimes be a little imprecise about absolute cohomological and quantum gradings.  Especially in proofs, the extra decorations with shifts (depending on the writhe) tends to obscure the argument.  When we wish to be more precise we shall use the postscript $[s]\{t\}$ following a complex to denote a shift by $s$ in the cohomological direction and by $t$ in the quantum direction.

\subsection{Plan of the paper}
\label{subsec:plan}
We begin in Section \ref{sec:cjs} with a review of framed flow categories and how the Cohen-Jones-Segal construction associates stable homotopy types to such things.  In Section \ref{sec:assoc_flow_cat} we describe how we build a flow category given the data of a glued link diagram together with a choice of integer $n \geq 2$.  Section \ref{sec:sln_match_diag} considers the properties of the associated stable homotopy type.  In particular, the stable homotopy type returned for $n=2$ agrees with the Lipshitz-Sarkar stable homotopy type, while for $n > 2$ the stable homotopy type for a matched diagram has cohomology agreeing with $\SL_n$ Khovanov-Rozansky cohomology.  Finally in this section we provide some indictions discussed in the introduction that the stable homotopy type is the `correct' space for $n > 2$.  In Section \ref{sec:steen_squ} we discuss how to compute the second Steenrod square both in general flow categories and adjusted to our specific situations.  The paper closes with Section \ref{sec:examples} in which we provide explicit computations of several interesting stable homotopy types.

\section{Framed flow categories and the Cohen-Jones-Segal construction}
\label{sec:cjs}

\subsection{Framed flow categories}
\label{subsec:ffc}
To define flow categories we need a sharpening of a manifold with corners that goes back to \cite{ich}. Recall that a smooth manifold with corners is defined in the same way as an ordinary smooth manifold, except that the differentiable structure is now modelled on the open subsets of $([0,\infty))^k$.

So if $X$ is a smooth manifold with corners and $x\in X$ is represented by $(x_1,\ldots,x_k)\in ([0,\infty))^k$, let $c(x)$ be the number of coordinates in this $k$-tuple which are $0$. Denote by
\begin{eqnarray*}
 \partial^iX&=&\{x\in X\,|\,c(x)=i\}
\end{eqnarray*}
the codimension-$i$-boundary. Note that $x$ belongs to at most $c(x)$ different connected components of $\partial^1 X$. We call $X$ a \em smooth manifold with faces \em if every $x\in X$ is contained in the closure of exactly $c(x)$ components of $\partial^1X$. A \em connected face \em is the closure of a component of $\partial^1 X$, and a \em face \em is any union of pairwise disjoint connected faces (including the empty face). Note that every face is itself a manifold with faces. We define the boundary of $X$, $\partial X$, as the closure of $\partial^1 X$.

\begin{definition}
 Let $k,n$ be non-negative integers and $X$ a smooth manifold with faces. An \em $n$-face structure \em for $X$ is an ordered $n$-tuple $(\partial_1 X,\ldots, \partial_n X)$ of faces of $X$ such that
\begin{enumerate}
 \item $\partial_1 X\cup \cdots \cup \partial_nX = \partial X$.
 \item $\partial_i X\cap \partial_j X$ is a face of both $\partial_i X$ and $\partial_j X$ for $i\not=j$.
\end{enumerate}
A smooth manifold with faces $X$ together with an $n$-face structure is called a \em smooth $\langle n \rangle$-manifold\em.

If $a=(a_1,\ldots,a_n)\in \{0,1\}^n$, we define
\begin{eqnarray*}
 X(a)&=&\bigcap_{i\in \{j\,|\,a_j=0\}} \partial_i X
\end{eqnarray*}
and note that this is an $\langle |a| \rangle$-manifold, where $|a|=a_1+\cdots+ a_n$. If $a=(1,\ldots,1)$ we interpret the empty intersection as $X$.
\end{definition}

There is an obvious partial order $\leq£$ on $\{0,1\}^n$ such that $X(a)\subset X(b)$ for $a\leq b$.

\begin{definition} \label{euclidcorners}
 Given an $(n+1)$-tuple $\mathbf{d}=(d_0,\ldots,d_n)$ of non-negative integers, let
\begin{eqnarray*}
 \E^\mathbf{d}&=&\R^{d_0}\times [0,\infty) \times \R^{d_1}\times [0,\infty) \times \cdots \times [0,\infty) \times \R^{d_n}.
\end{eqnarray*}
Furthermore, if $0\leq a < b \leq n+1$, we denote $\E^\mathbf{d}[a:b]=\E^{(d_a,\ldots,d_{b-1})}$.
\end{definition}

We can turn $\E^\mathbf{d}$ into an $\langle n\rangle$-manifold by setting
\begin{eqnarray*}
 \partial_i \E^\mathbf{d}&=& \R^{d_0}\times \cdots \times \R^{d_{i-1}} \times \{0\} \times \R^{d_i} \times \cdots \times \R^{d_n}.
\end{eqnarray*}
We will refer to this boundary part as the \em $i$-boundary\em. In the case of $\E^\mathbf{d}[a:b]$ we also refer to the set
\begin{eqnarray*}
 \partial_{i-a} \E^\mathbf{d}[a:b]&=& \R^{d_a}\times \cdots \times \R^{d_{i-1}} \times \{0\} \times \R^{d_i} \times \cdots \times \R^{d_{b-1}}
\end{eqnarray*}
as the $i$-boundary, although strictly speaking this should be the $(i-a)$-boundary.

\begin{definition}
 A \em neat immersion $\imath$ \em of an $\langle n \rangle$-manifold is a smooth immersion $\imath\colon X \looparrowright \E^\mathbf{d}$ for some $\mathbf{d}\in \Z^{n+1}$ such that
\begin{enumerate}
 \item For all $i$ we have $\imath^{-1}(\partial_i\E^\mathbf{d})=\partial_i X$.
 \item The intersection of $X(a)$ and $\E^\mathbf{d}(b)$ is perpendicular for all $b<a$ in $\{0,1\}^n$.
\end{enumerate}
A \em neat embedding \em is a neat immersion that is also an embedding.

Given a neat immersion $\imath\colon X \looparrowright \E^\mathbf{d}$ we have a normal bundle $\nu_{\imath(a)}$ for each immersion $\imath(a)\colon X(a) \looparrowright \E^\mathbf{d}(a)$ as the orthogonal complement of the tangent bundle of $X(a)$ in $T\E^\mathbf{d}(a)$.
\end{definition}

\begin{definition}
 A \em flow category \em is a pair $(\Cat,\gr{\cdot})$ where $\Cat$ is a category with finitely many objects $\Ob=\Ob(\Cat)$ and $\gr{\cdot}\colon \Ob \to \Z$ is a function, called the \em grading\em, satisfying the following:
\begin{enumerate}
 \item $\Hom(x,x)=\{\id\}$ for all $x\in \Ob$, and for $x\not=y \in \Ob$, $\Hom(x,y)$ is a smooth, compact $(\gr{x}-\gr{y}-1)$-dimensional $\langle \gr{x}-\gr{y}-1\rangle$-manifold which we denote by $\mathcal{M}(x,y)$.
 \item For $x,y,z\in \Ob$ with $\gr{z}-\gr{y}=m$, the composition map
$$\circ\colon \mathcal{M}(z,y) \times \mathcal{M}(x,z) \to \mathcal{M}(x,y)$$
is an embedding into $\partial_m\mathcal{M}(x,y)$. Furthermore,
\begin{eqnarray*}
 \circ^{-1}(\partial_i \mathcal{M}(x,y))&=&\left\{ \begin{array}{lr}
                                             \partial_i \mathcal{M}(z,y)\times \mathcal{M}(x,z) & \mbox{for }i<m \\
                                             \mathcal{M}(z,y)\times \partial_{i-m}\mathcal{M}(x,z) & \mbox{for }i>m
                                            \end{array}
\right.
\end{eqnarray*}
\item For $x\not= y\in \Ob$, $\circ$ induces a diffeomorphism
\begin{eqnarray*}
 \partial_i\mathcal{M}(x,y)&\cong & \coprod_{z,\,\gr{z}=\gr{y}+i} \mathcal{M}(z,y) \times \mathcal{M}(x,z).
\end{eqnarray*}

\end{enumerate}
We also write $\mathcal{M}_\Cat(x,y)$ if we want to emphasize the flow category. The manifold $\mathcal{M}(x,y)$ is called the \em moduli space from $x$ to $y$\em, and we also set $\mathcal{M}(x,x)=\emptyset$.
\end{definition}

Note that $\mathcal{M}(x,y)=\emptyset$ whenever $\gr{y}\geq \gr{x}$, as the empty set is the only negative dimensional manifold.

\begin{example}
 Let $f\colon M\to \R$ be a Morse function on a closed manifold $M$, and let $v$ be a Morse-Smale gradient for $f$, meaning that all stable and unstable manifolds of $v$ intersect transversally. We then define the \em Morse flow category $\Cat_f$ \em as follows.

The objects are exactly the critical points of $f$, with grading given by the index. If $p$ is a critical point, define the stable and unstable manifolds with respect to the positive gradient flow, so that
\begin{eqnarray*}
 W^s(p)&=&\{x\in M\,|\, \lim_{t\to \infty}\gamma_x(t)=p\}
\end{eqnarray*}
where $\gamma_x\colon \R\to M$ is the flowline of $v$ with $\gamma_x(0)=x$, and
\begin{eqnarray*}
 W^u(p)&=&\{x\in M\,|\, \lim_{t\to -\infty}\gamma_x(t)=p\}.
\end{eqnarray*}
Given two different critical points $p$ and $q$, let
\begin{eqnarray*}
\tilde{\mathcal{M}}(p,q)&=& W^s(p)\cap W^u(q)/\R,
\end{eqnarray*}
where $\R$ acts on this intersection using the flow. Then $\tilde{\mathcal{M}}(p,q)$ can be embedded into $W^s(p)\cap f^{-1}(\{a\})$ for every $a\in (f(q),f(p))$, and it follows from the transversality condition that this is a smooth manifold of dimension $\ind(p)-\ind(q)-1$. To get the moduli space $\mathcal{M}(p,q)$ we compactify this space by adding all the broken flowlines between $p$ and $q$ using \cite[Lm.2.6]{AusBra}.
\end{example}

\begin{definition} \label{def:neat_immersion}
 Let $\Cat$ be a flow category and $\mathbf{d}=(d_A,\ldots,d_{B-1})\in \Z^{B-A}$ a sequence of non-negative integers with $A\leq \gr{x} \leq B$ for all $x\in \Ob(\Cat)$. A \em neat immersion $\imath$ \em of the flow category $\Cat$ relative $\mathbf{d}$ is a collection of neat immersions $\imath_{x,y}\colon \mathcal{M}(x,y)\looparrowright \E^\mathbf{d}[\gr{y}:\gr{x}]$ for all objects $x,y$ such that for all objects $x,y,z$ and all points $(p,q)\in \mathcal{M}(z,y)\times \mathcal{M}(x,z)$ we have
\begin{eqnarray*}
 \imath_{x,y}(p\circ q)&=&(\imath_{z,y}(p),0,\imath_{x,z}(q)).
\end{eqnarray*}
The neat immersion $\imath$ is called a \em neat embedding\em, if for all $i,j$ with $A\leq j<i\leq B$ the induced map
$$
\imath_{i,j}\colon \coprod_{(x,y),\gr{x}=i,\gr{y}=j}\mathcal{M}(x,y) \to \E^\mathbf{d}[j:i]
$$
is an embedding.
\end{definition}

\begin{definition}
 Let $\imath$ be a neat immersion of a flow category $\Cat$ relative $\mathbf{d}$. A \em coherent framing $\varphi$ \em of $\imath$ is a framing for the normal bundle $\nu_{\imath_{x,y}}$ for all objects $x,y$, such that the product framing of $\nu_{\imath_{z,y}}\times \nu_{\imath_{x,z}}$ equals the pullback framing of $\circ^\ast \nu_{\imath_{x,y}}$ for all objects $x,y,z$.

A \em framed flow category \em is a triple $(\Cat,\imath,\varphi)$, where $\Cat$ is a flow category, $\imath$ a neat immersion and $\varphi$ a coherent framing of $\imath$.
\end{definition}

Given a framed flow category $(\Cat,\imath,\varphi)$, we can associate a chain complex $C_\ast(\Cat,\imath,\varphi)$ as follows. The $n$-th chain group is the free abelian group generated by the objects with grading $n$, and if $x,y\in Ob$ are objects with $|x|=|y|+1=n$, the coefficient in the boundary between $x$ and $y$ is the sign of the $0$-dimensional compact moduli space $\mathcal{M}(x,y)$ obtained from the framing in $\R^{d_{n-1}}=\E^\mathbf{d}[n-1:n]$. The condition of a coherent framing ensures that we get indeed a chain complex.

Dually we can also associate a cochain complex $C^\ast(\Cat,\imath,\varphi)$.

There are various ways to think of a framing of an immersed manifold. For our constructions it will be useful to think of a framing of an immersion $\imath\colon \mathcal{M}(a,b)\to \E^{\bf d}[|b|:|a|]$ as an immersion 
$$\varphi\colon \mathcal{M}(a,b)\times [-\varepsilon,\varepsilon]^{d_{|b|}+\cdots d_{|a|-1}}\to \E^{\bf d}[|b|:|a|]$$
such that $\varphi(x,0)=\imath(x)$ for all $x\in \mathcal{M}(a,b)$.

\subsection{The associated stable homotopy type}
\label{subsec:CJS_space}
As was briefly alluded to in the introduction, there is a process that allows one to construct, from a given framed flow category $\Cat$, a CW complex $|\Cat|$. This CW complex is constructed in such a way that its cellular cochain complex $C^{*}(|\Cat|)$ is isomorphic (after some grading shift) to the cochain complex $C^{*}(\Cat)$ obtained from $\Cat$. An outline of the construction of $|\Cat|$ was first given by Cohen-Jones-Segal (inspired by Franks \cite{Franks}) in attempt to achieve a spectrum (or space-level refinement) for Floer homology. 
As expressed in \cite{CJS}, their attempt was not entirely successful but they do outline a detailed recipe for constructing a CW complex from any given framed flow category; a recipe that was implemented successfully in \cite{LipSarKhov} to produce such a spectrum for Khovanov cohomology, namely the Lipshitz-Sarkar stable homotopy type. The immediate output of the Cohen-Jones-Segal machine is a CW complex that we shall define in this section. It should be noted that the Lipshitz-Sarkar 
stable homotopy type is defined as a (de-)suspension of this output where the input is a particular Khovanov flow category, constructed in \cite{LipSarKhov}.

\begin{definition} \label{cwcomplex}
Let $(\Cat, \imath , \varphi)$ be a framed flow category relative ${\bf d}$. For an arbitrary object $a$ in ${\rm Ob}(\Cat)$ of degree $m$, recall that for each object $b$ in ${\rm Ob}(\Cat)$ of degree $n<m$, we have a framed neat embedding
\[
\imath_{a,b} : \mathcal{M}(a,b) \times [-\varepsilon, \varepsilon]^{d_n + \cdots + d_{m-1}} \rightarrow [-R,R]^{d_n} \times [0,R] \times \cdots \times [0,R] \times [-R,R]^{d_{m-1}}
\]
where $R$ is chosen to be large enough that all moduli spaces $\mathcal{M}(a,b)$ can be embedded in this way. Moreover, choose $B<A \in \Z$ as in Definition \ref{def:neat_immersion} so that every object $a \in {\rm Ob}(\Cat)$ satisfies $B \leq |a| \leq A$. The CW complex $|\Cat|$ consists of one $0$-cell (the basepoint) and one $(d_B + \cdots + d_{A-1} -B + m)$-cell $\mathcal{C}(a)$ for every object $a$ of $\Cat$ defined as
\[
[0,R] \times [-R,R]^{d_{B}} \times \cdots \times [-R,R]^{d_{m-1}} \times \{ 0 \} \times [- \varepsilon , \varepsilon]^{d_{m}} \times \{ 0 \} \times \cdots \times \{0\} \times [-\varepsilon , \varepsilon]^{d_{A-1}} {\rm .}
\]

Each cell $\mathcal{C}(a)$ is considered a subset of a different copy of the ambient space $\R_+ \times \R^{d_B} \times \cdots \times \R_+ \times \R^{d_{A-1}}$. The neat embedding $\imath$ can be used to identify particular subsets 
\begin{equation} \label{product_in_boundary_cell}
\mathcal{M}(a,b) \times \mathcal{C}(b) \cong \mathcal{C}_b(a) \subset \partial_n \mathcal{C}(a)
\end{equation}
in the following way:
\begin{align*}
\mathcal{C}_b(a) = & [0,R] \times [-R,R]^{d_{B}} \times \cdots \times [-R,R]^{d_{n-1}} \times \{ 0 \} \times \\
& \imath_{a,b} \big( \mathcal{M}(a,b) \times [-\varepsilon, \varepsilon]^{d_n + \cdots + d_{m-1}} \big) \times \\
& \{ 0 \} \times [-\varepsilon, \varepsilon]^{d_m} \times \cdots \times \{0\} \times [-\varepsilon , \varepsilon]^{d_{A-1}} \\
& \subset \partial \mathcal{C}(a) {\rm .}
\end{align*}

It will be useful to introduce notation for this identification by letting
\begin{equation} \label{gamma_identification}
\Gamma_{a,b}: \mathcal{M}(a,b) \times \mathcal{C}(b) \rightarrow \partial_n \mathcal{C}(a)
\end{equation}
be the identification $\mathcal{M}(a,b) \times \mathcal{C}(b) \cong \mathcal{C}_b(a)$. Let
\begin{equation} \label{C_shift}
C = d_B + \cdots + d_{A-1} -B {\rm .}
\end{equation}
Then the attaching map for each cell $\partial \mathcal{C}(a) \rightarrow |\Cat|^{(C+m-1)}$ is defined via the Thom construction for each embedding into $\partial \mathcal{C}(a)$ simultaneously. That is, for each subset $\mathcal{M}(a,b) \times \mathcal{C}(b) \cong \mathcal{C}_b(a) \subset \partial \mathcal{C}(a)$, the attaching map projects to $\mathcal{C}(b)$ (which carries trivialisation information), and sends the rest of the boundary $\partial \mathcal{C}(a) \setminus \bigcup_{b} \mathcal{C}_b(a)$ to the basepoint.
\end{definition}
The fact that this construction is well-defined is shown in \cite[Lemma 3.25]{LipSarKhov} which also describes how the attaching maps give a natural isomorphism of chain complexes.

The isomorphism type of $|\Cat|$ is shown to be independent of the choice of real numbers $R$ and $\varepsilon$ in \cite[Lemma 3.25]{LipSarKhov} and, by considering a one-parameter family of framed neat embeddings between two perturbations $(\imath_0 , \varphi_0)$ and $(\imath_1 , \varphi_1)$ of $(\imath, \varphi)$, it can be shown that the CW complexes $|\Cat|_{\imath_0, \varphi_0}$ and $|\Cat|_{\imath_1, \varphi_1}$ are isomorphic (also \cite[Lemma 3.25]{LipSarKhov}). A choice of different $A$, $B$, and ${\bf d}$ gives rise to a stably homotopy equivalent CW complex (see \cite[Lemma 3.26]{LipSarKhov}) that is a suspension of the original CW complex a number of times.

This construction of $|\Cat|$ is also shown to agree with the construction of \cite{CJS} in \cite[Prop3.27]{LipSarKhov}, and is referred to as the {\it realisation} of $\Cat$.

\section{A flow category associated to a glued diagram}
\label{sec:assoc_flow_cat}
The framed flow category that we associate to a matched diagram will be obtained in a somewhat similar way as that associated to a diagram in \cite{LipSarKhov}.  The first difference is that we shall require a different flow category than the cube flow category downstairs. Again this category is essentially obtained from a Morse product construction.  In the case of an elementary tangle with two crossings the factor category requires three consecutively graded objects, with one morphism point between the first two, two morphism points between the last two, and the $1$-dimensional moduli space given by an interval. Morse-theoretically, this category can be obtained as in Figure \ref{fig:drawing_of_a_sock}. For more general diagram decompositions, the construction is a little more involved.
\begin{figure}[ht]
 \includegraphics[width=3cm,height=2.5cm]{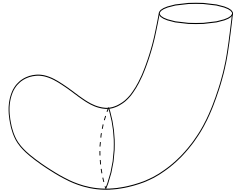}
 
 \caption{A Morse function on a sock.}
\label{fig:drawing_of_a_sock}
\end{figure}

\subsection{The sock flow category}
\label{subsec:the_n-sock}
\label{subsec:the_n_sock}

For any integer $n\geq 2$ and $N\geq 1$ one can define the Lens space $L^{2N+1}(n)=L^{2N+1}(n;1,\ldots,1)$ as the quotient space of $S^{2N+1}\subset \C^{N+1}$ by the action of $\Z/n$ given via
\[
 t\cdot (z_0,\ldots,z_N) = (e^{2\pi i/n}z_0,\ldots,e^{2\pi i/n}z_N)
\]
where $t$ represents a generator of $\Z/n$. It is well known that $H_j(L^{2N+1}(n);\Z/n)\cong \Z/n$ for all $j=0,\ldots,2N+1$ and $0$ otherwise. We want to define a Morse function $g\colon L^{2N+1}(n)$ and a Morse-Smale gradient for it which is perfect with respect to $\Z/n$ coefficients.

Let $f\colon \CP^N\to \R$ be the well-known Morse function given by
\[
 f([z_0:\cdots:z_N]) = \sum_{j=1}^N j|z_j|^2,
\]
where $\sum_{j=0}^N|z_j|^2=1$. Now let $F\colon S^{2N+1}\to \R$ be given by $F=f\circ p$, where $p\colon S^{2N+1}\to \CP^N$ is the quotient map
\[
 p(z_0,\ldots,z_N)=[z_0:\cdots :z_N].
\]
Then $F$ is a Morse-Bott function with $N+1$ critical manifolds given by
\[
 S_j=\{(0,\ldots,0,z_j,0,\ldots,0)\in S^{2N+1}\}
\]
each of which is a circle and whose index is given by $2j$, for $j=0,\ldots,N$.

For $n\geq 2$ let $g_n\colon S^1\to \R$ be given by
\[
 g_n(z)=z^n+\bar{z}^n,
\]
a $\Z/n$-invariant Morse function with $n$ maxima and $n$ minima, where $\Z/n$ acts on $S^1$ by
\[
 t\cdot z = z\cdot e^{2\pi i/n}.
\]
Using $g_n$, we can modify $F$ to a $\Z/n$-invariant Morse function $F_n$ on $S^{2N+1}$ which has exactly $2n(N+1)$ critical points, and which induces a $\Z/n$-perfect Morse function on $L^{2N+1}(n)$. In fact, it induces a $\Z/n$-perfect Morse function on any Lens space $L^{2N+1}(n;q_1,\ldots,q_N)$, but we will only need a particular one.

For $t\in \R$ denote by $\psi_t\colon S^1\to S^1$ the $\Z/n$-equivariant negative gradient flow of $g_n$. We can extend the flow $\psi_t$ to $\C$ radially. Then $\Phi\colon S^{2N+1}\times \R\to S^{2N+1}$ given by
\[
 ((z_0,\ldots,z_N),t) \mapsto \frac{(\psi_tz_0,e^{-t}\psi_tz_1,e^{-2t}\psi_tz_2,\ldots,e^{-Nt}\psi_tz_N)}{\|(\psi_tz_0,e^{-t}\psi_tz_1,e^{-2t}\psi_tz_2,\ldots,e^{-Nt}\psi_tz_N)\|}
\]
is a $\Z/n$-equivariant negative gradient flow for $F_n$.

In particular, this induces a negative gradient flow on $L^{2N+1}(n)$ to the induced Morse function $g\colon L^{2N+1}(n)\to \R$. It is easy to see that the critical points are given by
\begin{align*}
 p_{2j} &= [0,\ldots,0,e^{\pi i/n},0,\ldots,0] \\
 p_{2j+1} &= [0,\ldots,0,1,0,\ldots,0]
\end{align*}
for $j=0,\ldots,N$ with the non-zero entry in the $j$-th coordinate, and the stable and unstable manifolds intersect transversely. We can therefore form the Morse flow category.

It is clear that $\M(p_{2j+1},p_{2j})=\{P,M\}$ consists of two points, one will be framed $+$ and the other will be framed $-$.

\begin{lemma}\label{lem:lens_comp}
 For $0\leq j < k \leq N$ and $\varepsilon_1,\varepsilon_2\in \{0,1\}$ the open moduli spaces $\tilde{\M}(p_{2k+\varepsilon_1},p_{2j+\varepsilon_2})$ are a disjoint union of $n$ components, each of which is diffeomorphic to an open disc. Furthermore, the action of $\Z/n$ on $L^{2N+1}(n)$ given by
\begin{equation}\label{action:Z/n}
 t\cdot [z_0,z_1,\ldots,z_N] = [e^{2\pi i/n}z_0,z_1,\ldots,z_N]
\end{equation}
induces an action on moduli spaces which permutes the components.
\end{lemma}

\begin{proof}
 Points $x\in W^s(p_{2k+\varepsilon_1})\cap W^u(p_{2j+\varepsilon_2})$ are of the form
\[
 x=[0,\ldots,0,z_j,\ldots,z_k,0,\ldots,0]
\]
with $z_j\not=0\not=z_k$. Furthermore, we can assume that $\arg z_k\in (-\pi/n,\pi/n)$ for $\varepsilon_1=1$ and $\arg z_k=\pi/n$ for $\varepsilon_1=0$.

Also, if $\varepsilon_2=0$, then
\[
 \arg z_j \in (0,2\pi/n)\cup (2\pi/n,4\pi/n)\cup \cdots \cup ((2n-2)\pi/n,2\pi)
\]
and if $\varepsilon_2=1$, then $\arg z_j\in \{0,2\pi/n,\ldots,(2n-2)\pi/n\}$. Therefore there are at least $n$ components depending on the component that $z_j$ is in. To see that these are all the components, and each component is a disc, consider the point $$[0,\ldots,0,e^{(1-\varepsilon_2)\pi/n}/\sqrt{2},0,\ldots,0,e^{\varepsilon_1\pi/n}/\sqrt{2},0,\ldots,0].$$
This point can be considered as an origin, and all points $x\in W^s(p_{2k+\varepsilon_1})\cap W^u(p_{2j+\varepsilon_2})$ are in a straight line to this point or to the analogue point in a different component, and we can identify each component as a star-like open subset of Euclidean space, which shows that each component is diffeomorphic to an open disc.

The statement about the action on moduli spaces is clear, with the action being trivial for $j>0$.
\end{proof}

In particular, we get $\M(p_{2j},p_{2j-1})=\{P_1,\ldots,P_n\}$ for $j=1,\ldots,N$, where each $P_k$ corresponds to the trajectory $t\to [0,\ldots,0,1,e^{-t+(2k-1)\pi i/n},0,\ldots,0]$.

The $1$-dimensional moduli spaces $\M(p_{2j},p_{2j-2})$ are easily identified as $n$ intervals with boundaries pairing as $(P_k,P)$ with $(P_{k+1},M)$ for $k=1,\ldots,n-1$ and $(P_n,P)$ with $(p_1,M)$. Similarly, $\M(p_{2j+1},p_{2j-1})$ consists of $n$ intervals with boundaries pairing as $(P,P_k)$ with $(M,P_{k+1})$ for $k=1,\ldots,n-1$ and $(P,P_n)$ with $(M,P_1)$.

We will also need to know the $2$-dimensional moduli spaces $\M(p_{2j+1},p_{2j-2})$. As the boundary is determined by the lower dimensional moduli spaces between these objects, and we know the number of components by  Lemma \ref{lem:lens_comp}, we see that the moduli space are squares with corners given by $(P,P_k,P)$, $(P,P_{k+1},M)$, $(M,P_{k+2},M)$ and $(M,P_{k+1},P)$ for $k=1,\ldots,n$, where we identifiy $P_1=P_{n+1}$ and $P_2=P_{n+2}$. As we shall see later, we do not need to know other moduli spaces.

Given $r>0$, we can define a flow category $\cS^n_r$ whose objects are given by $0,\ldots,r$, and the moduli spaces are given by
\[
 \M(k,j)=\M(p_{k+1},p_{j+1})/(\Z/n).
\]
In particular, $\M(k,j)=\M(p_{k+1},p_{j+1})$ if $j>0$. Here $N$ is chosen so that $2N\geq r$ to ensure that all moduli spaces are defined.

For $r<0$ and $k>j\in \{r,r+1,\ldots,0\}$ we define
\[
 \M(k,j)=\M(p_{-j+1},p_{-k+1})/(\Z/n)
\]
to get the flow category $\cS^n_r$ which is dual to $\cS^n_{-r}$. Note that this is a quotient of the Morse flow category of the Morse function $-g$.

For $\mathbf{r}=(r_1,\ldots,r_m)\in (\Z-\{0\})^m$ let $G\colon (L^{2N+1}(n))^m\to \R$ be the Morse function given by
\[
 G(x_1,\ldots,x_m) = \sgn(r_1) g(x_1) + \cdots + \sgn (r_m) g(x_m)
\]
and $\Cat_G$ the corresponding Morse-flow category, where $2N\geq \max\{|r_1|,\ldots,|r_m|\}$. The group $(\Z/n)^m$ acts on $(L^{2N+1}(n))^m$ by letting each coordinate in $(\Z/n)^m$ act on the corresponding coordinate in $(L^{2N+1}(n))^m$ by (\ref{action:Z/n}). This action induces an action of $(\Z/n)^m$ on the moduli spaces of $\Cat_G$ and we define the flow category $\cS^n_\mathbf{r}$ as follows.

\begin{definition}
\label{def:catSn}
The objects of $\cS^n_\mathbf{r}$ are given by $\mathbf{a}=(a_1,\ldots,a_m)\in \Z^m$ with $0\leq a_j \leq r_j$ if $r_j>0$ or $r_j \leq a_j \leq 0$ if $r_j<0$, for all $j=1,\ldots,m$. The grading is given by $|\mathbf{a}|=a_1+\cdots+a_m$ and the moduli spaces are given by
\begin{multline*}
 \M((a_1,\ldots,a_m),(b_1,\ldots,b_m)) = \\ \M_G((p_{|a_1|+1},\ldots,p_{|a_m|+1}),(p_{|b_1|+1},\ldots,p_{|b_m|+1}))/(\Z/n)^m.
\end{multline*}
\end{definition}

Note that since we use the negative Morse function $-g$ for coordinates with $r_j<0$, we do not have to interchange the role of $a_j$ and $b_j$ in those coordinates.

\begin{proposition}
 Let $\mathbf{r}\in (\Z-\{0\})^m$ and $a,b$ objects in $\mathscr{S}^n_\mathbf{r}$. Then $\mathcal{M}(a,b)$ is a disjoint union of discs.
\end{proposition}

\begin{proof}
 From the construction of the flow category $\mathscr{S}^n_\mathbf{r}$ it is enough to show that the same holds for the flow category of the Morse function $G$ defined on the product of lens spaces. To see that this holds we use induction on $m$. For $m=1$ it follows from the description of the Morse flow category given above, where all moduli spaces are disjoint unions of cubes. The induction step follows from Lemma \ref{lem_prodflowcat} below.
\end{proof}

\begin{lemma}
 \label{lem_prodflowcat}
Let $f\colon M\to \R$ and $g\colon N \to \R$ be Morse functions on closed, smooth manifolds $M$ and $N$, and let $v_M$, $v_N$ be Morse-Smale gradients for the respective Morse function. Let $\Cat_F$ be the Morse flow category of the Morse function $F\colon M\times N\to \R$ given by $F(x,y)=f(x)+g(y)$ with Morse-Smale gradient $v(x,y)=(v_M(x),v_N(y))$. If $a$, $b$ are critical points of $f$ with $\mathrm{ind}(a)\geq \mathrm{ind}(b)$ and $p$, $q$ are critical points of $g$ with $\mathrm{ind}(p)\geq \mathrm{ind}(q)$, then $\mathcal{M}_{\Cat_F}((a,p),(b,q))$ is PL-homeomorphic to
\begin{itemize}
\item $\mathcal{M}_{\Cat_f}(a,b)\times \mathcal{M}_{\Cat_g}(p,q)\times [0,1]$ if $\mathrm{ind}(a)> \mathrm{ind}(b)$ and $\mathrm{ind}(p)> \mathrm{ind}(q)$.
\item $\mathcal{M}_{\Cat_f}(a,b)$ if $\mathrm{ind}(a)> \mathrm{ind}(b)$ and $p=q$.
\item $\mathcal{M}_{\Cat_g}(p,q)$ if $a=b$ and $\mathrm{ind}(p)> \mathrm{ind}(q)$.
\end{itemize}
\end{lemma}

Note that for $\mathrm{ind}(a)> \mathrm{ind}(b)$ and $\mathrm{ind}(p)> \mathrm{ind}(q)$ we usually do not get a diffeomorphism between $\mathcal{M}_{\Cat_F}((a,p),(b,q))$ and $\mathcal{M}_{\Cat_f}(a,b)\times \mathcal{M}_{\Cat_g}(p,q)\times [0,1]$, as there will usually be more corners in $\mathcal{M}_{\Cat_F}((a,p),(b,q))$.

\begin{proof}
The cases where $a=b$ or $p=q$ are easy to see, so we will focus on the case where $\mathrm{ind}(a)> \mathrm{ind}(b)$ and $\mathrm{ind}(p)> \mathrm{ind}(q)$. The proof is by induction on $\mathrm{ind}(a,p)-\mathrm{ind}(b,q) \geq 2$ with the root case being trivial.

To simplify our notations we will drop the $\mathscr{C}$ from the moduli spaces. We will write $ap$ for the critical point $(a,p)$ of $F$ and similarly with other combinations of critical points of $f$ and $g$.

Let $c$ be a critical point of $f$ with $\mathrm{ind}(a)>\mathrm{ind}(c)>\mathrm{ind}(b)$ and $r$ a critical point of $g$ with $\mathrm{ind}(p)>\mathrm{ind}(r)>\mathrm{ind}(q)$. Since $\mathscr{C}_F$ is a flow category, the boundary of $\mathcal{M}(ap,bq)$ is
\begin{equation*}
\begin{split}
 \partial \mathcal{M}(ap,bq)=& \mathcal{M}(bp,bq)\times \mathcal{M}(ap,bp) \cup \mathcal{M}(aq,bq)\times \mathcal{M}(ap,aq) \cup \\
& \bigcup_{(c,r)}\mathcal{M}(cp,bq)\times \mathcal{M}(ap,cp) \cup \mathcal{M}(cq,bq)\times \mathcal{M}(ap,cq) \cup \\
 & \phantom{(c,r)} {} \mathcal{M}(br,bq)\times \mathcal{M}(ap,br) \cup \mathcal{M}(ar,bq)\times \mathcal{M}(ap,ar) \cup \\
 & \phantom{(c,r)} {} \mathcal{M}(cr,bq) \times \mathcal{M}(ap,cr).
\end{split}
\end{equation*}
If $\mathrm{ind}(a)-\mathrm{ind}(b)=1$ or $\mathrm{ind}(p)-\mathrm{ind}(q)=1$, this picture simplifies, and the following arguments simplify also.

Note that $\mathcal{M}(bp,bq)\times \mathcal{M}(ap,bp) \cong \mathcal{M}(p,q)\times \mathcal{M}(a,b)$, and the same holds for $\mathcal{M}(aq,bq)\times \mathcal{M}(ap,aq)$. We want to show that $\mathcal{M}(ap,bq)$ is a cylinder between these two boundary parts. To distinguish these parts more easily, we write
$$\mathcal{M}(a,b)\boxtimes \mathcal{M}(p,q) \subset \mathcal{M}(ap,bq)$$
for $\mathcal{M}(bp,bq)\times \mathcal{M}(ap,bp)$ to indicate that these are the broken flow lines that first go from $ap$ to $bp$, and then from $bp$ to $bq$. Hence we also write $\mathcal{M}(aq,bq)\times \mathcal{M}(ap,aq)= \mathcal{M}(p,q)\boxtimes \mathcal{M}(a,b)$.

By induction hypothesis, we have 
\[
\mathcal{M}(cp,bq) \times \mathcal{M}(ap,cp) \cong (\mathcal{M}(c,b)\times \mathcal{M}(p,q) \times [0,1])\times \mathcal{M}(a,c)
\]
and
\[
 \mathcal{M}(cq,bq)\times \mathcal{M}(ap,cq) \cong \mathcal{M}(c,b) \times (\mathcal{M}(a,c)\times \mathcal{M}(p,q) \times [0,1])
\]
and we can think of these two boundary parts as combining to a cylinder $C_1$ between
$$\mathcal{M}(a,c)\boxtimes \mathcal{M}(c,b) \boxtimes \mathcal{M}(p,q) \subset \mathcal{M}(a,b)\boxtimes \mathcal{M}(p,q)$$
and $$\mathcal{M}(p,q)\boxtimes \mathcal{M}(a,c)\boxtimes \mathcal{M}(c,b) \subset \mathcal{M}(p,q)\boxtimes \mathcal{M}(a,b)$$
with $\mathcal{M}(a,c)\boxtimes \mathcal{M}(p,q) \boxtimes \mathcal{M}(c,b)$ in the middle.

Similarly, $\mathcal{M}(br,bq)\times \mathcal{M}(ap,br) \cup \mathcal{M}(ar,bq)\times \mathcal{M}(ap,ar)$ is a cylinder $C_2$ between
$$\mathcal{M}(a,b)\boxtimes \mathcal{M}(p,r) \boxtimes \mathcal{M}(r,q) \subset \mathcal{M}(a,b)\boxtimes \mathcal{M}(p,q)$$
and $$\mathcal{M}(p,r)\boxtimes \mathcal{M}(r,q)\boxtimes \mathcal{M}(a,b) \subset \mathcal{M}(p,q)\boxtimes \mathcal{M}(a,b)$$
with $\mathcal{M}(p,r)\boxtimes \mathcal{M}(a,b) \boxtimes \mathcal{M}(r,q)$ in the middle.

The two cylinders $C_1$ and $C_2$ intersect in two disjoint cylinders
\begin{multline*}
\mathcal{M}(r,q)\times(\mathcal{M}(c,b)\times \mathcal{M}(p,r)\times [0,1])\times \mathcal{M}(a,c)\,\sqcup \\ \mathcal{M}(c,b)\times (\mathcal{M}(a,c)\times \mathcal{M}(r,q)\times [0,1])\times \mathcal{M}(p,r).
\end{multline*}

Finally,
\begin{align*}
 \mathcal{M}(cr,bq) \times \mathcal{M}(ap,cr) & \cong (\mathcal{M}(c,b)\times \mathcal{M}(r,q)\times [0,1])\times \\
& (\mathcal{M}(a,c)\times \mathcal{M}(p,r)\times [0,1]),
\end{align*}
which we can think of as a square between
$$\mathcal{M}(a,c) \boxtimes \mathcal{M}(p,r) \boxtimes \mathcal{M}(c,b) \boxtimes \mathcal{M}(r,q),
\mathcal{M}(a,c) \boxtimes \mathcal{M}(p,r) \boxtimes \mathcal{M}(r,q) \boxtimes \mathcal{M}(c,b),$$
$$\mathcal{M}(p,r) \boxtimes \mathcal{M}(a,c) \boxtimes \mathcal{M}(c,b) \boxtimes \mathcal{M}(r,q),\mathcal{M}(p,r) \boxtimes \mathcal{M}(a,c) \boxtimes \mathcal{M}(r,q) \boxtimes \mathcal{M}(c,b),$$
which fits exactly between the two cylinders $C_1$ and $C_2$.
Because of the collar neighborhoods the boundary parts
\begin{multline*}
 \mathcal{M}(cp,bq)\times \mathcal{M}(ap,cp) \cup \mathcal{M}(cq,bq)\times \mathcal{M}(ap,cq) \cup \mathcal{M}(cr,bq) \times \mathcal{M}(ap,cr) \,\cup\\
\mathcal{M}(br,bq)\times \mathcal{M}(ap,br) \cup \mathcal{M}(ar,bq)\times \mathcal{M}(ap,ar)
\end{multline*}
combinatorally combine to a cylinder between
$$\mathcal{M}(a,c)\boxtimes \mathcal{M}(c,b)\boxtimes \mathcal{M}(p,q) \cup \mathcal{M}(a,b)\boxtimes \mathcal{M}(p,r)\boxtimes \mathcal{M}(r,q)$$
and $$\mathcal{M}(p,q)\boxtimes \mathcal{M}(a,c)\boxtimes \mathcal{M}(c,b) \cup \mathcal{M}(p,r)\boxtimes \mathcal{M}(r,q)\boxtimes \mathcal{M}(a,b).$$
Repeating this for every pair $(c,r)$ gives a combinatorial cylinder between $$\partial (\mathcal{M}(a,b)\boxtimes \mathcal{M}(p,q)) \mbox{ and } \partial(\mathcal{M}(p,q) \boxtimes \mathcal{M}(a,b)).$$
Using the interior of $\mathcal{M}(ap,bq)$ we can fill this cylinder of the boundary to a cylinder of $\mathcal{M}(a,b)\times \mathcal{M}(p,q)$. This involves the gluing maps which come from the existence of collar neighborhoods in flow categories \cite{Laures}, see also \cite{AusBra}.
\end{proof}

\subsection{A cover of the sock flow category}
\label{subsec:matched_category}
We begin with a few definitions whose parallels in the construction of the Khovanov flow category will be apparent to the \emph{cognoscenti}.  In what follows the number $n \geq 2$ is fixed and relates to the $\SL_n$ polynomial specialization of the HOMFLYPT polynomial (in particular $n=2$ refers Khovanov cohomology).

\begin{definition}
	\label{def:weight_match_config}
A \emph{weighted glued configuration} is an embedding of a finite number of circles in $\R^2$ together with a finite collection of arcs, disjoint from the circles and each other, except that the arcs meet the circles transversely.  Furthermore we make a choice of the endpoints of $A$, the L-endpoint and the R-endpoint  $\{ p_L(A), p_R(A) \} = \partial A$; and we require that each arc carries a weighting by a pair of integers $w(A) = (r,s)$ such that either $0 < s \leq r$ or $r \leq s < 0$.
\end{definition}

We note here that the choice of left and right endpoints will be shown to have no effect on the stable homotopy type eventually produced - see Proposition \ref{prop:left_and_right_invariance}.

\begin{definition}
	\label{def:label_weight_match_config}
A \emph{labelled weighted glued configuration} is a weighted resolution configuration in which each circle of the resolution is labelled with labels drawn from the set $\{ 1, x, x^2, \ldots x^{n-1} \}$.  We write this as a pair $(D,y)$ where $D$ is the weighted glued configuration and $y$ is the labeling.
\end{definition}

\begin{definition}
\label{def:partorder}
We define a partial order $<$ on labelled weighted glued configurations by first defining the primitive relation $<_1$.  We define $(E,z) <_1 (D,y)$ if

\begin{enumerate}
\item $(E,z)$ and $(D,y)$ are the same as embeddings of circles and arcs. In this case we require that the weightings of the arcs all agree except at a single arc $A$ where the weightings are $w(A) = (r,s)$ for $E$ and $w(A) = (r,s+1)$ for $D$.  We require that the labellings $y$ and $z$ agree everywhere apart from the circle or pair of circles which contain the endpoints of $A$.

\begin{enumerate}
\item If the arc $A$ has endpoints which lie on the same circle $C$ and $s$ is both positive and odd or both negative and even, then we require that $y(C) = x z(C)$

\item If the arc $A$ has endpoints which lie on the same circle $C$ and $s$ is both positive and even or both negative and odd, then we require that $z(C) = 1$ and $y(C) = x^{n-1}$.

\item If the arc $A$ has endpoints lying on different circles $p_L(A) \in C_L$ and $p_R(A) \in C_R$ and $s$ is both positive and odd or both negative and even, then we require that either $y(C_L) = z(C_L)$ and $y(C_R) = xz(C_R)$ or $y(C_L) = xz(C_L)$ and $y(C_R) = z(C_R)$.

\item If the arc $A$ has endpoints lying on different circles $p_L(A) \in C_L$ and $p_R(A) \in C_R$ and $s$ is both positive and even or both negative and odd then we require that $y(C_L)y(C_R) = x^{n-1}z(C_L)z(C_R)$.
\end{enumerate}

\item $(D,y)$ as a diagram forgetting weights and labels is the result of doing surgery along an arc $A$ of $(E,z)$ and then deleting the arc.  The weights on the common arcs of $(D,y)$ and $(E,z)$ agree, and the labels on the common circles also agree.  We require that the weighting $w(A) = (r,s)$ satisfies $s = -1$.

Alternatively, $(E,z)$ as a diagram forgetting weights and labels is the result of doing surgery along an arc $A$ of $(D,y)$ and then deleting the arc.  The weights on the common arcs of $(E,z)$ and $(D,y)$ agree, and the labels on the common circles also agree.  We require that the weighting $w(A) = (r,s)$ satisfies $s=1$.

In either case we require that the labellings $y$ and $z$ agree everywhere apart from on the three circles involved in the surgery.  On the three circles involved in the surgery we require the following:

\begin{enumerate}
\item If $D$ has one more circle than $E$, then we write the $C$ for the surgery-involved circle of $E$ and $C_1, C_2$ for the surgery-involved circles of $D$.  We require that $y(C_1)y(C_2) = x^{n-1}z(C)$.

\item If $D$ has one fewer circle than $E$, then we write the $C$ for the surgery-involved circle of $D$ and $C_1, C_2$ for the surgery-involved circles of $E$.  We require that $y(C) = z(C_1)z(C_2)$.
\end{enumerate}
\end{enumerate}

We extend this primitive relation $<_1$ to the smallest possible partial ordering $<$.
\end{definition}

For notation we shall write $(E,z) <_i (D,y)$ if $(E,z) < (D,y)$ and there is a chain of $i$ primitive relations connecting $(E,z)$ to $(D,y)$.

We next describe how a glued link diagram $D_\br$ (see Definition \ref{def:glued_d}) gives rise to an associated flow category  .Eventually, we shall wish to distinguish the cases $n=2$ and $n \geq 2$.  In the latter case we shall be interested mainly in the situation in which $D_\br$ is a matched diagram (see Definition \ref{defn:matched_diagram}), but for now we lose nothing by proceeding with no assumptions.

\begin{figure}
\centerline{
{
\psfrag{ldots}{\ldots}
\psfrag{V}{Vertical}
\psfrag{H}{Horizontal}
\includegraphics[height=1in,width=2.5in]{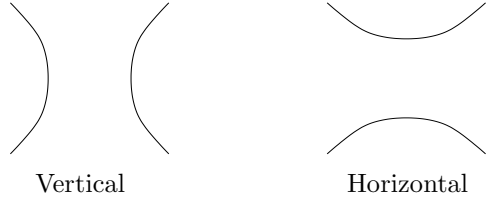}
}}
\caption{We show what we mean by the vertical and horizontal smoothings of an elementary tangle (assumed to be $|r|$ horizontal half-twists as in Figure \ref{2twist}).}
\label{smoothings}
\end{figure}

Suppose $D_\br$ consists of $m$ elementary tangles of indices $\br = (r_1,r_2,\ldots, r_m)$.  We shall construct a flow category $\cL^n(D_\br)$ as a \emph{disc cover} of $\cS^n_\br$.

\begin{definition}
	\label{defn:morphism_flow_cat}
	A \emph{morphism} of flow categories $\Psi : \Cat_1 \rightarrow \Cat_2$ is a grading-preserving functor $\Psi$ such that if $x,y \in \Ob(\Cat_1)$ then
	\[ \Psi : \Hom(x,y) \rightarrow \Hom(\Psi(x), \Psi(y)) \]
	\noindent is a local diffeomorphism.
\end{definition}

\begin{definition}
	\label{defn:disc_cover}
	If $\Psi : \Cat_1 \rightarrow \Cat_2$ is a morphism of flow categories and each morphism space of $\Cat_2$ is homeomorphic to a disc, then we call $\Cat_1$ an \emph{disc cover} of $\Cat_2$.
\end{definition}

The useful point about disc covers is the following proposition.

\begin{proposition}
	\label{prop:disc_cover}
	Suppose that $\widetilde{\Ob_1}$ is a finite $\Z$-graded set, and suppose that $\widetilde{\Cat_1}_i$ is a flow category whose graded object set consists of those elements of grading $i$, $i+1$, and $i+2$ of $\widetilde{\Ob_1}$.  Let $\widetilde{\Cat_1}$ be the category with object set $\widetilde{\Ob_1}$ with the smallest morphism sets such that each $\widetilde{\Cat_1}_i$ is a full subcategory.
	Further suppose that $\Cat_2$ is a flow category in which each morphism space is a disc.  Finally suppose that $\widetilde{\Psi} : \widetilde{\Cat_1} \rightarrow \Cat_2$ is a grading-preserving functor such that its restriction to each $\widetilde{\Cat_1}_i$ is a morphism of flow categories.
	
	Now for $a,b \in \widetilde{\Ob_1}$ with $|b| - |a| = 3$
	\[ \Hom_{\widetilde{\Cat_1}}(b,a) = \bigcup_{c \, \, : \, \, |a| < |c| < |b|} \Hom_{\widetilde{\Cat_1}}(c,a) \times \Hom_{\widetilde{\Cat_1}}(b,c) \]
	\noindent is topologically a disjoint union of circles.
	
	The functor $\widetilde{\Psi}$ induces maps
	\[\widetilde{\Psi}_{b,a} : \Hom_{\widetilde{\Cat_1}}(b,a) \rightarrow \Hom_{\Cat_2}(\widetilde{\Psi}(b),\widetilde{\Psi}(a)) \]
	\noindent which are covering maps.  If these maps are trivial covering maps on each component of each $\Hom_{\widetilde{\Cat_1}}(b,a)$ then there is a unique flow category $\Cat_1$ and disc cover $\Psi : \Cat_1 \rightarrow \Cat_2$ such that
	
	\begin{itemize}
	\item $\Cat_1$ has object set $\widetilde{\Ob_1}$,
	\item Restricting $\Cat_1$ to objects of grading $i$, $i+1$, and $i+2$ gives the flow category $\widetilde{\Cat_1}_i$,
	\item $\widetilde{\Cat_1}$ is a subcategory of $\Cat_1$,
	\item $\widetilde{\Psi}$ is the restriction of the functor $\Psi$.
	\end{itemize}
\end{proposition}

\begin{proof}
	We determine $\Cat_1$ and $\Psi$ inductively.  Firstly note that by hypothesis $\Cat_1$ and $\Psi$ are determined on the level of objects and on the level of the $0$- and $1$-dimensional moduli spaces.
	
	Now taking $a,b \in \Ob(\Cat_1)$ with $|b| - |a| = 3$ we note that $\partial \M(b,a)$, if $\M(b,a)$ exists, is already determined and is a disjoint union of circles.  We know that $\M(b,a)$ must be obtained by filling each circle with a disc since $\Psi$ is a disc cover.  Furthermore, doing this we can find a $\Psi$ defined on each $2$-dimensional moduli space because of the trivial covering map hypothesis on $\widetilde{\Psi}_{b,a}$.
	
	Now taking $a,b \in \Ob(\Cat_1)$ with $|b| - |a| = 4$ we note that the map
	\[ \Psi_{b,a} |_{\partial \M (b,a)} : \partial \M (b,a) \rightarrow \partial \M(\Psi(b), \Psi(a)) {\rm ,}\]
	\noindent (if we can define $\M_{b,a}$ and $\Psi_{b,a}$ consistently) is determined and is a local covering map of a 2-sphere, and hence a trivial covering.  Hence we must obtain $\M (b,a)$ by filling in each 2-sphere boundary with a 3-ball.  The induction can then continue, increasing the relative index of $b,a$, until $\M (b,a)$ is determined for each $a,b \in \Ob(\Cat_1)$.
\end{proof}

Proposition \ref{prop:disc_cover} provides a quick way to construct the flow category $\cL^n(D_{\br})$.  We shall describe the objects of $\cL^n(D_{\br})$, along with the moduli spaces of dimensions less than $2$.  At the same time, we shall construct a disc cover $\Phi$ from $\cL^n(D_{\br})$ to $\cS_{\br}$ by giving it on the moduli spaces of dimensions less than $2$, verifying that the $1$-dimensional trivial cover condition is satisfied, and then appealing to Proposition \ref{prop:disc_cover} to give the rest.

\begin{definition}
\label{def:objects_category}
We define the objects of the category $\cL^n(D_\br)$ to be all possible labellings of a set of weighted glued configurations that occur as \emph{resolutions} of $D_\br$.

To construct such a resolution, for each of the $m$ tangle summands either choose the horizontal or the vertical smoothing.  If the vertical smoothing is chosen at crossing $i$ then add an arc connecting the two strands of that smoothing and decorate the arc by $(r_i,s_i)$ for some admissible choice of $s_i$.  This gives a resolution of $D_\br$, a weighted glued configuration $C$.

The objects of the category $\cL^n(D_\br)$ are exactly all possible labellings on all possible resolutions $C$ constructed in this way.
\end{definition}

The local picture for the construction of a resolution is illustrated in Figure \ref{bipeg}.

\begin{figure}
\centerline{
{
\psfrag{ldots}{\ldots}
\psfrag{V}{Vertical}
\psfrag{H}{Horizontal}
\psfrag{(4,s)}{$(4,s)$}
\psfrag{L}{$L$}
\psfrag{R}{$R$}
\includegraphics[height=1.5in,width=3.5in]{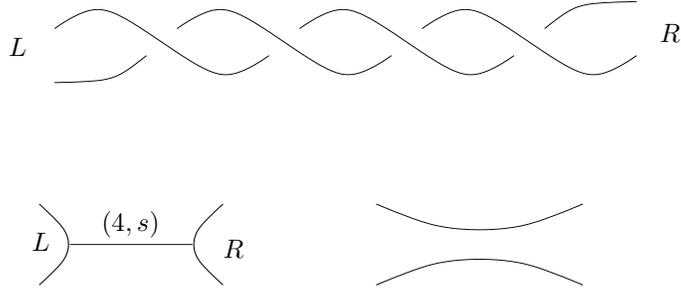}
}}
\caption{We show how an elementary tangle of index $4$ (pictured in the top half of the figure) turns into five possible local pictures of a weighted glued configuration (pictured in the bottom half of the figure).  Here $s$ is either 1, 2, 3, or 4.}
\label{bipeg}
\end{figure}

Now the objects of the flow category $\cS^n_\br$ are naturally described as $m$-tuples $\bs = (s_1, s_2, \ldots, s_m)$ where $| s_i | \leq |r_i|$ and $s_i$ is either $0$ or agrees in sign with $r_i$.  Recall that we wish to define the flow category $\cL^n(D_\br)$ via a locally covering morphism of flow categories $\Phi : \cL^n(D_\br) \rightarrow \cS^n_\br$.  First we describe what this morphism does on the objects of $\cL^n(D_\br)$.

\begin{definition}
Let $C$ and $s_i$ be as in Definition \ref{def:objects_category}.  Then the map
\[ C \mapsto (s_1, s_2, \ldots, s_m) \]
\noindent where we take $s_i = 0$ if the horizontal smoothing has been chosen at tangle $i$, when composed with the map that forgets labelings, gives a map from $\Ob(\cL(D_\br))$ to the objects of $\Ob(\cS_\br)$.
\end{definition}

Now we wish to describe the morphisms of $\cL^n(D_\br)$.  There is a (non-identity) morphism from object $(E,y)$ to object $(D,x)$ iff we have $(E,y) < (D,x)$.  The cohomological grading of the objects of $\cL^n(D_\br)$ is inherited through the map $\Phi$ from the cohomological grading of the objects of $\cS^n_\br$.  We begin by describing the $0$-dimensional morphisms.

\begin{definition}
\label{defn:numberofpoints}
Suppose that $(E,y) <_1 (D,x)$, then we have that the moduli space $\M ((D,x),(E,y))$ has the structure of a compact $0$-manifold.  The number of points of $\M ((D,x),(E,y))$ is:

\begin{enumerate}
\item 2 if we are in case 1a of Definition \ref{def:partorder},
\item n if we are in case 1b,
\item 1 if we are in any of the remaining cases.
\end{enumerate}
\end{definition}

Now we give the morphism of flow categories $\Phi$ on the $0$-dimensional moduli spaces already defined, and for this we need some notation.  Suppose then that $p = (i_1, \ldots, i_\alpha, \ldots, i_m)$ and $q = (i_1, \ldots, i_\alpha + 1, \ldots, i_m)$ are two objects of the flow category $\cS_{\br}$.  In the case that $0 \in \{ i_{\alpha}, i_{\alpha} + 1 \}$, we write the point in the moduli space $\M_{\cS_\br} (q,p)$ as $P$.
In the case that $i_{\alpha}$ is both positive and odd or both negative and even, the moduli space $\M_{\cS_\br} (q,p)$ consists of two points $P$ and $M$.  In the remaining case $\M_{\cS_\br} (q,p)$ consists of $n$ points $P_1, \ldots , P_n$.

To give completely the morphism $\Phi$ on the $0$-dimensional moduli spaces, we need to say which points of $\M ((D,x),(E,y))$ get sent to $P$, to $M$, or to $P_i$ for $1 \leq i \leq n$.

\begin{definition}
	\label{defn:Phi_0dim}
	We refer to the list of Definitions \ref{defn:numberofpoints} and \ref{def:partorder}, and describe where the points of $\M ((D,x),(E,y))$ get sent under $\Phi$.
	
\begin{itemize}
	\item (1a) The points map surjectively onto $\{P,M\}$.
	\item (1b) The points map surjectively onto $\{P_1, P_2, \ldots, P_n\}$.
	\item (1c) In the first case the point is sent to $P$ and in the second it is sent to $M$.
	\item (1d) If $y(C_L) = x^k z(C_L)$ then the point is sent to $P_{k+1}$.
	\item (2a), (2b) The point is sent to $P$.
\end{itemize}
\end{definition}

Next we wish to describe both the 1-dimensional morphism spaces of $\cL(D_\br)$, and how the morphism of flow categories $\Phi$ behaves on these morphism spaces of $\cL(D_\br)$.

\begin{definition}
Suppose now that $(D,y) >_2 (E,z)$.

There are essentially two cases.  One case is when $(D,y)$ and $(E,z)$ are related by a double surgery in a `ladybug' formation: in this case there is a choice to be made about the moduli spaces $\M ((D,y),(E,z))$.  The second case is the non-ladybug case, and here there is a unique choice of moduli spaces $\M ((D,y),(E,z))$ and morphism $\Phi |_{\M ((D,y),(E,z))}$ consistent with the existence of a disc-cover $\Phi$ with the prescribed behavior already given on the $0$-dimensional moduli spaces.  We leave the verification of this latter case to the reader.

Now suppose that we are in the ladybug case.  This means that we have $(D,y) >_1 (F,w) >_1 (E,z)$ where one performs a surgery to get from $E$ to $F$ and then another surgery to arrive at $D$, and the associated handlebody to the double surgery is a disjoint union of a torus with two boundary components and some cylinders embedded in $[0,1] \times \R^2$ with boundary $D \subset \{ 0 \} \times \R^2$ and $E \subset \{ 1 \} \times \R^2$.

\begin{figure}
\centerline{
{
\psfrag{ldots}{\ldots}
\psfrag{V}{Vertical}
\psfrag{H}{Horizontal}
\psfrag{(4,s)}{$(4,s)$}
\psfrag{L}{$L$}
\psfrag{R}{$R$}
\includegraphics[height=1.5in,width=1.2in]{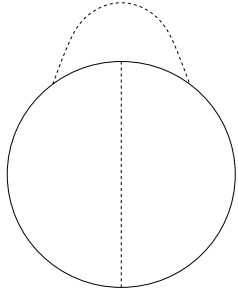}
}}
\caption{A ladybug configuration in a standard position in the plane, showing the surgery arcs as dashed lines.  The internal surgery arc is vertical, and the external surgery arc lies above the centre of the circle.}
\label{fig:ladybug}
\end{figure}

By an isotopy of the plane we can ensure that the circle $c_E$ of $E$ undergoing two surgeries lies in the plane as shown in Figure \ref{fig:ladybug}.  We write the surgered circle as $c_D$.

We have that $z(c_E) = 1$ and $y(c_D) = x^{n-1}$.  If $(F,w)$ is such that $(D,y) >_1 (F,w) >_1 (E,z)$ then $F$ is the result of surgery along exactly one of the arcs of Figure \ref{fig:ladybug}, and the case of either surgery there are $n$ possible labellings $w$ of $F$ (one of the new circles of $F$ gets labelled with $x^k$ and the other with $x^{n-1-k}$ for each $k = 0,1,\ldots,n-1$).

Now we wish to define the moduli space $\M((D,y),(E,z))$ such that there is a an extension of the morphism $\Phi$ already defined on the $0$-dimensional moduli spaces of $\cL^n(D_\br)$ to $\M((D,y),(E,z))$.  By counting the number of points of the boundary of $\M((D,y),(E,z))$, we see that $\M((D,y),(E,z))$ has to be the disjoint union of $n$ closed intervals.
It remains to determine which pairs of points in the boundary are joined by an interval.  This amounts to giving a bijection between the $n$ labellings $w_i$ following internal surgery (ie surgery along the internal arc) and the $n$ labellings $w_e$ following external surgery.
Lipshitz and Sarkar have given a way to identify the two circles $c_{i,1}$ and $c_{i,2}$ obtained from $c_E$ by internal surgery with those circles $c_{e,1}$ and $c_{e,2}$ obtained by external surgery, and this is called the \emph{ladybug matching} $L: \{ c_{i,1}, c_{i,2} \} \rightarrow \{ c_{e,1}, c_{e,2} \}$.  We give a bijection of labellings via the ladybug matching by 
requiring $w_e(L(c_{i,1})) = w_i(c_{i,1})$ and $w_e(L(c_{i,2})) = w_i(c_{i,2})$.

This choice of bijection for each ladybug configuration then determines both the $1$-dimensional moduli spaces of $\cL^n(D_\br)$ and how they cover the 1-dimensional moduli spaces of $\cS^n_\br$.

\end{definition}

\begin{figure}
\centerline{
{
\psfrag{ldots}{\ldots}
\psfrag{V}{Vertical}
\psfrag{H}{Horizontal}
\psfrag{1}{$1$}
\psfrag{2}{$2$}
\psfrag{3}{$3$}
\includegraphics[height=1.4in,width=2.8in]{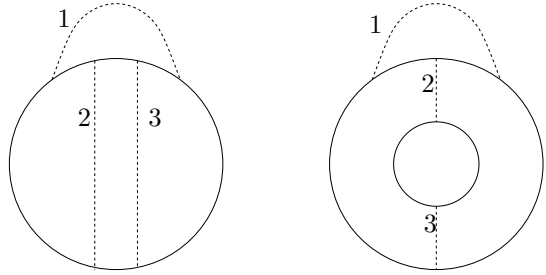}
}}
\caption{The `double ladybug' configurations of three dotted surgery arcs giving rise to 2-dimensional moduli space components.}
\label{fig:double_ladybird}
\end{figure}

It still remains to check that this choice of 1-dimensional covering is consistent with the circle covering condition of Proposition \ref{prop:disc_cover}.  The critical case in which the ladybug choice plays a role turns out to be that of the `double ladybug' as illustrated in Figure \ref{fig:double_ladybird}.  Indeed, in this case the boundary of the corresponding 2-dimensional moduli space can be verified to be the disjoint union of $n$ hexagons.

We check that each double ladybug of Figure \ref{fig:double_ladybird} is consistent with the circle covering condition of Proposition \ref{prop:disc_cover}.

The disk in the sock flow category that we must consider has the cornered structure of a solid hexagon.  The six vertices correspond to orderings of the surgery arcs (shown as numbers $1$, $2$, $3$ in Figure \ref{fig:double_ladybird}), and the six edges correspond to pairs of orderings that differ by a transposition of the first two or last two arcs.

After having performed one or more of the surgeries, we shall also order the resulting circles (for notation), this time in increasing order of the $x$-coordinates of their leftmost points.

To specify a lift to $\cL^n$ of a corner of the disk in the sock flow category we need to give a sequence of four objects in $\cL^n$, the $i$th object being a decoration of the circles resulting from the first $i-1$ surgeries.  For grading reasons, the circles before any surgery must be decorated with $1$, and after all three surgeries they must be decorated with $x^{n-1}$.  This means that it is enough to give the decoration after the first surgery, and the decoration after the first two surgeries.  We shall therefore write a corner as
\[ [(A,B,C), v_1, v_2] \]
\noindent where $\{A,B,C\} = \{1,2,3\}$ so that the first vector gives the ordering of the surgeries, and $v_i$ gives the decorations after $i$ surgeries for $i=1,2$.

We now verify that first double ladybug of Figure \ref{fig:double_ladybird} satisfies the circle covering condition.  Indeed, for each $k = 0,1,\ldots,n-1$ there is a trivial cover of the boundary hexagon in the sock flow category given by the following cyclic ordering of vertices:
\begin{gather*}
[(123),(x^k,x^{n-k-1}),(x^{n-1})], [(213),(x^k,x^{n-k-1}),(x^{n-1})], \\
{[(231),(x^k,x^{n-k-1}),(x^k,x^{n-1},x^{n-k-1})]}, [(321),(x^k,x^{n-k-1}),(x^k,x^{n-1},x^{n-k-1})], \\
{[(312),(x^k,x^{n-k-1}),(x^{n-1})]}, [(132),(x^k,x^{n-k-1}),(x^{n-1})] {\rm .}
\end{gather*}

Similarly, the second double ladybug gives
\begin{gather*}
[(123), (x^k,x^{n-k-1},1), (x^k, x^{n-k-1})], [(213),(1),(x^k, x^{n-k-1})], \\
{[(231),(1),(x^k, x^{n-k-1})]}, [(321),(1), (x^k,x^{n-k-1})], \\
{[(312), (1), (x^k, x^{n-k-1})]}, [(132),(x^k,x^{n-k-1}),(x^k,x^{n-k-1})] {\rm .}
\end{gather*}

\emph{A priori}, one needs to check that the circle covering condition is satisfied for all 3-dimensional `resolution configurations'.  This can be done, but we suppress the other cases (which are much simpler) in the interests of length.

\begin{remark}
\label{rem:lens}
We note that $\cL^n(D_\br)$ can also be constructed as a cover of a framed flow category which itself has no dependence on $n$ (essentially one could take the underlying flow category to be the `$n=1$' Lens space).  This approach makes the obstruction theory easier and also serves as a further indication that the stable homotopy type that we produce for $n>2$ is the `correct' object since its construction is seen to be even closer to the case $n=2$ which returns the Lipshitz-Sarkar homotopy type.  The downside is that there are then many more non-trivial ladybug configurations that need to be checked before one can be certain that one has defined a framed flow category.
\end{remark}

\section{The stable homotopy type of a diagram}
\label{sec:sln_match_diag}

\subsection{Obstruction theory and framings}
\label{subsec:obstruction_n_theory}
\label{subsec:obstruction_theory}

Recall that the glued flow category $\cL^n(D_\br)$ arises as a disc cover of the sock flow category $\cS^n_\br$, and so a framing of $\cS^n_\br$ induces a framing of $\cL^n(D_\br)$.

In this subsection we consider the problem of framing the $\cS^n_\br$, and whether, given a glued diagram $D_\br$, different framings may give rise to glued flow categories producing inequivalent stable homotopy types.  In fact, not every moduli space of $\cS^n_\br$ is in the image of the covering map, and so we shall be able to make do with a \emph{partial} framing of $\cS^n_\br$ in order to induce a complete framing on $\cL^n(D_\br)$.

Rather than work with $\cS^n_\br$, we work at first with the more general notion of a flow category all of whose morphism spaces are homeomorphic to disjoint unions of discs.  We call such a flow category a \emph{disc} flow category.  Our work in this subsection generalizes work by Lipshitz-Sarkar in Subsection 4.2 of \cite{LipSarKhov} carried out there for the cube flow category.  The reader with that to hand is at an advantage.

Let $\cD$ be a disc flow category.  From $\cD$ we shall construct a CW-complex (not the complex $|\cD|$, to define which we would require a framed embedding of $\cD$) written $\cw(\cD)$.  This complex $\cw(\cD)$ has a $0$-cell for each object of $\cD$, and an $(r+1)$-cell $C_M$ for each component $M$ of the $r$-dimensional moduli spaces of $\cD$.

We attach the cells of $\cw(\cD)$ following the procedure that we now outline.

Assume that all cells of index less than or equal to $i$ have been attached, giving us the $i$-skeleton $\cw(\cD)^i$.  Let $a,b \in \Ob(\cD)$ with $|a| = r$, $|b| = r + i + 1$, and let $M \subseteq \M (b,a)$ be a component of the moduli space from $b$ to $a$ (which we assume is non-empty).  Now, $M$ is a manifold with corners that is homeomorphic to the closed $i$-disc $D^i$.  We write $C(M) = M \times [r, r+i+1]$ for the $(i+1)$-cell of $\cw(\cD)$ corresponding to $M$.  We need to give the attaching map
\[ \xi : \partial C(M) \rightarrow \cw(\cD)  {\rm .} \]

We first require that $\xi(M \times \{ r \}) = \{ a \}$ and $\xi(M \times \{ r + i + 1 \}) = \{ b \}$.  Next suppose that $c \in \Ob(\cD)$ with $|c| = s$ such that $r < s < r + i + 1$.  Then $\M(c,a) \times \M(b,c) \subset \partial \M(b,a)$ so consider a component $M_{ca}$ of $\M(c,a)$ and a component $M_{bc}$ of $\M(b,c)$ such that $M_{ca} \times M_{bc} \subset \partial M$.  We now define
\[ \xi|_{M_{ca} \times M_{bc} \times [r, s]} : {M_{ca} \times M_{bc} \times [r, s]} \rightarrow M_{ca} \times [r,s] = C(M_{ca})  \]

\noindent and
\[ \xi|_{M_{ca} \times M_{bc} \times [s, r+i+1]} : {M_{ca} \times M_{bc} \times [s,r+i+1]} \rightarrow M_{bc} \times [s,r+i+1] = C(M_{bc})  \]

\noindent by projection.

It is straightforward to check that $\xi$ is well-defined.

\begin{definition}
	\label{defn:partial_flow}
	We call $\widetilde{\cF}$ a \emph{partial flow category} of the flow category $\cF$ if $\widetilde{\cF}$ is a wide subcategory of $\cF$ such that the following condition holds.  For $a,b \in \Ob(\cF)$, if $M$ is a component of $\M_{\cF}(b,a)$, then either $M \subset \M_{\widetilde{\cF}}(b,a)$ or no interior point of $M$ is in $\M_{\widetilde{\cF}}(b,a)$.
\end{definition}

If we are in the case of Definition \ref{defn:partial_flow} (and using the same notations), then if $\cF$ is a disc flow category then $\widetilde{\cF}$ determines a subcomplex $\cw(\widetilde{\cF})$ of $\cw(\cF)$ by restricting to those cells corresponding to components of the moduli spaces of $\cw(\cF)$ of full dimension.

Suppose that $\cF$ comes with an embedding, and a framing of each $0$-dimensional moduli space in subcategory $\widetilde{\cF}$, which extends coherently to the full-dimensional moduli space components of dimension $1$.  By comparison with the standard framing of the Euclidean spaces into which the points are embedded, this determines a $1$-cochain $s \in C^1(\cw(\wt{\cF}), \F_2)$, which we call a \emph{sign assignment} for $\wt{\cF}$.

We wish to ask whether we can extend the framing of the $0$-dimensional moduli spaces of $\wt{\cF}$ to a framing of the full-dimensional moduli space components of $\wt{\cF}$ \emph{coherently} in the sense of Lipshitz-Sarkar.  In other words, if $M$ is a component of $\M_{\cF}(b,a)$ and $M \subset \M_{\widetilde{\cF}}(b,a)$ then the framing of $\M_{\wt{\cF}}(b,a)$ should agree with the product framing of $\M_{\wt{\cF}}(c,a) \times \M_{\wt{\cF}}(b,c)$ on the boundary of $\M_{\wt{\cF}}(b,a)$ where defined for each $c \in \Ob(\wt{\cF})$.  The answer lies in obstruction theory.

In order to define an obstruction class, we need to be careful with orientations.  So for all full-dimensional moduli space components $N \subset \M_{\wt{\cF}}(b,a)$, fix a choice of orientation, choosing the positive orientations for each $0$-dimensional moduli space.  Then the interior of the cell $C(N)$ is identified with $\mathring{N} \times \R$ and given the product orientation.  If $b >_1 c > a$ and $M \subset \M_{\wt{\cF}}(c,a)$ is a full-dimensional component then it may be embedded, perhaps multiple times, in $\partial N$ as a product of itself with a finite number of points $p$.  For those triples $(M,N,p)$ where it makes sense, we write $t(M,N,p) \in \F_2$ where $t(M,N,p) = 0$ or $1$ depending as the chosen orientation of $M$ agrees with or differs from the orientation induced on $M = M \times \{ p \}$ as the boundary of $N$.  We make a similar definition for when $b > c >_1 a$ and $M \subset \M_{\wt{\cF}}(b,c)$.

For a full-dimensional component $M \subset\M_{\wt{\cF}}(b,a)$, we note that the differential $\delta$ in the cochain complex $C^*(\cw(\wt{\cF}), \Z)$ is given by
 \begin{equation}
 \label{eq:diffsy}
 \delta (C(M)') = \sum_{\substack{c >_1 b > a \\ N \subset \M_{\wt{\cF}}(c, a) \\ p \in \M_{\wt{\cF}}(c,b) \\ M \times \{ p \} \subset \partial N } } (-1)^{t(M,N,p)} C(N)'  + \sum_{\substack{b > a >_1 c \\ p \in \M_{\wt{\cF}}(a, c) \\ N \subset \M_{\wt{\cF}}(b,c) \\  \{ p \} \times M \subset \partial N } } (-1)^{t(M,N,p)} C(N)'
 \end{equation}

\noindent where we write $C(M)'$ for the element of the standard dual basis whose dual is $C(M)$.

\begin{definition}
	\label{defn:obstruction_class}
	Assume that all full-dimensional moduli space components of $\wt{\cF}$ of dimension less than $k$ have been coherently framed.  We shall produce a cochain $\oc \in C^{k+1}(\cw(\wt{\cF}), \pi_{k-1}(O))$, where we write $O$ for the direct limit of the orthogonal groups.
	
	Let $M$ be a full-dimensional component of the moduli space $\M_{\wt{\cF}}(b,a)$, of dimension $k$.  Then $\partial M$ is a sphere $S^{k-1}$ oriented as the boundary of $M$, which already has a framing coming from the product framing of lower-dimensional moduli spaces.
		
	Then, as in Definition 4.9 of \cite{LipSarKhov}, we compare with a null-concordant framing of $S^{k-1}$ to produce an element of $\pi_{k-1}(O)$.  Hence we get the cochain $\oc$ as required.
\end{definition}

Note that if $\oc = 0$ then the coherent framing of $\wt{\cF}$ can be extended to the full-dimensional $k$-dimensional moduli space components of $\wt{\cF}$.

\begin{proposition}
	The cochain $\oc$ is a cocycle.
\end{proposition}

\begin{proof}
	Let $M \subset \M_{\wt{\cF}}(b,a)$ be a $(k+1)$-dimensional full-dimensional moduli space component.  Then $C(M)$ is a $(k+2)$-cell of the complex $\cw(\wt{\cF})$ and we wish to show that
	\[ \langle d \oc, C(M)\rangle = 0 {\rm .} \]
	
	We write $S^k$ for $\partial M$.  Suppose that $b >_1 c > a$ and $N \times \{ p \} \subset \M_{\wt{\cF}}(c,a) \times \M_{\wt{\cF}}(b,c)$ is a disc in $\partial M$, where $N$ is a component of  $\M_{\wt{\cF}}(c,a)$.  We write $D(N,p) = N \times \{ p \} \subset S^k$ for the closed disc.  Similarly, if $b > c >_1 a$ we get closed discs $D(p,N) \subset  \M_{\wt{\cF}}(c,a) \times \M_{\wt{\cF}}(b,c)$.
	
	Let
	\[ B = S^k \setminus \left( \bigcup_{N,p} \mathring{D}(N,p) \cup \bigcup_{p,N}  \mathring{D}(p,N) \right) {\rm ,}\]
	
	\noindent where the sum is taken over all $p$, $N$ for which it makes sense.
	
	Now a map $\Phi : B \rightarrow O$ induces by restriction maps $\partial D(p,N) \rightarrow O$ and $\partial D(N,p) \rightarrow O$ and hence gives rise to $\phi_{p,N} \in \pi_{k-1}(O)$ and $\phi_{N,p} \in \pi_{k-1}(O)$.  We claim that such maps must satisfy
	\begin{equation}
	\label{eq:phisy}
	\sum_{p,N} \phi_{p,N} + \sum_{N,p} \phi_{N,p} = 0 {\rm .}
	\end{equation}
	
	\noindent This follows from the fact that $B$ is $S^k$ with a number of $k$-discs removed.
	
	To construct the map $\Phi$, note that $B$ is comprised of products of moduli spaces, each of which has already been coherently framed by hypothesis.  So take any framing $f$ of $M$, and denote by $f_0$ the pullback to $B$.  Then there is map $g_0: B \rightarrow O(d)$ for some $d$ (depending on the dimension of the Euclidean space taken for the embedding of $\cF$) such that $g_0(f_0)$ is the coherent framing we have by hypothesis.  The map $\Phi$ is the composition of $g_0$ with the inclusion of $O(d)$ in $O$.
	
	Now we ask how the group elements $\phi_{p,N}$ and $\phi_{N,p}$ are related to the group elements $\langle\oc, N\rangle$.
	
	In the case of $\phi_{p,N}$, simply note that $f$ restricted to $\{ p \} \times N$ gives a null-concordant framing of $\{ p \} \times N$, so that $\phi_{p,N}$ is exactly $\pm\langle\oc , N\rangle$, depending as our chosen orientation on $N$ agrees or differs with the orientation on $N = \{ p \} \times N$ induced as part of the boundary of $M$.  In other words,
	\[ \langle\oc , N\rangle = (-1)^{t(M,N,p)}\phi_{p,N} {\rm ,}\]
	
	\noindent and similarly for $\phi_{N,p}$
	\[ \langle\oc, N\rangle = (-1)^{t(M,N,p)} \phi_{N,p} {\rm .} \]
	
	Combining this with \eqref{eq:diffsy} and \eqref{eq:phisy}, we see that
	\begin{align*}
	\langle\delta \oc,C(M)\rangle &= \sum_{\substack{b >_1 c > a \\ N\subset \M(c,a) \\ p \in \M(b,c) \\ N \times \{ p \} \subset \partial M }} (-1)^{t(N,M,p)} \oc(C(N)) + \sum_{\substack{b>c >_1 a \\ p \in \M(c,a) \\ N \subset \M(b,c) \\ \{ p \} \times N}} (-1)^{t(N,M,p)} \oc(C(N)) \\ &= \sum_{p,N} \phi_{p,N} + \sum_{N,p} \phi_{N,p} = 0 {\rm .} \qedhere
	\end{align*}
\end{proof}

Now suppose that $\cF$ has been embedded and that all the full-dimensional moduli space components of $\wt{\cF}$ of dimension less than $k$ have been coherently framed.  Let $\tau \in C^{k}(\cw(\wt{\cF}), \pi_{k-1}(O))$ be a cochain.  We use $\tau$ to modify the framings of the moduli spaces components of dimension $k-1$ as follows.  If $M$ is such a component then take a small ball $B \subset M$.  Now $\tau(C(M))$ is an element of $\pi_{k-1}(O)$ and so we represent $\tau$ by a map $B^{k-1} \rightarrow O(n)$ which takes the boundary $S^{k-2}$ to the identity (for some large $n$).  Hence, after stabilization if necessary, we can use this map to alter the framing of $M$ in the interior of $B$.

This $\tau$-modification clearly alters the cocycle $\oc$ by addition of $\delta \tau$.  The next proposition is immediate.

\begin{proposition}
	\label{prop:obstruction_class}
If the cocycle $\oc$ of Definition \ref{defn:obstruction_class} satisfies
\[ [\oc] \in H^{k+1}(\cw(\wt{\cF}), \pi_{k-1}(O)) = 0 \]

\noindent then the coherent framing of the moduli space components of $\wt{\cF}$ can be extended to the full-dimensional $k$-dimensional moduli space components of $\wt{\cF}$ after perhaps modifying the framing on the $(k-1)$-dimensional components. \qed
\end{proposition}

Now we wish to ask under which conditions, if we are given two framings of the full-dimensional moduli space components of $\wt{\cF}$, we can be sure of obtaining the same stable homotopy type for a framed flow category described as a disc flow cover of $\wt{\cF}$ (see Definition \ref{defn:disc_cover}).  A situation in which we certainly get the same stable homotopy type is when the two framings of $\wt{\cF}$ can be related through an isotopy of framed embeddings, and this situation can be analysed by obstruction theoretic arguments very similar in flavor to those above.  We therefore allow ourselves to be more brief.

Firstly, let us write $\iota_t$ for a 1-parameter family of embeddings of $\cF$ where $0 \leq t \leq 1$, and assume that $\iota_0$ and $\iota_1$ both have given coherent framings of the full-dimensional moduli space components of $\wt{\cF}$ which give identical sign assignments.  Since the sign assignments of $\iota_0$ and $\iota_1$ agree, we can extend the framings at $t=0$ and $t=1$ to frame each $0$-dimensional moduli space of $\iota_t$ for $0 < t < 1$.

Assume inductively then that we have extended the framings of the full-dimensional $i$-dimensional moduli spaces of $\iota_0$ and $\iota_1$ to framings of the full-dimensional $i$-dimensional moduli spaces of each $\iota_t$ for all $i < k$ for some $k \geq 1$.  We define an obstruction class $\oc$.

\begin{definition}
	\label{defn:obstruction_class_2}
If $M$ is a $k$-dimensional full-dimensional moduli space component of $\M_{\wt{\cF}}(b,a)$, then the $k$-sphere 
\[ S^k_M = \iota_0(M) \cup \iota_1(M) \cup \bigcup_t \iota_t{\partial M} \]

\noindent has already been framed.  Comparing this framing to a null-homotopic framing for each $M$ gives an element of $\pi_k (O)$ thus giving a cochain
\[ \oc \in C^{k+1}(\cw(\wt{\cF}),\pi_k(O) ) {\rm .}\]
\end{definition}

\noindent It is straightforward to check that $\oc$ is a cocycle.

If now $\tau \in C^{k}(\cw(\wt{\cF}),\pi_k(O) )$ is a cochain, we can use $\tau$ to modify the framings of the $(k-1)$-dimensional full-dimensional components of each embedding $\iota_t$ as follows.  Let $C_N \in C_{k}(\cw(\wt{\cF}))$ and consider points $p$ in the $0$-dimensional moduli spaces of $\wt{\cF}$ such that either $N \times \{ P \}$ or $\{p\} \times N$ appears in the boundary of $k$-dimensional full-dimensional moduli spaces components.  Then for each such $p$ consider a small ball in the interior of each $k$-disc $[0,1] \times N \times \{ p \}$ or $[0,1] \times \{ p \} \times N$.  After stabilization if necessary, modify the framings on these small balls by the element $\tau(C_N) \in \pi_k(O(r)) = \pi_k(O)$ for large enough $r$.  It is then an easy check that this modifies the cocycle $\oc$ by the coboundary $\delta \tau$.  Hence we immediately obtain the following proposition.

\begin{proposition}
	\label{prop:obstruction_class_2}
	If the cocycle $\oc$ of Definition \ref{defn:obstruction_class_2} satisfies
	\[ [\oc] \in H^{k+1}(\cw(\wt{\cF}), \pi_{k}(O)) = 0 \]
		
	\noindent then the coherent framing of the full-dimensional moduli space components $\iota_0$ and $\iota_1$ of dimension $k$ can be extended to the $1$-parameter family of framed embeddings $\iota_t$, after perhaps modifying the framing on the $(k-1)$-dimensional components away from $t=0$ and $t=1$. \qed
\end{proposition}

Now we wish to apply Propositions \ref{prop:obstruction_class} and \ref{prop:obstruction_class_2} to the case of the glued flow category $\cL^n(D_\br)$ defined in Subsection \ref{subsec:matched_category} associated to a glued diagram $D_\br$.  This arose as a \emph{disc cover} (see Definition \ref{defn:disc_cover}) of the sock flow category $\cS^n_\br$.  In fact, the covering morphism $\Phi: \cL^n(D_\br) \rightarrow \cS^n_\br$ factors through a partial flow category $\wt{\cS^n_\br}$ which we now describe.

We begin with $\wt{\cS^n_r}$ where $r\in \Z-\{0\}$, and include all $0$-dimensional moduli spaces of $\cS^n_r$. Recall from Subsection \ref{subsec:the_n_sock} that the $0$-dimensional moduli spaces are denoted $P$, $M$ and $P_i$ with $i=1,\ldots,n$. The $1$-dimensional moduli spaces have boundary $\{(P,P_i),(M,P_{i+1})\}$ or $\{(P_i,P),(P_{i+1},M)\}$ for $i=1,\ldots,n$ (with identification $P_1=P_{n+1}$).
We discard the interior of intervals with boundary $\{(P,P_n),(M,P_1)\}$ and $\{(P_n,P),(P_1,M)\}$, but include in $\wt{\cS^n_r}$ all other $1$-dimensional moduli spaces. Note that the moduli spaces $\M(2,0)$ and $\M(0,-2)$ contain exactly one interval each, and these are meant to be in $\wt{\cS^n_r}$. 

The $2$-dimensional moduli spaces $\M(2j+2,2j-1)$ for $j\geq 1$ or $\M(2j+1,2j-2)$ for $j\leq -1$ are $n$ disjoint copies of squares, as observed in Subsection \ref{subsec:the_n_sock}. We discard the squares with corners $\{(P,P_{n-1},P),(P,P_n,M),(M,P_1,M),(M,P_n,P)\}$ and $\{(P,P_n,P),(P,P_1,M),(M,P_2,M),(M,P_1,P)\}$, but include the other $n-2$ squares in $\wt{\cS^n_r}$. This describes all the full-dimensional moduli spaces of $\wt{\cS^n_r}$. Note that $\wt{\cS^2_r}$ does not contain $2$-dimensional moduli spaces.

\begin{definition}
	\label{defn:partial_sock}
	For $\br = (r_1, r_2, \ldots, r_m)$, the partial flow category $\wt{\cS^n_\br}$ is defined as the largest partial flow category of $\cS^n_\br$ satisfying the following property.
	
	If $a=(a_1,\ldots,a_j,\ldots,a_m)$ and $b=(a_1,\ldots,a_j+i,\ldots,a_m)$ are objects of $\cS^n_\br$, then
\[
\M_{\wt{\cS^n_\br}}(b,a)=\M_{\wt{\cS^n_{r_j}}}(a_j+i,a_j).
\]
\end{definition}

\begin{proposition}
	\label{prop:cover_partial_sock}
	The covering morphism $\Phi: \cL^n(D_\br) \rightarrow \cS^n_\br$ factors through $\wt{\cS^n_\br}$.
\end{proposition}

\begin{proof}
	Certainly those moduli spaces in the image of $\Phi$ form a partial flow category of $\cS^n_\br$.  It is therefore sufficient to check that the discarded intervals from Definition \ref{defn:partial_sock} are never in this image.  This follows from the observation that there are no three objects $a,b,c$ of $\cL^n(D_\br)$ such that
	\begin{align*} 
	 \Phi(a) &= (i_1, \ldots, i_j, \ldots, i_m),\\ \Phi(b) &= (i_1, \ldots, i_j + 1, \ldots, i_m),\\ \Phi(c) &= (i_1, \ldots, i_j + 2, \ldots, i_m) 
	\end{align*}

	\noindent with $0 \notin \{ i_j, i_j + 2 \}$ that give rise to a composition of points $P_n$ and $P$, or $P_1$ and $M$ in their moduli spaces.  At the chain complex level this is merely the observation that a $n$-dotted cylinder cobordism is trivial, or in other words that multiplication by $x^n$ is the trivial map on $\Z[x]/x^n$. Also, compositions that correspond to the discarded $2$-dimensional moduli spaces are not possible for the same reason.
\end{proof}

We wish to use the obstruction theory that we have set up in the specific case of the partial flow category $\wt{\cS^n_\br}$ and its cover $\cL^n(D_\br)$.  It turns out that this is a particularly amenable situation.

\begin{proposition}
	\label{prop:obstruct_vanishes}
	The CW-complex $\cw(\wt{\cS^n_\br})$ deformation retracts to a point.
\end{proposition}

We break the proof of this into a couple of lemmata.

\begin{lemma}
	\label{lem:morseproductobst}
	Suppose that $f : M \rightarrow \R$ and $g : N \rightarrow \R$ are Morse functions and let $F = f + g : M \times N \rightarrow \R$.  Given Morse-Smale gradients for $f$ and $g$ we form the flow categories $\Cat_f$ and $\Cat_g$.  Concatenating the gradients we form the flow category $\Cat_F$.  We assume that $\Cat_f$ and $\Cat_g$ are both disc flow categories (it follows that $\Cat_F$ is also by Lemma \ref{lem_prodflowcat}).  Then as topological spaces, we have 
	\[ \cw(\Cat_F) = \cw(\Cat_f) \times \cw(\Cat_g) {\rm .} \]
\end{lemma}

\begin{proof}
	Let $a$, $b$ be critical points of $f$ with $\ind(a) - \ind(b) = m > 1$ and let $p$, $q$ be critical points of $g$ with $\ind(p) - \ind(q) = n > 1$.  We induct on $m+n$, with the root case being clear.  We assume without affecting the argument (apart from the clarity of notation) that $\ind(b) = \ind(q) = 0$.
	
	We borrow some notations from the proof of Lemma \ref{lem_prodflowcat}.
	
	Consider a component $M_{(ap,bq)} \subset \M_{\Cat_F}(ap, bq)$.  And let $M_{(a,b)} \subset \M_{\Cat_f}(a,b)$ and $M_{(p,q)} \subset \M_{\Cat_g}{(p,q)}$ be the corresponding components so that $M_{(ap,bq)} \cong M_{(a,b)} \times M_{(p,q)} \times [0,1]$.
	
	For a general disc flow category $\Cat$ and a component $M \subset \M_\Cat(x,y)$ with $|x| = T$ and $|y| = 0$, the attaching map of the corresponding $N$-cell $C(M)$ of $\cw(\Cat)$, factors through the map
	\[ \partial (C(M)) = \partial( M \times [0,T]) \rightarrow S^{T-1}_M \cong S^{T-1}  \]
	
	\noindent given as a composition of quotient maps in the following way.  First crush each disc $M \times \{ 0 \}$ and $M \times \{ T \}$ to points.  Then we consider each sequence of components
	\[ N_1 \subset \M_\Cat(c_1,c_0), \ldots , N_k \subset \M_\Cat(c_k,c_{k-1}) \]
	
	\noindent such that $c_0 = b$, $c_k = a$, and $0 < |c_1|< \cdots < |c_{k-1}| < T$ and $N_1 \times \cdots \times N_k \subset \partial M$.  We crush $N_1 \times \cdots \times N_k \times [0, T]$ to 
	\[ (N_1 \times [0,|c_1|]) \cup \cdots \cup (N_k \times [|c_{k-1}|,T]) \]
	
	\noindent in the obvious way.  That this gives topologically a $(T-1)$-sphere follows from the fact that each $N_i$ is a disc.
	
	This gives us a useful shorthand for thinking of the attaching map since it includes a description of $S^{T-1}_M$ itself as a cell-complex (the cells being subsets of the form $N_i \times [|c_i|,|c_{i-1}|]$) with an implicit cellular map into the $(T-1)$-skeleton of $\cw(\Cat)$ that gives the attaching map of $C(M)$.
	
	Now consider the cell of $\cw(\Cat_F)$ corresponding to $M_{(ap,bq)}$,
	\[ C(M_{(ap,bq)}) = M_{(ap,bq)} \times [0,m+n] {\rm .} \]
	
	\noindent The attaching map of $C(M_{(ap,bq)})$ factors through the quotient map
	\[ \partial(C(M_{(ap,bq)})) \rightarrow S^{m+n-1}_{M_{(ap,bq)}} \cong S^{m+n-1} \]
	
	\noindent as described above.
	
	We wish to show that
	\[ S^{m+n-1}_{M_{(ap,bq)}} = C(M_{(a,b)}) \times S^{n-1}_{M_{(p,q)}} \cup S^{m-1}_{M_{(a,b)}} \times C(M_{(p,q)}) {\rm ,} \]
	
	\noindent which implies that $\cw(\Cat_F)$ is homeomorphic to the product CW-complex $\cw(\Cat_f) \times \cw(\Cat_g)$.
	
	For each $c \in \Ob(\Cat_f)$ such that $M_{(a,b)} \cap (\M_{\Cat_f}(c,b) \times \M_{\Cat_f}(a,c)) \not= \phi$, we have $\M_{\Cat_F}(cp,cq) \cap M_{(ap,bq)} \cong M_{(p,q)}$.
	
	We know from the construction of the obstruction complex associated to a disc flow category that
	\begin{eqnarray*} S_{(M_{(ap,bq)})} &=& \left( \bigcup_{\substack{c <_1 a \\ N \in \comp(\M_{\Cat_F}(cp,bq)) \\ x \in \M_{\Cat_F}(ap,cp) \\ N \times \{ x \}  \in \partial M_{(ap,bq)}  }} C(N) \right) \cup \left( \bigcup_{\substack{b <_1 c \\ N \in \comp(\M_{\Cat_F}(ap,cq)) \\ x \in \M_{\Cat_F}(cq,bq) \\ \{ x \} \times N \in \partial M_{(ap,bq)}  }} C(N) \right) \\
		&\cup& \left( \bigcup_{\substack{r <_1 p \\ N \in \comp(\M_{\Cat_F}(ar,bq)) \\ x \in \M_{\Cat_F}(ap,ar) \\ N \times \{ x \}  \in \partial M_{(ap,bq)}  }} C(N) \right) \cup \left( \bigcup_{\substack{q <_1 r \\ N \in \comp(\M_{\Cat_F}(ap,br)) \\ x \in \M_{\Cat_F}(br,bq) \\ \{ x \} \times N \in \partial M_{(ap,bq)}  }} C(N) \right)
		\end{eqnarray*}

By the induction hypothesis, the bottom line is $C(M_{(a,b)}) \times S^{n-1}_{M_{(p,q)}}$, while the top line is $S^{m-1}_{M_{(a,b)}} \times C(M_{(p,q)})$.

Furthermore, we see that the intersection is given by
\begin{align*}
	(C(M_{(a,b)}) \times S^{n-1}_{M_{(p,q)}}) &\cap (S^{m-1}_{M_{(a,b)}} \times C(M_{(p,q)})) \\
	&=  \bigcup_{\substack{a \geq c_1 >_{(m-1)} c_2 \geq b, \,\, p \geq d_1 >_{(n-1)} d_2 \geq q \\ N \in \comp(\M_{\Cat_F}(c_1 d_1,c_2 d_2)) \\ x \, \, {\rm such} \,\, {\rm that} \,\, N \times \{ x \}  \in \partial M_{(ap,bq)}  }} C(N) \\
	&= S^{m-1}_{M_{(a,b)}} \times S^{n-1}_{M_{(p,q)}}  \qedhere
\end{align*}
\end{proof}

Lemma \ref{lem:morseproductobst} concerns disc flow categories arising as Morse flow categories.  We have used in the proof the fact that we understand for a Morse flow category how $M_{(ap,bq)}$ is constructed as a manifold with corners from $M_{(a,b)}$ and $M_{(p,q)}$.  The flow category $\cS^n_\br$ is built from a Morse flow category for products of lens space by identifying certain moduli spaces.  From the construction it is immediate that the moduli spaces for $\cS^n_{(r_1,\ldots,r_m)}$ can be constructed as manifolds with corners from those of $\cS^n_{(r_1,\ldots,r_{m-1})}$ and $\cS^n_{r_m}$ and that this construction is the same as if the categories were Morse.  Hence the arguments of the lemma apply to $\cS^n_\br$ and we see that
\[ \cw(\cS^n_{(r_1, \ldots, r_m)}) = \prod_{i=1}^{m} \cw(\cS^n_{r_i}) {\rm .} \]

Now $\wt{\cS^n_\br}$ is a partial flow category of $\cS^n_\br$, defined to the be largest such that omits the interiors of certain $1$- and $2$-dimensional moduli spaces of $\cS^n_\br$ (see Definition \ref{defn:partial_sock} for details).  Hence the associated CW-complex $\cw(\wt{\cS^n_\br})$ is the largest subcomplex of $\cw(\cS^n_\br)$ that omits the corresponding $2$- and $3$-cells.  In the identification above, the omitted $2$-cells are products of $2$-cells in each factor with all the $0$-cells of the other factors, and similarly with the $3$-cells.  It follows that
\[ \cw(\wt{\cS}^n_{(r_1, \ldots, r_m)}) = \prod_{i=1}^{m} \cw(\wt{\cS}^n_{r_i}) {\rm .} \]

So to prove Proposition \ref{prop:obstruct_vanishes} it is enough to verify the following lemma.

\begin{lemma}
	\label{lem:sockobscomptriv}
	The complex $\cw(\wt{\cS^n_r})$ is contractible.
\end{lemma}

\begin{proof}
We give the proof in the case $r>0$, the other case being similar.

The partial flow category $\wt{\cS^n_r}$ has full-dimensional moduli spaces only of dimensions $0$, $1$ and $2$, hence the associated complex $\cw(\wt{\cS^n_r})$ is a 3-complex.  To construct $\cw(\wt{\cS^n_r})$ we start with $r+1$ $0$-cells, one for each object of $\wt{\cS^n_r}$, labelled $0, 1, \ldots , r$. 
We then add a $1$-cell $P$ between $0$ and $1$, and $1$-cells $P_i$, $M_i$ between $2i-1$ and $2i$ for $i=1,\ldots,\lfloor r/2\rfloor$, and $1$-cells $P_{i,0},\ldots,P_{i,n-1}$ between $2i$ and $2i+1$ for $i=1,\ldots,\lfloor (r-1)/2 \rfloor$.

From the general construction of $\cw(\cS^n_r)$ it is clear that the 2-cells are attached by mapping the edges of a square linearly over intervals, and the $2$-cells with the algebraic boundary information is given as follows.

There is a $2$-cell $\tilde{N}_0$ with
\[
 \partial \tilde{N}_0 = P_1 - M_1,
\]
and $2$-cells $N_{i,0},\ldots, N_{i,n-2}$ for $i=1,\ldots,\lfloor (r-1)/2 \rfloor$ with
\[
 \partial N_{i,j} = P_i + P_{i,j}-P_{i,j+1} - M_i,
\]
and finally $2$-cells $\tilde{N}_{i,0},\ldots, \tilde{N}_{i,n-2}$ for $i=1,\ldots,\lfloor (r-1)/2 \rfloor$ with
\[
 \partial \tilde{N}_{i,j} = P_{i,j}+P_{i+1}-M_{i+1}-P_{i,j+1}.
\]
The $3$-cells are given by $Q_{i,0},\ldots, Q_{i,n-3}$ for $i=1,\ldots,\lfloor (r-1)/2 \rfloor$ with
\[
 \partial Q_{i,j} = N_{i,j}+\tilde{N}_{i,j+1}- \tilde{N}_{i,j} - N_{i,j+1}.
\]
Using induction on $r$ it is now easy to see that $\cw(\wt{\cS^n_r})$ is simply connected, and has trivial homology.
\end{proof}

Now, noting that by Proposition \ref{prop:obstruction_class} and the vanishing of the class described there, we see that there exists a framed embedding of $\wt{\cS^n_\br}$ extending any given sign assignment.  Furthermore, we have the following theorem.

\begin{theorem}
	\label{thm:sock_frame_doesnt_matter}
	Given two framed embeddings $\iota_0$ and $\iota_1$ of the sock flow category $\wt{\cS^n_\br}$ extending a given sign assignment, we obtain as perturbed lifts two framed embeddings $i_0$, $i_1$ of the flow category $\cL^n(D_\br)$.  Then we have that $| (\cL^n(D_\br), i_0) |$ and  $|(\cL^n(D_\br), i_1)|$ are stably homotopy equivalent.
\end{theorem}

\begin{proof}
	This follows from Proposition \ref{prop:obstruction_class_2}, since the obstruction class there described automatically vanishes by Lemma \ref{lem:sockobscomptriv}.
\end{proof}

We wish also to check invariance of the stable homotopy type under the choice of sign assignment for the sock flow category $\wt{\cS^n_\br}$.  The proof of this result is very similar to that of part (4) of Proposition 6.1 of \cite{LipSarKhov}, so we allow ourselves to be brief.

\begin{proposition}
	\label{prop:ind_of_sign_assignment}
	Given two sign assignments $s_0$ and $s_1$ of the sock flow category $\wt{\cS^n_\br}$ we extend each to two framed embeddings, giving rise to two framed embeddings of $\cL^n(D_\br)$.  Then we have that $| (\cL^n(D_\br), s_0) |$ and  $|(\cL^n(D_\br), s_1)|$ are stably homotopy equivalent.
\end{proposition}

\begin{proof}
	If $\br = (b_1, \ldots, b_s)$ let $\br' = (b_1, \ldots, b_s, 1)$ and $D'_{\br'}$ be the disjoint union of $D_\br$ with the 1-crossing diagram of the unknot chosen so that the handle addition corresponding to the last coordinate of $\br'$ joins two circles into one.
	
	Now $\wt{\cS^n_\br}$ is a full subcategory of $\wt{\cS^n_{\br'}}$ twice: once by restricting to all objects with last coordinate 0, and once by restricting to last coordinate 1 (and so once as an upward-closed and once as a downward-closed subcategory).  It is straightforward to check that there exist exactly two choices of sign assignment on $\wt{\cS^n_{\br'}}$ such that the induced sign assignments on the subcategories are $s_0$ and $s_1$.  Let $s$ be one of these choices.
	
	Let now $\cC$ be the full subcategory of $\cL^n(D'_{\br'})$ containing those objects in which each circle arising from a resolution of the 1-crossing diagram of the unknot is decorated with $1$.  In the obvious way, we see that $\cL^n(D_\br)$ appears as both an upward-closed and as a downward-closed subcategory of $\cL^n(D'_{\br'})$ and that the sign assignments induced by $s$ on these copies of $\cL^n(D_\br)$ are $s_0$ and $s_1$.
	
	It is clear that the cohomology of $|(\cC, s)|$ is trivial so that $|(\cC,s)|$ is contractible.  Hence it follows that $| (\cL^n(D_\br), s_0) |$ and  $|(\cL^n(D_\br), s_1)|$ are stably homotopy equivalent.
\end{proof}

We verify invariance of the stable homotopy type $|\cL^n(D_\br)|$ under various changes of input data in the next couple of subsections.  We include here the proof of invariance under the choice of left and right endpoints of an elementary tangle summand of $D_\br$.

\begin{proposition}
	\label{prop:left_and_right_invariance}
	The stable homotopy type $|\cL^n(D_\br)|$ is independent of the choice of left and right endpoints of the elementary summands of $D_\br$.
\end{proposition}

\begin{proof}
	Suppose $\br = (b_1, \ldots, b_s)$ and let $1 \leq i \leq s$.  The result follows from the observation that there exists an isomorphism of partial flow categories
	\[ \Phi_i : \wt{\cS^n_\br} \rightarrow  \wt{\cS^n_\br} \]
	\noindent with the following two properties.  Firstly $\Phi_i$ is the identity on objects and the identity on $0$-dimensional moduli spaces corresponding to a shift of a coordinate different from the $i$th.  Secondly $\Phi_i$ acts in the following way on all $0$-dimensional moduli spaces corresponding to a shift in the $i$th coordinate:
	\[ \Phi_i : P \mapsto M, M \mapsto P, P_j \mapsto P_{n+1-j} \]
	\noindent for $j = 1, \ldots, n$.
	
	Finally note that precomposing the covering morphism $\cL^n(D_\br) \rightarrow \wt{\cS^n_\br}$ with $\Phi_i$ gives exactly the covering corresponding to exchanging the left and right endpoints on the $i$th elementary tangle summand of $D_\br$.
\end{proof}

\subsection{Equivalence to the Lipshitz-Sarkar stable homotopy type for $n=2$}
\label{subsec:easier_induction}
Suppose that we are given a glued diagram $D_{\br}$ consisting of $m$ elementary tangles of indices $\br = (r_1, \ldots, r_m)$ (see Figure \ref{2twist}).  We can apply the construction of Subsection \ref{subsec:matched_category} to $D_\br$ and produce a framed flow category $\cL^n(D_\br)$.  This then produces via the construction of Cohen-Jones-Segal a stable homotopy type $|\cL^n(D_\br)|$.

The relation of the stable homotopy type $|\cL^2(D_\br)|$ to the Lipshitz-Sarkar stable homotopy type will be seen to be given by cohomological and quantum shifts.  Let us be more specific about these shifts here.  We have already defined the cohomological degree of the objects in a matched flow category associated to a diagram, it remains to give the quantum degree.  Since the quantum degree is not invariant under the isomorphism between Khovanov cohomology and $\SL_2$ Khovanov-Rozansky cohomology, we shall treat these cases differently.

Let therefore $\cL^{\Kh} (D_\br)$ be a copy of $\cL^2(D_\br)$.  Consider the object of $\cL^{\Kh}(D_\br)$ sitting in multidegree $(r_1, \ldots, r_m)$ obtained by decorating all the circles in the corresponding resolution by $1$.  We set the quantum degree of this object in $\cL^{\Kh}(D_\br)$ to be
\[ R - \sum_{i=1}^m {\rm sgn}(r_i) + w(D) + \sum_{i=1}^m s_i r_i + c{\rm ,}\]

\noindent where $R := r_1 + \cdots r_m$, $w(D)$ is the writhe of the diagram $D$, $s_i$ is either $1$ or $0$ depending on whether the strands of the $i$th tangle carry the orientation of a 2-braid or not, and $c$ is the number of circles in the resolution.  The quantum gradings of all other objects can then be determined by the requirements that the differential preserves the quantum grading and that multiplication by $x$ on any resolution circle shifts the quantum grading by $-2$.

We can break $D_\br$ into $|r_1| + \cdots +|r_m|$ tangles, where the $i$th elementary tangle becomes $|r_i|$ tangles, each of index $r_i/|r_i|$.  We write this $D_{(\pm 1,\ldots, \pm 1)}$.  Restricting now to $n=2$, the construction of Subsection \ref{subsec:matched_category} assigns to $D_{(\pm 1,\ldots, \pm 1)}$ the framed flow category $\cL^{\Kh}(D_{(\pm 1,\ldots, \pm 1)})$, which clearly agrees with the framed flow category assigned to the underlying diagram $D$ by the Lipshitz-Sarkar construction up to overall shifts.  More specifically, denoting a cohomological shift in the square parentheses,
\[ \cL^{\Kh}(D_{(\pm 1,\ldots, \pm 1)}) \left[ \frac{w(D) - R}{2} \right] \]

\noindent is seen to be the bigraded framed flow category associated to the diagram $D$ by the Lipshitz-Sarkar construction.

\begin{definition}
We shall write $\X^{\Kh}(D_\br)$ for the stable homotopy type associated to the framed flow category $\cL^{\Kh}(D_\br)$ by the Cohen-Jones-Segal construction.  Note that the cohomology of $\X^{\Kh}(D_\br)$ is bigraded by the quantum degree and the underlying cohomological degree of the objects of $\cL^{\Kh}(D_\br)$.
\end{definition}

We shall shortly prove the following theorem

\begin{theorem}
	\label{thm:main_knot_theorem}
	We have
	\[ \X^{\Kh}(D_\br) \simeq \X^{\Kh}(D_{(\pm 1,\ldots, \pm 1)}) {\rm .} \]
\end{theorem}

This theorem holds in the bigraded sense, and it clearly implies Theorem \ref{thm:Lip-Sar_equivalence} stated in the introduction.

Suppose that $D$ is a knot diagram admitting a decomposition into $m$ elementary tangles of indices $\br = (r_1 + 1, r_2, \ldots, r_m)$ where we assume that $r_1 \geq 1$.  We write the diagram together with the data of the decomposition as $D_\br$.  By breaking the first elementary tangle $T_{r_1 + 1}$ into the concatenation of $T_1$ and $T_{r_1}$, we obtain a new decomposition of $D$ into tangles of indices $\wt{\br} = (1, r_1, r_2, \ldots, r_m)$, which we write as $D_{\wt{\br}}$.  Assume that $D$ is oriented.

\begin{proposition}
	\label{prop:all_discs_consequence}
	We have
	\[ \X^{\Kh}(D_{\wt{\br}}) \simeq \X^{\Kh}(D_\br) {\rm .} \]
\end{proposition}

Together with its obvious sister theorem covering the case $r_1 < 0$ (we shall not discuss the proof of this negative case as it follows the proof of the positive case very closely) Proposition \ref{prop:all_discs_consequence} certainly implies Theorems \ref{thm:main_knot_theorem} and \ref{thm:Lip-Sar_equivalence}.

\begin{figure}
	\centerline{
		{
			\psfrag{ldots}{$\ldots$}
			\psfrag{L}{$L$}
			\psfrag{R}{$R$}
			\psfrag{n+1}{$r_1 + 1$}
			\psfrag{n}{$r_1$}
			\psfrag{1}{$1$}
			\psfrag{-1}{$-1$}
			\includegraphics[height=2in,width=3.5in]{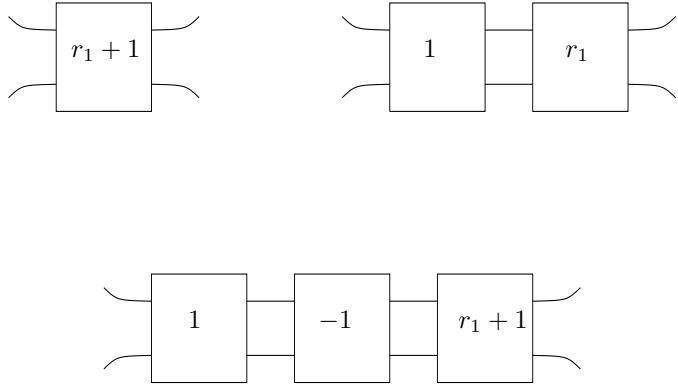}
		}}
		\caption{In this diagram the numbered boxes represent tangle diagrams consisting of the given number of horizontal positive half-twists (negative numbers mean negative half-twists).  Clearly all three tangles are isotopic.}
		\label{fig:induction_twist}
	\end{figure}

We know of two approaches to proving Proposition \ref{prop:all_discs_consequence}.  The first is the most direct: one performs repeated `handle cancellation' on the category $\cL^{\Kh}(D_{\wt{\br}})$ (mirroring the gauss elimination that can be performed at the cochain level) until one arrives at the category 
$\cL^{\Kh}(D_\br)$.  This direct approach does work, but the argument needed to show that the framed flow category resulting from the cancellation is indeed $\cL^{\Kh}(D_\br)$ (with a framing induced as a cover of the sock flow category $\cS_{\br}$) turns out to be rather long.  The second approach avoids these difficulties by beginning with a trick.

Figure \ref{fig:induction_twist} shows three tangle diagrams of the same tangle, representing the flow categories $\cL^{\Kh}(D_{\br})$, $\cL^{\Kh}(D'_{(1, -1, r_1 + 1,  r_2, \ldots, r_m)})$, and $\cL^{\Kh}(D_{\wt{\br}})$.  Rather than show directly that $|\cL^{\Kh}(D_{\wt{\br}})| \simeq  |\cL^{\Kh}(D_\br)|$, the second approach goes via the intermediate flow category.

We suppose now $\wt{D}$ is a knot diagram admitting a decomposition into $m$ elementary tangles of indices $\wt{\bs} = (-1, s_1, s_2, \ldots, s_m)$ where we assume that $s_1 \geq 1$ and we include the data of the decomposition with diagram as $\wt{D}_{\wt{\bs}}$.  We assume that the first tangle of index $-1$ is glued to the left of the second tangle of index $s_1$.  There is as before an isotopic composition of tangles of indices $\bs = (s_1 - 1 , s_2, \ldots, s_m)$ we write as $D_\bs$.
Now Figure \ref{fig:induction_twist} reveals that Proposition \ref{prop:all_discs_consequence} is a consequence of the following proposition.

\begin{proposition}
	\label{prop:Reidemeister_II_trick}
	We have
	\[ \X^{\Kh}(\wt{D}_{\wt{\bs}}) \simeq \X^{\Kh}(D_\bs)
	{\rm .} \]
\end{proposition}

Before proceeding we examine the situation in more detail and establish some notation.  We start by considering the cochain complex associated to the flow category $\cL^{\Kh}(\wt{D}_{\wt{\bs}})$, as depicted in Figure \ref{fig:complex}.

\begin{figure}
\begin{psfrags}
		\psfrag{+}{$+$}
		\psfrag{-}{$-$}
		\psfrag{pm}{$\pm$}
		\psfrag{mp}{$\mp$}
		\psfrag{ldots}{$\ldots$}
		\begin{center}
			\includegraphics[height=3.2in,width=4.8in]{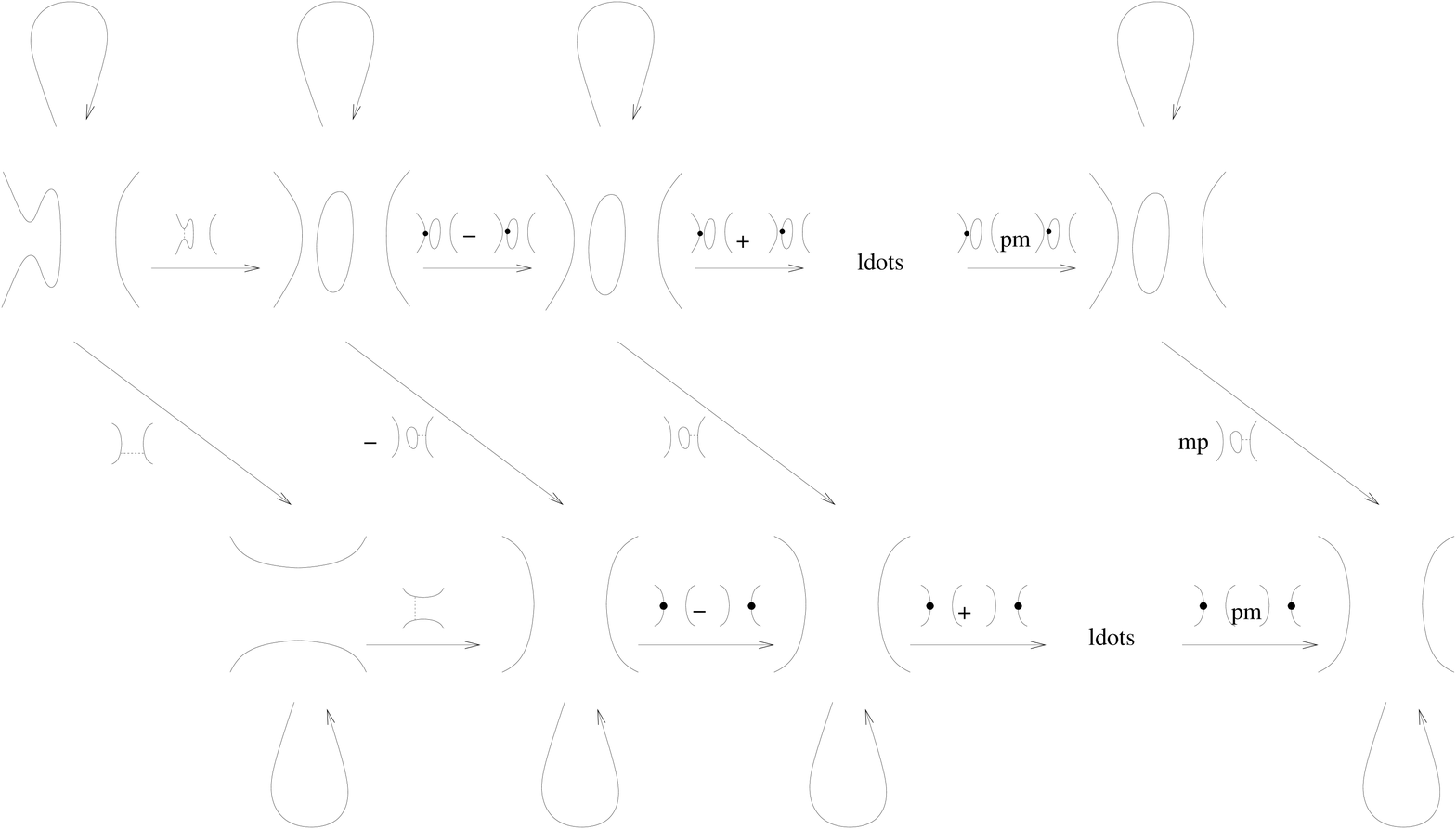}
		\end{center}
	\end{psfrags}
	\caption{The cochain complex associated to the flow category $\cL^{\Kh}(\wt{D}_{\wt{\bs}})$.}
	\label{fig:complex}
\end{figure}

In the cochain complex, each standard generator is given by a choice of vertex (an $(m+1)$-tuple of integers $(i_0, i_1, \ldots, i_m)$, in which the $i_j$ lies between $0$ and $r_m$) and then a decoration of each circle in the smoothing at that vertex with either $x_+$ or $x_-$.  These generators are in one-to-one correspondence with the set $\Ob(\cL^{\Kh}(\wt{D}_{\wt{\bs}}))$, so we shall move freely between the two concepts.

The picture we have drawn of the complex is collapsed to the first two coordinates.  The top left smoothing stands for all corners in which $i_0 = -1$ and $i_1 = 0$, and the bottom right smoothing stands for all corners in which $i_0 = 0$ and $i_1 = s_1$.  Note that we are free to pick the left and right endpoints of the first two elementary tangles by Proposition \ref{prop:left_and_right_invariance}.

\begin{proof}[Proof of Proposition \ref{prop:Reidemeister_II_trick}]

Let the set $Y_1 \subset \Ob(\cL^{\Kh}(\wt{D}_{\wt{\bs}}))$ consist of all objects of bigrading $(i_0, i_1) = (-1,s_1)$ in which the small circle is decorated with an $x_+$.  Let the set $X_1 \subset \Ob(\cL^{\Kh}(\wt{D}_{\wt{\bs}}))$ consist of all objects of bigrading $(i_0,i_1) = (0,s_1)$.

Let the set $Y_2 \subset \Ob(\cL^{\Kh}(\wt{D}_{\wt{\bs}}))$ consist of all objects of bigrading $(i_0, i_1) = (-1,0)$ and all objects of bigrading $(i_0,i_1)$ where $i_0 = -1$ and $1 \leq i_1 \leq s_1 - 1$, in which the small circle is decorated with an $x_+$.  Let the set $X_2 \subset \Ob(\cL^{\Kh}(\wt{D}_{\wt{\bs}}))$ consist of all objects of bigrading $(i_0,i_1)$ where $i_0 = -1$ and $1 \leq i_1 \leq s_1$, in which the small circle is decorated with an $x_-$.

Note that for $i=1,2$ there are maps $x_i: Y_i \rightarrow X_i$ determined by the condition that $\M(x_i(y),y)$ consists of a single point for each $y \in Y_i$, and that these maps are bijections.

Now the group generated by the elements of $X_1 \cup Y_1$ forms a subcomplex $C_1$ of the cochain complex.  Since $C_1$ is generated by elements corresponding to objects of the flow category, the full subcategory corresponding to $C_1$ is an upward closed subcategory.  Furthermore, $C_1$ is contractable, as can be seen by successive Gauss elimination of the pairs $(y, x_1(y))$ for each $y \in Y_1$.  Hence the upwards closed subcategory $\cC_1$ generated as a full subcategory by the objects $\Ob(\cL^{\Kh}(\wt{D}_{\wt{\bs}})) \setminus (X_1 \cup Y_1)$ gives rise to the same stable homotopy type as $\cL^{\Kh}(\wt{D}_{\wt{\bs}})$.

Similarly, $X_2 \cup Y_2$ corresponds to a downward closed subcategory of $\cC_1$ with its associated cochain complex being contractible.  Hence the full subcategory $\cC_2$ of $\cL^{\Kh}(\wt{D}_{\wt{\bs}})$ generated by the objects $\Ob(\cL^{\Kh}(\wt{D}_{\wt{\bs}})) \setminus (X_1 \cup Y_1 \cup X_2 \cup Y_2)$ gives rise to the same stable homotopy type as $\cL^{\Kh}(\wt{D}_{\wt{\bs}})$.

Finally, it is clear that $\cC_2$ can be identified with $\cL^{\Kh}(D_\bs)$, and furthermore that the induced framing on $\cC_2$ is a framing as a disc flow cover of the sock flow category $\cS_{\bs}$.
\end{proof}

Theorem \ref{thm:main_knot_theorem} now follows as discussed above.

\subsection{Extended Reidemeister moves for $n>2$}
\label{subsec:ex_rmr}

In this subsection we perform consistency checks for the space associated to $\cL^n(D_\br)$ for $n > 2$.  In particular we show that the space is independent of the decomposition of the diagram $D$ into elementary tangle diagrams, and is invariant under the extended Reidemeister moves as illustrated in Figures \ref{fig:sln_R1} and \ref{fig:sln_R2}.

But first we must verify that the cohomology of the stable homotopy type does indeed return Khovanov-Rozansky cohomology.  Let us start by being specific about the quantum gradings.  We define a quantum grading on $\Ob(\cL^n(D_{\br}))$ for $n \geq 2$ in the case when we have both

\begin{enumerate}
	\item Each of the $m$ 2-strand tangles making up $D_{\br}$ is \emph{not} oriented as a 2-braid,
	\item and $\br = (r_1,\ldots,r_m) \in (2 \Z)^m$.
\end{enumerate}

\noindent (Note that when $D_\br$ is a diagram of a knot, condition (2) implies condition (1).)

In this case consider the object of $\cL^n(D_\br)$ lying in cohomological multidegree $(0,\ldots,0)$ obtained by decorating all the circles in the corresponding resolution by $1$.  We set the quantum grading of this object to be
\[ nR + (1-n)c {\rm .}\]

The quantum gradings of all other objects are determined by the requirements that the differential preserves the quantum degree and that multiplying the decoration of any circle by $x$ increases the quantum grading by $2$.

\begin{definition}
	We shall write $\X^n(D_\br)$ for the stable homotopy type associated to the framed flow category $\cL^n(D_\br)$ by the Cohen-Jones-Segal construction.  Note that the cohomology of $\X^n(D_\br)$ is bigraded by the quantum degree and the underlying cohomological degree of the objects of $\cL^n(D_\br)$.
\end{definition}

\begin{proof}[Proof of Theorem \ref{thm:cohom_of_Xn}]
Krasner's theorem \cite{Krasner} gives a simplified cochain complex for the Khovanov-Rozansky cohomology of a matched diagram.  This is exactly the (bigraded) CW-complex associated to
\[ \X^n(D_\br) [-R] \]
\noindent where the square brackets denote a cohomological shift. 
\end{proof}

Now we move to the consistency checks outlined in the introduction.  The proofs of these will follow similar lines to that of the proof of Proposition \ref{prop:Reidemeister_II_trick}.  More specifically, the arguments proceed by identifying suitable downward closed and upward closed subcategories such that the cohomology of the corresponding subcomplexes or quotient complexes is trivial.  We therefore allow ourselves to be brief with details.

\begin{figure}
	\begin{psfrags}
		\psfrag{+}{$+$}
		\psfrag{-}{$-$}
		\psfrag{pm}{$\pm$}
		\psfrag{mp}{$\mp$}
		\psfrag{ldots}{$\ldots$}
		\begin{center}
			\includegraphics[height=0.6in,width=3.0in]{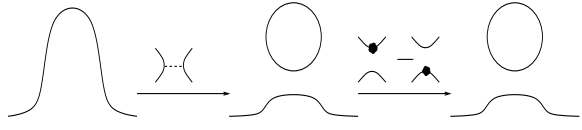}
		\end{center}
	\end{psfrags}
	\caption{The cochain complex associated to the flow category $\cL^n(D_\br)$ where $r_1 = 2$ and the first elementary tangle is capped as in Figure \ref{fig:sln_R1}.}
	\label{fig:fake_R1}
\end{figure}

We shall begin with the extended Reidemeister move RI as illustrated in Figure \ref{fig:sln_R1}.  Let $D_\br$ be a matched diagram containing the tangle on the left-hand side - in particular let us suppose without loss of generality that the two crossings comprise the first elementary tangle in the decomposition of $D_\br$ so that $r_1 = 2$.  Let $D'_{\br'}$ be a matched diagram corresponding to the right-hand side (so that $\br' = (r_2,\ldots,r_m)$).

\begin{proposition}
	We have
	\[\X^n(D_\br)[-2] \simeq \X^n(D'_{\br'}) {\rm .} \]
\end{proposition}

\begin{proof}
Figure \ref{fig:fake_R1} shows the cochain complex associated the flow category $\cL^n(D_\br)$ collapsed to the first coordinate $r_1$ which runs from $0$ to $2$.  The subcomplex generated by those objects in degree $r_1 = 2$ in which the small circle is decorated with an $x^{n-1}$ corresponds to the framed flow category $\cL(D'_{\br'})$.

It is easily seen that the corresponding quotient complex has trivial cohomology.  Indeed one can Gauss eliminate, first pairing all objects in degree $r_1 = 0$ with those objects in degree $r_1 = 1$ in which the small circle is decorated with $x^{n-1}$ and the remaining decorations correspond in the obvious way, and secondly pairing objects in degree $r_1 = 1$ in which the small circle is decorated with a $x^i$ with those in degree $r_1 = 2$ in which the small circle is decorated with and $x^{i+1}$ for $i = 0,1,\ldots, n-2$.
\end{proof}

\begin{figure}
	\begin{psfrags}
		\psfrag{+}{$+$}
		\psfrag{-}{$-$}
		\psfrag{pm}{$\pm$}
		\psfrag{mp}{$\mp$}
		\psfrag{ldots}{$\ldots$}
		\psfrag{i}{$i$}
		\psfrag{n-i}{$j$}
		\psfrag{sum_i}{$\sum\limits_{\stackrel{i,j \geq 0}{i+j = n-1}}$}
		\begin{center}
			\includegraphics[height=3.3in,width=4.9in]{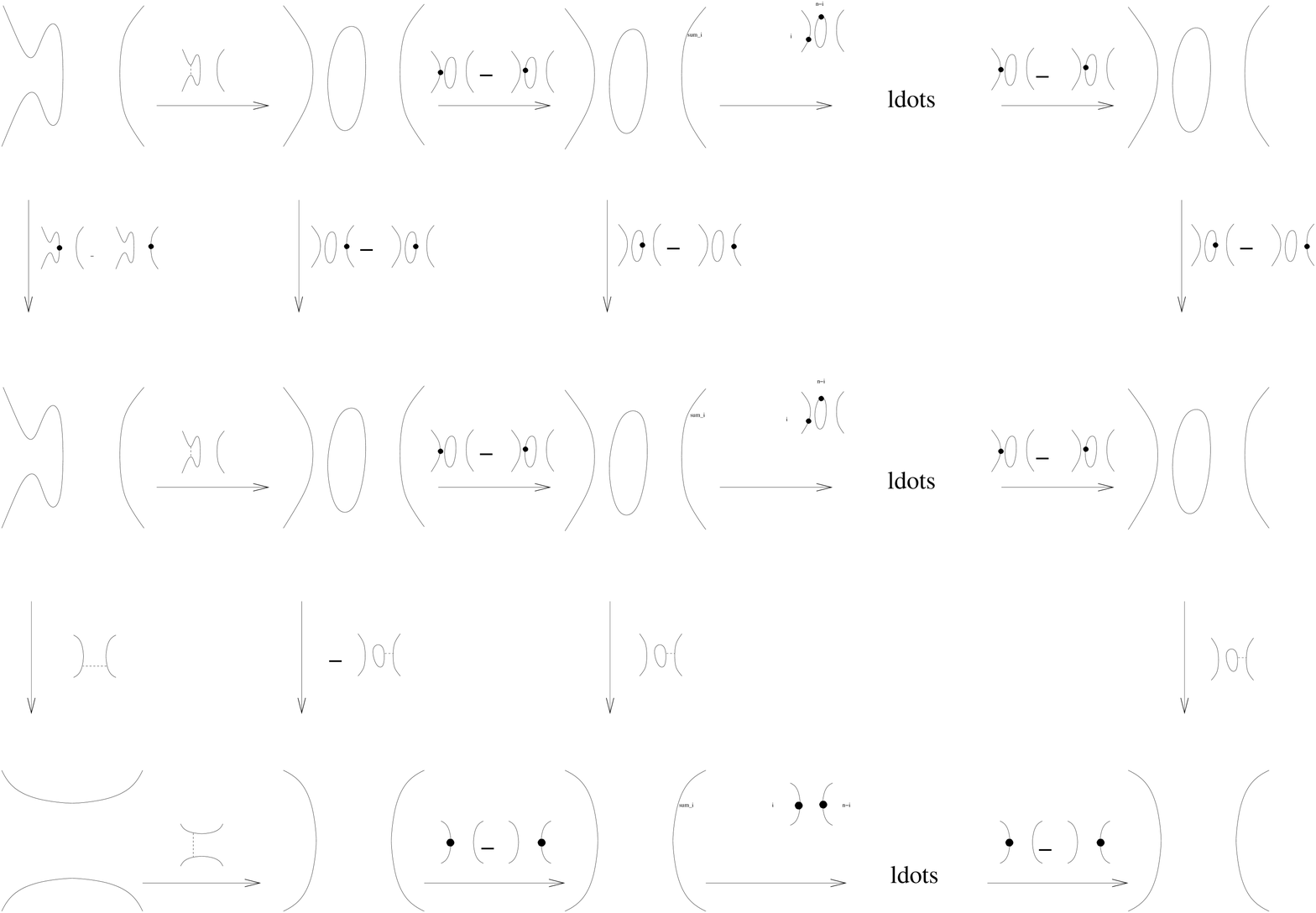}
		\end{center}
	\end{psfrags}
	\caption{The cochain complex associated to the flow category $\cL^{\Kh}(\wt{D}_{\wt{\bs}})$.}
	\label{fig:exzt_r2}
\end{figure}

With this in hand we have verified one half of Proposition \ref{prop:sln_R1+R2}.  We shall verify the other half at the same time as verifying Proposition \ref{prop:matched_independence}.  We use a similar trick to that involved in proving Proposition \ref{prop:all_discs_consequence}.  Namely, we take a matched diagram $D_\bs$ in which $\bs = (s_1 - 2, s_2, \ldots, s_m)$ with $s_1 \geq 2$ and an isotopic diagram $\wt{D}_{\wt{\bs}}$ where $\wt{\bs} = (-2, s_1, \ldots, s_m)$ obtained by performing an extended Reidemeister II move at one end of the first elementary tangle.

\begin{proposition}
	\label{prop:s-2=(s)-(2)}
	We have
	\[ \X(D_\bs) \simeq \X(\wt{D}_{\wt{\bs}}) {\rm .} \]
\end{proposition}

Note that Proposition \ref{prop:s-2=(s)-(2)} implies the remaining half of Proposition \ref{prop:sln_R1+R2} just by taking $s_1 = 2$.

\begin{figure}
	\centerline{
		{
			\psfrag{ldots}{$\ldots$}
			\psfrag{L}{$L$}
			\psfrag{R}{$R$}
			\psfrag{n+1}{$r_1 + 2$}
			\psfrag{n}{$r_1$}
			\psfrag{1}{$2$}
			\psfrag{-1}{$-2$}
			\includegraphics[height=2in,width=3.5in]{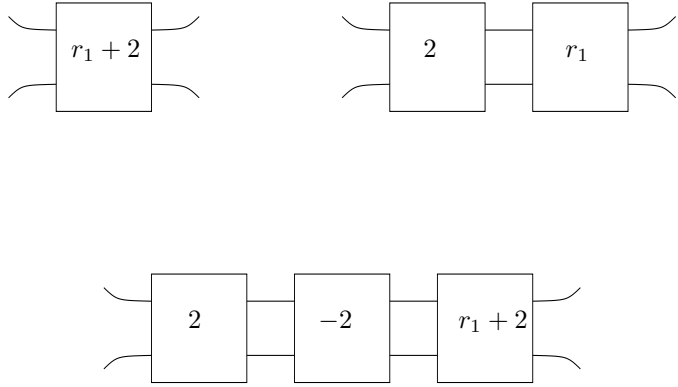}
		}}
		\caption{In this diagram the numbered boxes represent tangle diagrams consisting of the given number of horizontal positive half-twists (negative numbers mean negative half-twists).  We are taking $r_1$ to be a positive even integer.  Clearly all three tangles are isotopic.}
		\label{fig:sln_induction_twist}
	\end{figure}

Furthermore, if three matched decompositions differ locally as shown in Figure \ref{fig:sln_induction_twist}, then Proposition \ref{prop:s-2=(s)-(2)} shows that the bottom decomposition gives rise to the same stable homotopy type as both of the top decompositions.  This then implies Proposition \ref{prop:matched_independence}.

\begin{proof}[Proof of Proposition \ref{prop:s-2=(s)-(2)}]
	In Figure \ref{fig:exzt_r2} we have drawn the cochain complex associated to $\cL^{n}(\wt{D}_{\wt{\bs}})$ collapsed to the first two coordinates of the cohomological multigrading $(i_0, i_1, \ldots, i_m)$ where $-2 \leq i_0 \leq 0$ and $i_j$ lies between $0$ and $s_j$ for each $1 \leq j \leq m$.  In particular, the top left resolution stands for all cochain groups summands with $(i_0,i_1) = (-2, 0)$ while the bottom right consists of those summands with $(i_0,i_1) = (0,s_1)$.  We are free to choose the left and right endpoints of the first two elementary tangles by Proposition \ref{prop:left_and_right_invariance}.
	
	Consider the subcomplex formed by all generators of bigrading $(0,s_1 -1)$, $(0,s_1)$, and $(-1, s_1)$ as well as those generators of bigradings $(-1, s_1 -1)$ and $(-2, s_1)$ in which the small circle is not decorated with $x^{n-1}$ and those generators of bigrading $(-2, s_1 -2)$ in which the small circle is decorated with neither $x^{n-2}$ nor $x^{n-1}$.  It is an exercise to see that this is indeed a subcomplex and that it is contractible.
	
	Taking the quotient by this subcomplex we obtain a homotopy equivalent complex.  In this new complex those generators of bigradings $(0, i_1)$ for $0 \leq i_1 \leq s_1 - 2$ form a subcomplex whose associated quotient is again contractible.
	
	The upshot, when one translates this into statements about upward closed and downwards closed subcategories, is that the full subcategory of $\cL^{n}(\wt{D}_{\wt{\bs}})$ with object set consisting of those objects of bigradings $(0, i_1)$ for $0 \leq i_1 \leq s_1 - 2$ gives rise to the same stable homotopy type as $\cL^{n}(\wt{D}_{\wt{\bs}})$.  Since this full subcategory is isomorphic to $\cL^{n}({D}_{{\bs}})$ (framed as a cover of $\cS_\bs$) we are done.
\end{proof}

\section{The second Steenrod square}
\label{sec:steen_squ}
The construction of a stable homotopy type $\X^n(L)$ enables us to use the tools of stable homotopy theory to study the link $L$. This raises the question of with what extra information these tools may provide us. In \cite{LipSarSq} Lipshitz and Sarkar give a computable description of the second Steenrod square and use this to determine $\KhSpace(L)$ for all links up to 11 crossings. In particular, they observe many instances where the second Steenrod square is non-trivial. Seed \cite{Seed} extended these computations and identified several pairs of knots and links which can be distinguished through their stable homotopy type, but not through Khovanov cohomology.

The description of the second Steenrod square in \cite{LipSarSq} is somewhat optimized for the Khovanov flow category defined in \cite{LipSarKhov}, but their constructions can be easily modified to describe the second Steenrod square for any framed flow category. In the following subsection we recall this description and extend it to a framed flow category.

\subsection{The second Steenrod square of a framed flow category}
\label{subsec:steen_fl_cat}
Let $M_n=M(\Z/2,m)$ be the $m$-th Moore space for $\Z/2$, obtained from $S^m$ by attaching one $(m+1)$-cell using an attaching map of degree $2$. We can think of this space as the $(m+1)$-skeleton of an Eilenberg-MacLane space $K(\Z/2,m)$. We also assume that $m$ is at least $3$.
It is shown in \cite{LipSarSq} that $\pi_{m+1}(M_m)\cong \Z/2$, and that the inclusion $S^m\subset M_m$ induces an isomorphism on $\pi_{m+1}$. So in order to construct a $K(\Z/2,m)$ from $M_m$ we need to attach one $(m+2)$-cell $\tau$ to kill $\pi_{m+1}(M_m)$, and then cells of dimension $\geq m+3$. By a Theorem of Serre, see for example \cite{mccleary}, $H^{m+2}(K(\Z/2,m);\Z/2)\cong \Z/2$ generated by $\Sq^2(\iota)$, where $\iota\in H^m(K(\Z/2,m);\Z/2)$ is the generator.

Let $X$ be a CW-complex and $u\in H^m(X;\Z/2)$. Let $f\colon X\to K(\Z/2,m)$ be a cellular map with $f^\ast\iota=u$. Then $\Sq^2(u)=f^\ast\Sq^2\iota$. We want to give a cocycle representative for $\Sq^2 u$, so we need to determine $f^\ast \tau'$, where $\tau'\in C^{m+2}(K(\Z/2,m);\Z/2)$ is (necessarily) the cocycle representing $\Sq^2\iota$ and the dual of $\tau$.

Let $\sigma$ be an $(m+2)$-cell in $X$ with attaching map $\chi_\sigma\colon S^{m+1}\to X^{(m+1)}$. Note that we have a restriction $f|\colon X^{(m+1)}\to M_m$. The composition $f|\circ \chi_\sigma$ represents an element $\pi_{m+1}(M_m)\cong \Z/2$, and by the construction of $f$ this element also represents the evaluation of the cocycle $f^\ast\tau'$ on $\sigma$.

So a cocycle representation $\sq^2 u\in C^{m+2}(X;\Z/2)$ for $\Sq^2 u$ is determined by the elements $[f|\circ \chi_\sigma]\in \pi_{m+1}(M_m)$.

Now let $(\Cat,\imath,\varphi)$ be a framed flow category, and $u\in H^i(\Cat;\Z/2)$ a cohomology class. Because of \cite[Prop.3.6]{LipSarSq} we can assume that the grading has only values in $\{i,i+1,i+2\}$. The neat embedding $\imath$ is relative to some $\mathbf{d}=(d_i,d_{i+1})$. Write $m=d_i+d_{i+1}$, so that $H^i(\Cat,\Z/2)\cong \tilde{H}^m(|\Cat|;\Z/2)$.

Let $c\in C^i(\Cat;\Z/2)$ be a cocycle representing $u$. Then 
\[
c=\sum_{x\in \Ob_i(\Cat)}n_xx
\]
for some $n_x\in \Z/2$, where $\Ob_i(\Cat)$ are those objects $x$ with $|x|=i$. Now $|\Cat|^m$ is a wedge of $m$-spheres, one for each object $x$ with $|x|=i$, and we define a map $f^m\colon |\Cat|^m\to M_m$ so that $f^m|\mathcal{C}(x)$ is a degree one map onto the $m$-skeleton of $M_m$ if $n_x=1$ or the constant map to the basepoint if $n_x=0$.

Now let $y\in \Ob(\Cat)$ satisfy $|y|=i+1$. We have an inclusion
\[
\imath_y\colon \coprod_{x\in \Ob_i(\Cat)}\mathcal{M}(y,x)\times [-\varepsilon,\varepsilon]^{d_i}\to \R^{d_i}
\]
Since $c$ is a cocycle, the set $\coprod \mathcal{M}(y,x)$ has an even number of elements for each $y$, so let $n(y)$ be the number so that $2n(y)$ is the cardinality of this set.

\begin{definition}
A \em topological boundary matching for $c$ \em consists of the following data for each object $y$ with $|y|=i+1$: A collection of disjoint, embedded, framed arcs
\[
\eta^j_y\colon [0,1]\times [-\varepsilon,\varepsilon]^{d_i} \to [0,\infty)\times \R^{d_i}
\]
for $j=1,\ldots,n(y)$, such that
\[
(\eta^j_y)^{-1}(\{0\}\times \R^{d_i}) = \{0,1\}\times [-\varepsilon,\varepsilon]^{d_i}
\]
and 
\[
\bigcup_j \eta^j_y(\{0,1\}\times \{0\}) = \{0\}\times \imath_y\left(\coprod_x\mathcal{M}(y,x)\right)
\]
For each arc $\eta^j_y$ there exist $A_0,A_1\in \coprod_x\mathcal{M}(y,x)$ with $\eta_y^j(0,0)=(0,\imath_y(A_0,0))$ and $\eta_y^j(1,0)=(0,\imath_y(A_1,0))$. If the framings of $A_0$ and $A_1$ have a different sign, we assume that the framing of $\eta^j_y|\{0,1\}$ agrees with the framings of $A_0$ and $A_1$, in which case we call the arc \em boundary-coherent\em.

If the framings of $A_0$ and $A_1$ have the same sign, we assume that the framing of $\eta^j_y$ agrees with the corresponding framing at one endpoint, and at the other endpoint it agrees with the composition of the corresponding framing and the reflection $R\colon \R^{d_i}\to \R^{d_i}$ given by $R(x_1,\ldots,x_{d_i})=(-x_1,x_2,\ldots,x_{d_i})$. In this case the arc is called \em boundary-incoherent \em and we orient it from the endpoint at which the framings agree to the endpoint at which the framings disagree.
\end{definition}

For each object $y$ with $|y|=i+1$ we now get an embedding
\[
\eta_y\colon \coprod_{j=1}^{n(y)}[0,1] \times [-\varepsilon,\varepsilon]^{d_i+d_{i+1}}\to \mathcal{C}(y)
\]
given by
\[
\eta_y(t,x_1,\ldots,x_{d_i},y_1,\ldots,y_{d_{i+1}})=(\eta^j_y(t,x_1,\ldots,x_{d_i}),y_1,\ldots,y_{d_{i+1}})
\]
when restricted to the $j$-th copy of the disjoint union.

We can extend $f^m$ to $f^{m+1}\colon |\Cat|^{m+1}\to M_m$ as follows. Away from the image of $\eta_y$, the cell $\mathcal{C}(y)$ is send to the basepoint. On the image of a boundary-coherent arc we project
\[
[0,1]\times [-\varepsilon,\varepsilon]^{d_i+d_{i+1}} \to [-\varepsilon,\varepsilon]^{d_i+d_{i+1}} \to S^{d_i+d_{i+1}}=M_m^m.
\]
On the image of a boundary-incoherent arc we map $[0,1]\times [-\varepsilon,\varepsilon]^{d_i+d_{i+1}}$ directly over the $(m+1)$-cell of $M_m$ with degree 1 as in \cite[\S 3.4]{LipSarSq}.

We can now obtain a cocycle representing $\Sq^2(u)$ as described above for a CW-complex $X$. In order to get a computable description, we need to be more careful with the framings, compare \cite[\S 3.2]{LipSarSq}.

To begin with, we assume that every framing of a point in a $0$-dimensional moduli space is either given by the standard basis of $\R^{d_j}$, or by the basis we obtain by reflecting the standard basis in the first coordinate. This is done by choosing a 1-parameter family of framings for the $0$-dimensional moduli spaces, which then extends to a 1-parameter family of framings of flow categories from the original category to a slightly deformed one, using collar neighborhoods and the techniques of \cite[\S 3]{LipSarKhov}.

A framing of a path embedded in $\R^{m+1}$ gives rise to a path of $m$ orthonormal vectors in $\R^{m+1}$. Given $m$ orthonormal vectors in $\R^{m+1}$ there is a unique unit vector in $\R^{m+1}$ which will turn this collection into an element of $\SO(m+1)$. Therefore a framed path in $\R^{m+1}$ gives rise to a path in $\SO(m+1)$. Furthermore, isotopy classes of framed paths in $\R^{m+1}$ with fixed endpoints are in one-to-one correspondence with homotopy classes of paths in $\SO(m+1)$ with fixed endpoints, provided $m\geq 3$.

Moduli spaces $\mathcal{M}(z,x)$ where $|z|=i+2=|x|+2$ are disjoint unions of compact intervals and circles, and are framed embedded in $\R^{d_i}\times [0,\infty)\times \R^{d_{i+1}}$ with endpoints in $\R^{d_i}\times \{0\}\times \R^{d_{i+1}}$. The possible framings at the endpoints are therefore given by
\begin{align*}
&(e_1,\ldots,e_{d_i},e_{d_i+1},\ldots,e_{e_i+e_{i+1}}), (-e_1,e_2,\ldots,e_{d_i},e_{d_i+1},\ldots,e_{e_i+e_{i+1}}),\\ &(e_1,\ldots,e_{d_i},-e_{d_i+1},e_{d_i+2}\ldots,e_{e_i+e_{i+1}}), (-e_1,e_2,\ldots,e_{d_i},-e_{d_i+1},e_{d_i+2}\ldots,e_{e_i+e_{i+1}})
\end{align*}
which we denote by $\PP$, $\MP$, $\PM$, $\MM$, respectively.

\begin{definition}
A \em coherent system of paths joining $\PP$, $\MP$, $\PM$, $\MM$ \em is a choice of path $\overline{\varphi_1\varphi_2}$ in $\SO(m+1)$ from $\varphi_1$ to $\varphi_2$ for each pair of frames $\varphi_1,\varphi_2\in \{\PP,\MP,\PM,\MM\}$ satisfying the following cocycle conditions:
\begin{enumerate}
	\item For all $\varphi\in \{\PP,\MP,\PM,\MM\}$ the loop $\overline{\varphi\varphi}$ is null-homotopic;
	\item For all $\varphi_1,\varphi_2,\varphi_3\in \{\PP,\MP,\PM,\MM\}$ the path $\overline{\varphi_1\varphi_2}\cdot \overline{\varphi_2\varphi_3}$ is homotopic to $\overline{\varphi_1,\varphi_3}$ relative to the endpoints.
\end{enumerate}
\end{definition}

Coherent systems of paths exist, we will use the one described in \cite[Lm.3.1]{LipSarSq}. To describe it, we will refer to the first coordinate of $\R^{d_i}$ as the $e_1$-coordinate, to the first coordinate of $\R^{d_{i+1}}$ as the $e_2$ coordinate, and to the coordinate of $[0,\infty)$ as the $\bar{e}$-coordinate.

For $\varphi_1,\varphi_2\in \{\PP,\MP,\PM,\MM\}$ define $\overline{\varphi_1\varphi_2}$ as follows:
\begin{enumerate}
	\renewcommand\theenumi{\roman{enumi}}
	\item $\overline{\PP\MP}, \overline{\MP\PP}, \overline{\PM\MM}, \overline{\MM\PM}$: Rotate $180^\circ$ around the $e_2$-axis, such that the first vector equals $\bar{e}$ halfway through.
	\item $\overline{\PP\PM}, \overline{\PM\PP}$: Rotate $180^\circ$ around the $e_1$-axis, such that the second vector equals $\bar{e}$ halfway through.
	\item $\overline{\MP\MM}, \overline{\MM\MP}$: Rotate $180^\circ$ around the $e_1$-axis, such that the second vector equals $-\bar{e}$ halfway through.
	\item $\overline{\PP\MM}, \overline{\MM\PP}, \overline{\MP\PM}, \overline{\PM\MP}$: Rotate $180^\circ$ around the $\bar{e}$-axis, such that the second vector equals $-e_1$ halfway through.
\end{enumerate}

A framed path in $\R^{d_i}\times[0,\infty)\times \R^{d_{i+1}}$ with endpoints in $\R^{d_i}\times\{0\}\times \R^{d_{i+1}}$ given by two elements of $\{\PP,\MP,\PM,\MM\}$ is called a \em standard frame path\em, if its homotopy class in $\SO(m+1)$ is one of the classes described in (i)-(iv) relative endpoints, and a \em non-standard frame path\em, if it is not.

If $\Cat$ is a framed flow category such that all $0$-dimensional moduli spaces are framed using the standard $+$ or $-$ frame, each interval in a $1$-dimensional moduli space is either a standard or a non-standard frame path. Furthermore, each circle component of a $1$-dimensional moduli space represents an element of $\pi_1(\SO(m+1))\cong H_1(\SO(m+1))\cong \Z/2$. The framing of a $1$-dimensional moduli space $\mathcal{M}(z,x)$ is therefore encoded in a function $$fr\colon \pi_0(\mathcal{M}(z,x)) \to \Z/2$$ with $0$ corresponding to the trivial element of $\pi_1(\SO(m+1))$ and the standard frame path.

\begin{definition}
The framed components in a $1$-dimensional moduli space which are intervals, are called \em Pontryagin-Thom arcs\em, and the framed components which are circles are called \em Pontryagin-Thom circles\em.
\end{definition}

Now given a cohomology class $u\in H^i(\Cat;\Z/2)$, recall that we have constructed a map $f^{m+1}\colon |\Cat|^{(m+1)} \to M_m$ with $f^{m+1\ast}\iota=u$. For each $z\in \Ob(\Cat)$ with $|z|=i+2$ we can compose the attaching map $\chi_z\colon S^{m+1}\to |\Cat|^{(m+1)}$ of $\mathcal{C}(z)$ with $f^{m+1}$, and consider the resulting element $\pi_{m+1}(M_m)\cong \Z/2$. Furthermore, the inclusion $S^m\subset M_m$ induces an isomorphism on $\pi_{m+1}$. Let us denote the composition $f^{m+1}\circ \chi_z=f_z\colon S^{m+1}\to M_m$, and recall that we have
$$S^{m+1}\cong \partial \mathcal{C}(z)=\partial ([0,R]\times [-R,R]^{d_i}\times [0,R]\times [-R,R]^{d_{i+1}}).$$
For each $y\in \Ob(\Cat)$ with $|y|=i+1$ we have a set
$$C_y(z)=[0,R]\times [-R,R]^{d_i}\times \{0\}\times \mathcal{M}(z,y)\times[-\varepsilon,\varepsilon]^{d_{i+1}}\subset \partial \mathcal{C}(z)$$
which contains framed arcs $\gamma_j\times [-\varepsilon,\varepsilon]^{d_i}\times\{0\}\times \mathcal{M}(z,y)\times [-\varepsilon,\varepsilon]^{d_{i+1}}$ coming from the topological matchings. Together with the Pontryagin-Thom arcs, which also sit in $\partial \mathcal{C}(z)$ via $C_x(z)$ for $x$ in the support of the cocycle representing $u$, these combine to framed circles in $\partial\mathcal{C}(z)$. For the correct smoothing of these framed circles, see \cite[\S 3.4]{LipSarSq}. The Pontryagin-Thom circles in $\mathcal{M}(z,x)$ also give framed circles in $\partial\mathcal{C}(z)$.

Away from these framed circles, the map $f_z$ is the constant map to the basepoint. Also, if none of the arcs from the topological matching is boundary-incoherent, the image of $f_z$ is already contained in the $m$-skeleton of $M_m$, that is, in an $m$-sphere.

If a framed circle contains boundary-incoherent arcs, it contains an even number of them, compare \cite[Lm.3.9]{LipSarSq}, and we can use the same method as in \cite{LipSarSq} to alter the map $f_z$ to a map $f_z'\colon S^{m+1}\to S^m\subset M_m$ representing $f_z$ in $\pi_{m+1}(M_m)$. We will not repeat this construction here, but note that the proof of \cite[Prop.3.10]{LipSarSq} does not rely on the special form of flow categories considered there. We note that in order to use their construction, we should ensure that the topological matchings lead to standard framed arcs in $[0,\infty)\times \R^m$. This can be arranged, because the arcs in the framed circles above which come from the topological matching are always between points $(A,C)$ and $(B,C)$ with $A,B\in \mathcal{M}(y,x)$ and $C\in\mathcal{M}(z,y)$.

Overall, we obtain a finite number of framed circles $(K,\Phi)$ embedded in $S^{m+1}\cong \partial\mathcal{C}(z)$ which determine the value of the cocycle $c_f$ representing $\Sq^2(u)$. The value of each $(K,\Phi)$ as an element of the first framed bordism group $\Omega^{fr}_1\cong \Z/2$ can be read off from the following Proposition.

\begin{proposition}
\label{prop:add_form_frame}
Let $(K,\Phi)$ be a framed circle arising from the above. If $K$ is built from Pontryagin-Thom arcs and arcs from topological matchings, its value in $\Omega^{fr}_1\cong \Z/2$ is given as the sum of
\begin{enumerate}
	\item The number $1$.
	\item The number of Pontryagin-Thom arcs in $K$ with the non-standard framing.
	\item The number of arrows on $K$ which point in a fixed direction.
\end{enumerate}
If $K$ consists of a Pontryagin-Thom circle, its value in $\Omega^{fr}_1\cong \Z/2$ is given as the sum of
\begin{enumerate}
	\item The number $1$.
	\item The value of the map $fr\colon \pi_0(\mathcal{M}(z,x))\to \Z/2$ on the component corresponding to $K$.
\end{enumerate}
\end{proposition}

The value of the cocycle $c_f\in C^{i+2}(\Cat;\Z/2)$ on $z$ is then the sum of these numbers over all such framed circles $(K,\Phi)$.

\begin{proof}
In the case of $K$ being build from arcs, this follows directly from \cite[Prop.3.12]{LipSarSq}. If $K$ is a Pontryagin-Thom circle, then $fr([K])$ is determined by the element of $\pi_1(\SO(m+1))$ the framed circle represents. But the constant loop in $\SO(m+1)$ corresponds to the non-trivial element of $\Omega_1^{fr}$ which explains the extra summand $1$.
\end{proof}

\begin{example}
Consider the Morse function $f\colon\RP^n\to \R$ from Section \ref{subsec:the_n_sock} (disguised there as a Lens space), and let $\Cat_i$ be the framed flow category obtained by restricting to the critical points $p_i$, $p_{i+1}$ and $p_{i+2}$. Write $\{L_i,R_i\}=\mathcal{M}(p_{i+1},p_i)$ and $\{L_{i+1},R_{i+1}\}=\mathcal{M}(p_{i+2},p_{p+1})$. Then $\mathcal{M}(p_{i+2},p_i)$ consists of two intervals. From the description in Section \ref{subsec:the_n_sock} we know that one interval, say $I_1$, bounds  $\{(L_i,R_{i+1}),(R_i,L_{i+1})\}$, and the other interval $I_2$ bounds $\{(L_i,L_{i+1}),(R_i,R_{i+1})\}$.

\begin{figure}[h]
	\centerline{
			\psfrag{eta1}{$\eta\times\{L_{i+1}\}$}
			\psfrag{eta2}{$\eta\times\{R_{i+1}\}$}
			\psfrag{LL}{$(L_i,L_{i+1})$}
			\psfrag{LR}{$(L_i,R_{i+1})$}
			\psfrag{RL}{$(R_i,L_{i+1})$}
			\psfrag{RR}{$(R_i,R_{i+1})$}
			\psfrag{I1}{$I_1$}
			\psfrag{I2}{$I_2$}
			\includegraphics[height=1in,width=2in]{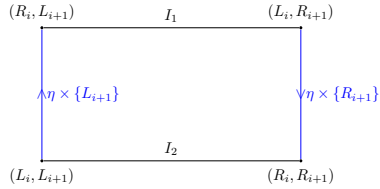}
		}
		\caption{The framed circle for the generator in degree $i+2$. Note that $i$ needs to be odd for the two vertical arcs to be oriented.}
		\label{fig_proj_space}
\end{figure}

For the topological matching, we choose an arc $\eta$ between $L_i$ and $R_i$. This arc is boundary-coherent if and only if $i$ is even. If $i$ is odd, we choose an orientation, say from $L_i$ to $R_i$. In the boundary of the sphere for the object $p_{i+2}$ we now get two arcs coming from the topological matching, namely $\eta\times \{L_{i+1}\}$ and $\eta\times \{R_{i+1}\}$. Note that if $i$ is odd, both of these arcs are oriented in the same direction, so they do not contribute according to Proposition \ref{prop:add_form_frame}.
The corresponding circle is indicated in Figure \ref{fig_proj_space}. The Steenrod square $\Sq^2(z^i)$ of the generator $z^i\in H^i(\RP^n;\Z/2)$ can be calculated using the Cartan formula, so we see that for $i=2j$ or $i=2j+1$ we get that the intervals $I_1$ and $I_2$ have to have the same framings for $j$ odd (in order to get a non-trivial Steenrod square), and different framings for $j$ even (in order to get the trivial Steenrod square). We will not attempt to work out the exact framing for the intervals. Furthermore, they do depend on the choice of rotation of the point in $\mathcal{M}(p_{i+1},p_i)$ which is framed by $-e_1,\ldots,-e_i$ rather than by the standard framing. Note that changing this rotation will always affect both intervals.
\end{example}

\begin{example}
If we use the $\CP^n$ version of this Morse function, we get the same kind of critical points $p_i$, but this time each $\mathcal{M}(p_{i+1},p_i)$ is a circle. To understand the framing, write $p_i=[\cdots:1:0:\cdots]$, $p_{i+1}=[\cdots:0:1:\cdots]$ and $$\mathcal{M}(p_{i+1},p_i)=\{[\cdots:z:1:\cdots]\,|\,z\in S^1\},$$ which we can think of as embedded in $W^s(p_{i+1})\cong \R^{2i+2}$. The first $2i$ frames come from the standard basis $\frac{\partial}{\partial x_0},\frac{\partial}{\partial y_0},\ldots,\frac{\partial}{\partial x_{i-1}},\frac{\partial}{\partial y_{i-1}}$ at $p_i$. The point $[\cdots:z:1:\cdots]$ is identified with $[\cdots:1:\bar{z}:\cdots]$ which has this standard framing.
Therefore the framing at $[\cdots:z:1:\cdots]$ is the standard basis multiplied with $z\in S^1$. The final frame comes from the flow direction. Together with the tangent vector, we see that the corresponding element of $\pi_1(\SO(2i+2))$ is represented by the loop of unitary matrices
\begin{equation*}
z \mapsto \begin{pmatrix}z & & \\ {} & \ddots & \\ {} & & z \end{pmatrix} \in \UU(i+1)
\end{equation*}
This represents the trivial element if and only if $i$ is odd. If we denote the generator of $H^{2}(\CP^n;\Z/2)$ by $u$, we therefore get from Proposition \ref{prop:add_form_frame} that $\Sq^2(u^i)=u^{i+1}$ for $i$ odd, and $\Sq^2(u^i)=0$ for $i$ even. Again this could be derived from the Cartan formula.
\end{example}

\subsection{A particular frame assignment for the sock}
\label{subsec_part_n_frames}
In order to use the previous section for calculations of Steenrod squares of links, we need to understand the framings of $1$-dimensional moduli spaces for the category $\cS^n_\mathbf{r}$, where $\mathbf{r}\in (\Z-\{0\})^m$.

Recall the CW-complex $\mathcal{C}(\wt{\cS^n_\mathbf{r}})$ for $\mathbf{r}=(r_1,\ldots,r_m)\in (\Z-\{0\})^m$ from Section \ref{subsec:obstruction_theory}. It is the product
\begin{equation*}
 \mathcal{C}(\wt{\cS^n_\mathbf{r}}) = \mathcal{C}(\wt{\cS^n_{r_1}}) \times \cdots \times \mathcal{C}(\wt{\cS^n_{r_m}}),
\end{equation*}
so in particular its cells are products of cells from the $\mathcal{C}(\wt{\cS^n_{r_i}})$.

The cells of $\mathcal{C}(\wt{\cS^n_r})$ for $r>0$ have already been described in the proof of Lemma \ref{lem:sockobscomptriv}. For the reader's convenience, we will now describe it for $r<0$. The $0$-cells are given by the objects $r,r+1,\ldots,-1,0$, and the $1$-cells correspond to elements of the $0$-dimensional moduli spaces, so we have a $1$-cell $P$ between $-1$ and $0$, $1$-cells $P_i$, $M_i$ between $2i$ and $2i+1$ for $i=\lfloor r/2\rfloor, \ldots, -1$, and $1$-cells $P_{i,0},\ldots, P_{i,n-1}$ between $2i-1$ and $2i$ for $i=\lfloor (r-1)/2\rfloor,\ldots,-1$.

We have the $2$-cell $N_0$ with
\[
 \partial N_0 = P_{-1}-M_{-1}
\]
and $2$-cells $\tilde{N}_{i,0},\ldots, \tilde{N}_{i,n-2}$ for $i=\lfloor (r-1)/2 \rfloor,\ldots,-1$ with
\[
 \partial \tilde{N}_{i,j} = P_{i,j}+P_{i}-M_{i}-P_{i,j+1},
\]
and finally $2$-cells $N_{i,0},\ldots, N_{i,n-2}$ for $i=\lfloor (r-1)/2,\ldots,-1 \rfloor$ with
\[
 \partial N_{i,j} = P_{i-1} + P_{i,j}-P_{i,j+1} - M_{i-1}.
\]
The $3$-cells are given by $Q_{i,0},\ldots, Q_{i,n-3}$ for $i=\lfloor (r-1)/2\rfloor,\ldots,-1$ with
\[
 \partial Q_{i,j} = N_{i,j}+\tilde{N}_{i,j+1}- \tilde{N}_{i,j} - N_{i,j+1}.
\]

To simplify our notation, we will write $a=(a_1,\ldots,a_m)\in \Z^m$ for $0$-cells in the product complex $\mathcal{C}(\wt{\cS^n_\mathbf{r}})$. For $1$-cells we will suppress the second index in $P_{i,j}$ as it will have no effect on the framing value. We then use the notation
\[
 P^j_a=a_1\times \cdots \times a_{j-1} \times P_{a_j} \times a_{j+1} \times \cdots \times a_k,
\]
and similarly for $M^j_a$. The $2$-cells which are products of two $1$-cells (and $0$-cells) are denoted by
\[
 (P^j_a,P^i_a),\, (P^j_a,M^i_a), \, (M^j_a,P^i_a), \,(M^j_a,M^i_a)
\]
where $1\leq j<i \leq k$, where, for example,
\[
 (P^j_a,P^i_a) = a_1 \times \cdots \times a_{j-1} \times P_{a_j} \times a_{j+1} \times \cdots \times a_{i-1} \times P_{a_i} \times a_{i+1} \times \cdots \times a_m.
\]
Other $2$-cells are denoted by $N^j_a$ and $\tilde{N}^j_a$, where $j$ indicates the coordinate of the non-$0$-cell.

We extend this notation to higher dimensional cells in the obvious way. The subscript $a$ indicates that the cell corresponds to a component of a moduli space $\mathcal{M}(b,a)$ for some object $b$. We also say that the cell is \em based at \em $a$.

Given $\mathbf{r}=(r_1,\ldots,r_m)\in (\Z-\{0\})^m$, we want to have a standard sign assignment for the elements of $0$-dimensional moduli spaces. For $m=1$, this is essentially encoded in whether a point is called $P$ or $M$. The resulting sign assignment below is then just obtained by using the usual product convention.
\begin{definition}
\label{def:signassign}
 The \em standard sign assignment \em $s\in C^1(\mathcal{C}(\wt{\cS^n_\mathbf{r}}), \Z/2)$ is the $1$-cochain defined by
\begin{align*}
 s(P^j_a) &= a_1+\cdots + a_{j-1} \\
 s(M^j_a) &= a_1+\cdots + a_{j-1} + 1.
\end{align*}
These sums are in $\Z/2$, with empty sums treated as $0$.
\end{definition}

The standard sign assignment controls the framings of the $0$-dimensional moduli spaces of $\cS_\mathbf{r}$, with $0$ corresponding to a positive framing, and $1$ corresponding to a negative framing.

When framing the $1$-dimensional moduli spaces, we need to accommodate the fact that some of the $r_j$ can be negative. Basically the difference is that $M_i$ cells are based at an even number for negative $r$, and based at an odd number for positive $r$. For this reason, define $\delta=(\delta_1,\ldots,\delta_m)\in \Z^m$ by $\delta_j=0$ if $r_j>0$ and $\delta_j=1$ if $r_j<0$.

\begin{definition}
\label{def:frameassign}
 The \em standard frame assignment \em $f\in C^2(\mathcal{C}(\wt{\cS^n_\mathbf{r}}), \Z/2)$ is defined as
\begin{align*}
 f(P_a^j,P_a^i) &= (a_1+\cdots+a_{j-1})(a_j+\cdots + a_{i-1})\\
 f(P_a^j,M_a^i) &= (a_1+\cdots+a_{j-1})(a_j+\cdots + a_{i-1} + 1)\\
 f(M_a^j,P_a^i) &= (a_1+\cdots+a_{j-1} + 1)(\delta_j + a_{j+1}+\cdots + a_{i-1})\\
 f(M_a^j,M_a^i) &= (a_1+\cdots+a_{j-1} + 1)(\delta_j + a_{j+1}+\cdots + a_{i-1} + 1)\\
 f(N_a^j) &= a_1+\cdots +a_{j-1}\\
 f(\tilde{N}_a^j) &= 0.
\end{align*}
Again these sums are in $\Z/2$.
\end{definition}

The distinction between cells $N$ and $\tilde{N}$ seems random, but the difference is that $N$ is based at an object at which also $M$ is based. We can unify the frame formula using a factor $(a_j+\delta_j)$, but we still have to distinguish the two cases in proofs below.

The standard sign and frame assignments provide us with a framing of the $0$- and $1$-dimensional moduli spaces of $\wt{\cS^n_{\mathbf{r}}}$. Also note that if $|r_i|=1$ for all $i=1,\ldots,m$, we only need $f(P_a^j,P_a^i)$, and the sign and frame assignments agree with the assignments of the cube in \cite{LipSarSq}. We now need the analogue of \cite[Lm.3.5]{LipSarSq}, namely that we can extend this framing to the higher dimensional moduli spaces.

Let $\tau$ be a 3-dimensional cell in $\mathcal{C}(\wt{\cS^n_\mathbf{r}})$. Then $\tau$ is of the form $(E^k_a,E^j_a,E^i_a)$ with the $E$'s corresponding to $1$-cells, that is, chosen from $P$ or $M$, $(E^j_a,F^i_a)$ or $(F^j_a,E^i_a)$ with $F$ chosen from $N$ or $\tilde{N}$, or $Q^j_a$. Furthermore, $\tau$ corresponds to the component of a $2$-dimensional moduli space $\mathcal{M}(b,a)$ where $b$ is an object with $|b|=|a|+3$. This component is a hexagon, except when $\tau= Q^j_a$, in which case it is a square.

The following lemma is the analogue of \cite[Lm.2.1]{LipSarSq}.

\begin{lemma}
\label{lem:coboundary}
 Let $\tau$ be a $3$-cell in $\mathcal{C}(\wt{\cS^n_{\mathbf{r}}})$ based at $a$.
\begin{enumerate}
 \item If $\tau=(E^k_a,E^j_a,E^i_a)$ with the $E$'s some combination of $P$'s and $M$'s, then
\[
 \delta f(\tau) = s(E_a^k)+s(E_a^j)+s(E_a^i).
\]
 \item If $\tau=(E^j_a, N^i_a)$ with $E$ either $P$ or $M$, then
\[
 \delta f(\tau) = s(E^j_a)+s(P^i_a)+s(M^i_a).
\]
\item If $\tau=(N^j_a,E^i_a)$ with $E$ either $P$ or $M$, then
\[
 \delta f(\tau) = s(P^j_a)+s(M^j_a)+s(E^i_a).
\]
\item If $\tau = (E^j_a,\tilde{N}^i_a)$ or $\tau = (\tilde{N}^i_a,E^j_a)$ with $E$ either $P$ or $M$, then
\[
 \delta f(\tau) = s(E^j_a).
\]
\item If $\tau = Q^j_a$, then $\delta f(\tau)=0$.
\end{enumerate}

\end{lemma}

\begin{proof}
 Let $\tau=(E^k_a,E^j_a,E^i_a)$. If we look at the boundary of $\tau$, we get six $2$-cells, three of which are based at $a$, and the others are based at $a^k$, $a^j$ and $a^i$, where $a^k$ agrees with $a$, except in the $k$-coordinate, where the entry is $a_k+1$, and similar with $a^j$ and $a^i$. Hence the boundary of $\tau$ with $\Z/2$ coefficients is
\[
 \partial \tau = (E^j_{a^k},E^i_{a^k}) + (E^j_a,E^i_a) + (E^k_{a^j},E^i_{a^j}) + (E^k_a,E^i_a) + (E^k_{a^i},E^j_{a^i}) + (E^k_a,E^j_a).
\]
Write $\varepsilon_m=1$ if $E_a^k=L_a^k$ and $\varepsilon_k=0$ if $E_a^k=R_a^k$, and similarly define $\varepsilon_j$ and $\varepsilon_i$. We then can write
\begin{equation*}
 f(E^j_a,E^i_a)=(a_1+\cdots+a_{j-1}+\varepsilon_j)(\varepsilon_j(\delta_j+a_j)+a_j+\cdots a_{i-1}+\varepsilon_i).
\end{equation*}
Observe that if $\varepsilon_j = 1$, then $(\delta_j+a_j)$ is odd, so we can replace $\varepsilon_j(\delta_j+a_j)$ with $\varepsilon_j$ in $\Z/2$.

Now
\begin{equation*}
 f(E^j_{a^k},E^i_{a^k})  = (a_1+\cdots+a_{j-1}+\varepsilon_j+1)(\varepsilon_j+a_j+\cdots a_{i-1}+\varepsilon_i)
\end{equation*}
so that
\begin{align*}
 f(E^j_{a^k},E^i_{a^k}) + f(E^j_a,E^i_a) &= \varepsilon_j+a_j+\cdots a_{i-1}+\varepsilon_ia'_i\\
 &= s(E^j_a)+s(E^i_a).
\end{align*}
Similarly
\begin{equation*}
 f(E^k_{a^j},E^i_{a^j}) = (a_1+\cdots+a_{k-1}+\varepsilon_k)(\varepsilon_k+a_k+\cdots a_{i-1}+\varepsilon_i+1)
\end{equation*}
so that
\begin{equation*}
 f(E^k_{a^j},E^i_{a^j}) + f(E^k_a,E^i_a) = a_1+\cdots +a_{k-1}+\varepsilon_k = s(E^k_a).
\end{equation*}
As
\[
 f(E^k_{a^i},E^j_{a^i}) + (E^k_a,E^j_a) = 0,
\]
we get the result.

If $\tau=(E^j_a, N^i_a)$, then\footnote{It seems that the cell $(E^j_{a^i},P^i_{a^i})$ appears twice in $\partial\tau$, but the cells are different due to the suppressed index in $P$. However the evaluation of $f$ is the same on both cells.}
\[
 \partial(E^j_a, N^i_a)=N_{a^j}^i+N_a^i + (E_{a^i}^j,P_{a^i}^i)+ (E_{a^i}^j,P_{a^i}^i) +(E_a^j,M_a^i) + (E_a^j,P_a^i).
\]
A similar calculation as before gives us
\begin{align*}
 \delta f(E^j_a, M^i_a) &= 1 + a_1+\cdots+a_{j-1}+\varepsilon_j \\
&= s(P^i_a)+s(M^i_a) + s(E^j_a).
\end{align*}
Replacing $N$ with $\tilde{N}$ has the effect that the summand $1$ disappears.

If $\tau = (N^j_a,E^i_a)$, then
\begin{align*}
 \delta f(N^j_a,E^i_a) &= f((P^j_{a^j},E^i_{a^j})+(P^j_a,E^i_a)) +f((P^j_{a^j},E^i_{a^j})+(M^j_a,E^i_a))+f(N^j_{a^i}+N^j_a)\\
&= (\delta_j+a_{j+1}+\cdots+a_{i-1}+\varepsilon_i) + (a_1+\cdots+a_{j-1})(a_j+\delta_j) + 0.
\end{align*}
Note that $a_j+\delta_j=1$ as we have $N^j_a$ rather than $\tilde{N}^j_a$, so
\[
 \delta f(N^j_a,E^i_a) = a_1+\cdots+a_{j-1}+ \delta_j + a_{j+1}+\cdots+a_{i-1}+\varepsilon_i = s(E_a^i)+1,
\]
which gives the result as before. Replacing $N$ with $\tilde{N}$ has the same effect as before, that is, the summand $1$ disappears.

It remains to check the case $\tau = Q^j_a$. As the boundary of $Q$ contains two cells $N$ and two cells $\tilde{N}$ based at $a$ and $a^j$, it is clear that $\delta f(Q^j_a)=0$.
\end{proof}

If $\tau$ is a product of three $1$-cells, there exist three objects $w_1,w_2,w_3$ with $|w_l|=|a|+1$ and three objects $t_1,t_2,t_3$ with $|t_l|+1=|b|$ which can be placed between $b$ and $a$.

These objects fit in a cube
\begin{equation}
 \label{diagram:three_cell}
\begin{split}
\xymatrix@!=1.5pc{  & & b \ar@{-}[ddll]|{c_3} \ar@{-}[dd]|{c_2+1} \ar@{-}[ddrr]|{c_1} \ar@{.}[ddddll]|{f_1} \ar@/_2pc/@{.}[dddd]|{f_2} \ar@{.}[ddddrr]|{f_3}\\
 & & & & \\ 
t_1 \ar@{-}[dd]|{c_2+1} \ar@{-}[ddrr]|(.25){c_1} \ar@{.}[ddddrr]|{g_1} & & t_2 \ar @{-}[dl]|(.75){c_3+1} \ar @{-}[dr]|(.75){c_1} \ar@/^2pc/@{.}[dddd]|{g_2} & & t_3 \ar@{-}[ddll]|(.25){c_3+1} \ar@{-}[dd]|{c_2} \ar@{.}[ddddll]|{g_3} \\
 & \ar@{-}[dl] & & \ar@{-}[dr] &  \\ 
w_1 \ar@{-}[ddrr]|{c_1} & & w_2 \ar@{-}[dd]|{c_2} & & w_3 \ar@{-}[ddll]|{c_3} \\
 & & & &  \\
 & & a & & }
\end{split}
\end{equation}
where $c_1=s(E^k_a)$, $c_2=s(E^j_a)$ and $c_3=s(E^i_a)$. The other labels for solid lines follow from the sign assignment using $k<j<i$. The labels for the dotted lines come from the frame assignment, namely $g_1=f(E^k_a,E^j_a)$, $g_2=f(E^k_a,E^i_a)$, $g_3=f(E^j_a,E^i_a)$, $f_1=f(E^j_{w_1},E^i_{w_1})$, $f_2=f(E^k_{w_2},E^i_{w_2})$ and $f_3=f(E^k_{w_3},E^j_{w_3})$. 

The component of $\mathcal{M}(b,a)$ corresponding to $\tau$ is determined as follows: The $1$-cells $E_a^k, E^j_a,E^i_a$ correspond to points that we denote by $e^k_a\in \M(w_1,a)$, $e^j_a\in \M(w_2,a)$, $e^i_a\in \M(w_3,a)$, respectively. They also give rise to points
\begin{align*}
 &e^j_{w_1}\in \M(t_1,w_1), e^i_{w_1}\in \M(t_2,w_1), e^k_{w_2}\in \M(t_2,w_2),\\ &e^i_{w_2}\in \M(t_3,w_2), e^k_{w_3}\in \M(t_2,w_3), e^j_{w_3}\in \M(t_3,w_3)
\end{align*}
and $e^i_{t_1}\in \M(b,t_1)$, $e^j_{t_2}\in \M(b,t_2)$, $e^k_{t_3}\in \M(b,t_3)$.

Dropping the subscripts of these points for brevity, the vertices of the hexagon corresponding to $\tau$ are described in Figure \ref{fig_hexagon}.

\begin{figure}[ht]
		{
			\psfrag{e1e2e3}{$e^ke^je^i$}
			\psfrag{e2e1e3}{$e^je^ke^i$}
			\psfrag{e1e3e2}{$e^ke^ie^j$}
			\psfrag{e2e3e1}{$e^je^ie^k$}
			\psfrag{e3e2e1}{$e^ie^je^k$}
			\psfrag{e3e1e2}{$e^ie^ke^j$}
			\includegraphics[height=4cm,width=4cm]{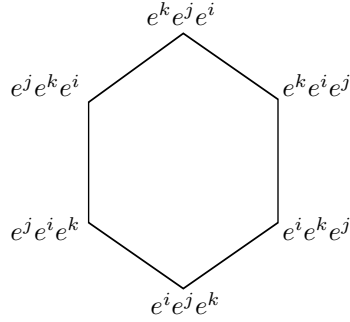}
		}
	\caption{The vertices of the hexagon corresponding to $\tau$.}
	\label{fig_hexagon}
\end{figure}

This hexagon is embedded in $\R^{d_l}\times [0,\infty)\times \R^{d_{l+1}} \times [0,\infty)\times \R^{d_{l+2}}$ for some values $d_l,d_{l+1},d_{l+2}$, and the boundary of this hexagon is, after rearranging coordinates and smoothing $\partial([0,\infty)\times [0,\infty))$, a framed submanifold of $\R^{d+1}$ with $d=d_l+d_{l+1}+d_{l+2}$, compare \cite[Fig.3.3]{LipSarSq}. In order to frame the whole hexagon, we need that the corresponding element of $\Omega_1^{fr}\cong \Z/2$ is the trivial element. Equivalently, we can show that the corresponding loop in $\SO(d+1)$ represents the non-trivial element in $H_1(\SO(d+1))\cong \pi_1(\SO(d+1))\cong \Z/2$.

The calculation of this element is completely analogous to the computation in the proof of \cite[Lm.3.5]{LipSarSq}. Namely, the loop in $\SO(d+1)$ represented by the framed boundary of the hexagon is homologous to the sum of the 18 loops shown in Figure \ref{fig_hexagonloop}.

\begin{figure}[h]
 \psfrag{top}{$c_1\,c_2+1\,c_3$}
 \psfrag{bot}{$c_3\,c_2\,c_1$}
 \psfrag{tol}{$c_2\,c_1\,c_3$}
 \psfrag{tor}{$c_1\,c_3+1\,c_2+1$}
 \psfrag{bol}{$c_2\,c_3+1\,c_1$}
 \psfrag{bor}{$c_3\,c_1\,c_2+1$}
 \psfrag{a1}{$a_1$}
 \psfrag{a2}{$a_2$}
 \psfrag{a3}{$a_3$}
 \psfrag{a4}{$a_4$}
 \psfrag{a5}{$a_5$}
 \psfrag{a6}{$a_6$}
 \psfrag{b1}{$b_1$}
 \psfrag{b2}{$b_2$}
 \psfrag{b3}{$b_3$}
 \psfrag{b4}{$b_4$}
 \psfrag{b5}{$b_5$}
 \psfrag{b6}{$b_6$}
 \psfrag{f1}{$f_1$}
 \psfrag{f2}{$f_2$}
 \psfrag{f3}{$f_3$}
 \psfrag{g1}{$g_1$}
 \psfrag{g2}{$g_2$}
 \psfrag{g3}{$g_3$}
 \includegraphics[height=6cm,width=7cm]{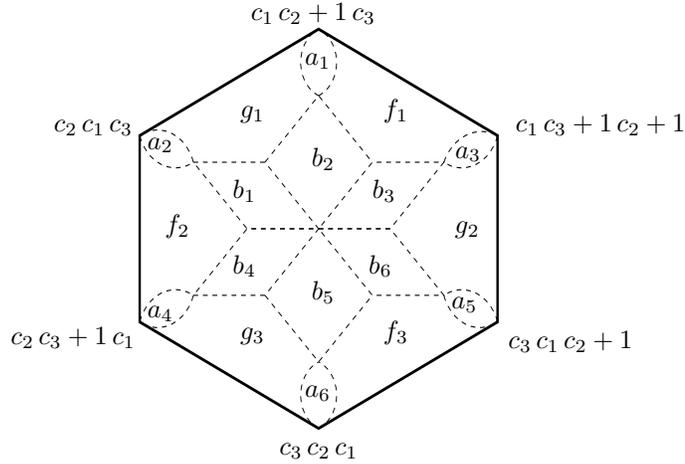}
 \caption{Loops in $\SO(d+1)$.}
 \label{fig_hexagonloop}
\end{figure}

Note that the vertices of the hexagon are represented by numbers $c\,b\,a$ that represent frames in $\{0\}\times \R^d$ given by
\[
 [0,(-1)^ce_1,e_2,\ldots,e_{d_l},(-1)^be_{d_l+1},e_{d_l+2},\ldots,e_{d_l+d_{l+1}},(-1)^ae_{d_l+d_{l+1}+1},\ldots,e_d]
\]
Each vertex in the interior is also represented by a number $c\,b\,a$ representing a frame in $\{0\}\times \R^d$. These numbers can be read off from \cite[Fig.3.4]{LipSarSq}. For example, the vertex in the middle is labeled $0\,0\,0$ referring to the standard basis. Each interior edge in the graph changes exactly one of the labels to $0$.

Each loop represents the value that is described in it, where the value is calculated in the proof of \cite[Lm.3.5]{LipSarSq} and can be read off \cite[Fig.3.4]{LipSarSq} in connection with (\ref{diagram:three_cell}). In particular, we get the values
\[
 \begin{array}{lcl}
  a_1 = c_2c_1+c_1+c_2+1 & & a_2 = c_1c_2+c_1 \\
  a_3 = c_3c_1 + c_1+c_3+1 & & a_4 = c_3c_2 + c_2 +c_3 +1 \\
  a_5 =c_1c_3+c_1 & & a_6 = c_2c_3+c_2.
 \end{array}
\]
We also get
\[
 \begin{array}{lclcl}
 b_1 = c_3c_2 & & b_2 = c_3c_1 & & b_3 = c_2c_1 + c_1 \\
 b_4 = c_1c_2 & & b_5 = c_1c_3 & & b_6 = c_2c_3 + c_3.
 \end{array}
\]
This implies
\begin{align*}
 a_1+\cdots + a_6 &= c_2 + 1\\
 b_1+\cdots + b_6 &= c_1 + c_3.
\end{align*}
From Lemma \ref{lem:coboundary} we get
\[
 f_1+f_2+f_3+g_1+g_2+g_3 = c_1+c_2+c_3
\]
and therefore the boundary of the hexagon, which represents the sum of all 18 loops in $\SO(d+1)$ does indeed represent the non-trivial element of $H_1(\SO(d+1))$, which is what we need to frame the hexagon.

If $\tau=(E_a^j,N_a^i)$, there are only four objects that fit between $a$ and $b$. However, we still get the cube (\ref{diagram:three_cell}) using $w_1=a^j$, $w_2=w_3=a^i$, $t_1=(a^j)^i=t_2$, $t_3=(a^i)^i$. Recall that for $a\in \Z^k$ we write $a^j=a+e_j\in \Z^k$. We still have the points $e^j_a\in \M(w_1,a)$, $e^{i_1}_a=P\in \M(w_2,a)$ and $e^{i_2}_a=M\in \M(w_3,a)$ which together with the shifted points make up the vertices of the hexagon corresponding to $\tau$.
This leads to the analogue of Figure \ref{fig_hexagon} and Figure \ref{fig_hexagonloop}, and the same calculation as in the case of $\tau=(E^m_a,E^j_a,E^i_a)$ shows that the framing can be extended to this hexagon. The fact that $e^{i_1}_a$ and $e^{i_2}_a$ have different sign means that $c_2+c_3=1$, which is irrelevant for the calculation. If we replace $N_a^i$ with $\tilde{N}_a^i$ the only difference is that $e^{i_1}_a$ and $e^{i_2}_a$ have the same sign, but the calculation remains the same.

The cases $\tau=(N^i_a,E_a^i)$ or $(\tilde{N}^i_a,E_a^i)$ are also completely analogous to the previous case.

It remains the case $\tau = Q^j_a$. In this case we only have four objects $a,w,t,b\in \Z^m$ with $\tau$ associated to a square contained in $\M(b,a)$. Also, these objects only differ in one coordinate.

The square has vertices
\begin{multline*}
 (P_i,P_{i,j},P_{i+1}),(M_i,P_{i,j+1},P_{i+1}), (M_i,P_{i,j+2},M_{i+1}), (P_i,P_{i,j+1},M_{i+1}) \\ \in \M(w,a)\times \M(t,w) \times \M(b,w).
\end{multline*}
As in the case of the hexagon, the boundary of the square is framedly embedded in $\R^{d+1}$ and we need to extend the framing to the whole square sitting in $[0,\infty)\times \R^{d+1}$.

The analogue of (\ref{diagram:three_cell}) is the slightly simpler diagram

\begin{equation}
 \label{diagram:three_cellb}
\begin{split}
\xymatrix@!=1.5pc{ & w \ar@{-}[rr]|{c} \ar@{-}[dl]|{c} \ar@{.}[rrrd]|(.75){g} \ar@{-}[rrdd]|(.625){c} & & t \ar@{.}[dlll]|(.75){f} \ar@{-}[dl]|(.75){c} \ar@{-}[rd]|{c}& \\
a \ar@{.}[rrrd]|(.25){f} \ar@{-}[rd]|{c+1} & & \ar@{-}[dl] & & b \ar@{.}[llld]|(.25){g} \ar@{-}[dl]|{c+1} \\
 & w \ar@{-}[rr]|{c} & & t }
\end{split}
\end{equation}
where $c=s(P_i)=s(P_{i,j})=s(P_{i+1})$ and $c+1=s(M_i)=s(M_{i+1})$, $f=f(N_{i,j})=f(N_{i,j+1})$, $g=f(\tilde{N}_{i,j})=f(\tilde{N}_{i,j+1})=0$.

The corresponding loop in $\SO(d+1)$ is represented in Figure \ref{fig_diamondloop}.

\begin{figure}[h]
\psfrag{f1}{$f$}
\psfrag{f2}{$f$}
\psfrag{c+cc+}{$(c+1)\,c\,(c+1)$}
\psfrag{0}{$0$}
\psfrag{c+}{$c+1$}
\psfrag{c}{$c$}
\psfrag{g1}{$g$}
\psfrag{g2}{$g$}
\psfrag{ccc}{$c\,c\,c$}
\psfrag{c+cc}{$(c+1)\,c\,c$}
\psfrag{ccc+}{$c\,c\,(c+1)$}
 \includegraphics[height=5cm,width=6cm]{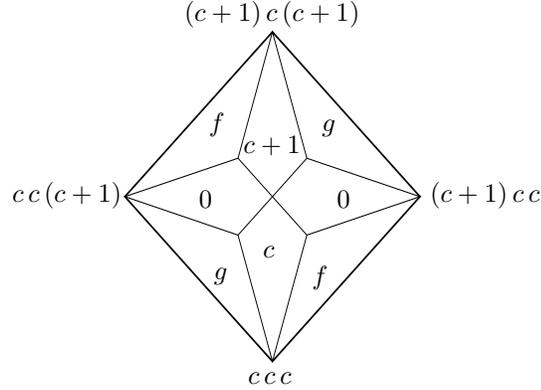}
 \caption{Loops in $\SO(d+1)$.}
 \label{fig_diamondloop}
\end{figure}
The four corners of the square are labelled by three numbers in $\Z/2$, for example $c\,c\,c$ at the bottom. This refers to the framing at that point in the same way it did for the hexagon above. The bottom point corresponds to $(P_i,P_{i,j},P_{i+1})$, similarly for the other points. Each edge represents a frame path joing frames at its endpoints. The point in the middle has framings $0\,c\,0$ while the four points around it have exactly one $0$.

The upper triangle on the left labelled $f$ is described as follows. The outside edge comes from the framing of the interval between the points $(P_i,P_{i,j+1},M_{i+1})$ and $(M_i,P_{i,j+2},M_{i+1})$, the two inside edges are standard paths between framings $c\,c\,(c+1)$ and $0\,c\,(c+1)$, and $(c+1)\,c\,(c+1)$ and $0\,c\,(c+1)$, respectively. The loop is therefore trivial if and only if $f$ is trivial. The other triangles are treated similarly.

For the upper $4$-gon, labelled $c+1$, we have two paths between the framings of $(c+1)\,c\,(c+1)$ and $0\,c\,0$, one by changing the first coordinate and then the third coordinate, and one where it is done in the different order. If $c+1=0$, then both are constant, but if $c=0$, then these two paths are different. This is because the standard path between $0\,0\,1$ and $0\,0\,0$ uses a different rotation than the standard path between $1\,0\,1$ and $1\,0\,0$.
The loop represented by this $4$-gon therefore represents $c+1\in \pi_1(\SO(d+1))$. A similar argument applies to the lower $4$-gon labelled $c$. The other two $4$-gons are in fact simpler, one checks that two of the four edges represent constant paths with the other cancelling each other, and the loop simply represents $0$.

The value of the outside loop is therefore just determined by summing the labels, and they clearly sum up to $1$, which means that the framing can be extended to the square.

\begin{proposition}
 The partial framing from Definitions \ref{def:signassign} and \ref{def:frameassign} can be extended to a framing of the category $\wt{\cS^n_\mathbf{r}}$.
\end{proposition}

\begin{proof}
 Our previous argument showed that all the $2$-dimensional moduli spaces in $\wt{\cS^n_\mathbf{r}}$ can be framed extending the standard sign- and frame assignments. The higher dimensional moduli spaces can then be framed using the obstruction theory of Section \ref{subsec:obstruction_theory}.
\end{proof}

\section{Examples}
\label{sec:examples}

\subsection{The (3,4)-torus knot by hand}
\label{subsec:example}
In this section, we consider the torus knot $T_{3,4}$ in the form of the pretzel knot $P(-2,3,3)$ (which appears as $8_{19}$ in the Rolfsen table), and use Theorem \ref{thm:Lip-Sar_equivalence} to calculate a non-trivial Steenrod square ``by hand''. The significance of the knot $8_{19}$ is that it is the first knot for which the Khovanov homotopy type is not a wedge sum of Moore spaces, yielding the nontriviality result in \cite[Theorem 1]{LipSarSq}. The $3$-component pretzel link $P(-2,2,2)$ is in fact the link with the least number of crossings that admits a non-trivial Steenrod square. The calculation is essentially the same as for $P(-2,3,3)$ as we shall see.

Consider the diagram of $T_{3,4} = P(-2,3,3)$ in Figure \ref{fig_pretzel-233}.
\begin{figure}[ht]
 \includegraphics*[scale=0.25]{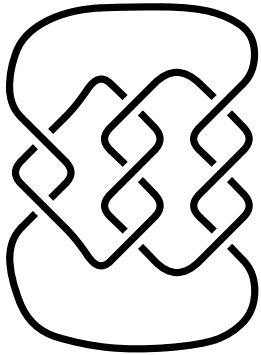}
 \caption{A glued diagram for $T_{3,4}$ with three elementary tangles.}
 \label{fig_pretzel-233}
\end{figure}

The flow category of this diagram has $31$ objects in quantum degree $11$. The cohomological degrees of these objects lie between $1$ and $4$, and are listed in Table \ref{objects_8_19}.
\begin{table}
	\caption{Objects of $\cL(D_{(-2,3,3)})$}
	\label{objects_8_19}
	\begin{center}
		\begin{tabular}{ | c | c | c |}
			\hline
			Object & Generator & Grading $| \; |$   \\ \hline
			$\mathbf{\alpha}_1$ & $(D_{(-1,0,0)}, x_+ x_+)$ & 1 \\ \hline
			$\mathbf{\alpha}_2$ & $(D_{(0,0,0)}, x_+ x_+ x_-)$ & 2 \\ \hline
			$\mathbf{\alpha}_3$ & $(D_{(0,0,0)}, x_+ x_- x_+)$ & 2 \\ \hline
			$\mathbf{\alpha}_4$ & $(D_{(0,0,0)}, x_- x_+ x_+)$ & 2 \\ \hline
			$\mathbf{\alpha}_5$ & $(D_{(-2,2,0)}, x_+)$ & 2 \\ \hline
			$\mathbf{\alpha}_6$ & $(D_{(-2,0,2)}, x_+)$ & 2 \\ \hline
			$\mathbf{\alpha}_7$ & $(D_{(-1,1,0)}, x_+)$ & 2 \\ \hline
			$\mathbf{\alpha}_8$ & $(D_{(-1,0,1)}, x_+)$ & 2 \\ \hline
			$\mathbf{\alpha}_9$ & $(D_{(-2,1,1)}, x_+x_+)$ & 2 \\ \hline
			$\mathbf{\alpha}_{10}$ & $(D_{(0,1,0)}, x_+x_-)$ & 3 \\ \hline
			$\mathbf{\alpha}_{11}$ & $(D_{(0,1,0)}, x_-x_+)$ & 3 \\ \hline
			$\mathbf{\alpha}_{12}$ & $(D_{(0,0,1)}, x_+x_-)$ & 3 \\ \hline
			$\mathbf{\alpha}_{13}$ & $(D_{(0,0,1)}, x_-x_+)$ & 3 \\ \hline
			$\mathbf{\alpha}_{14}$ & $(D_{(-1,1,1)}, x_+x_-)$ & 3 \\ \hline
			$\mathbf{\alpha}_{15}$ & $(D_{(0,0,1)}, x_-x_+)$ & 3 \\ \hline
			$\mathbf{\alpha}_{16}$ & $(D_{(-1,2,0)}, x_-)$ & 3 \\ \hline
			$\mathbf{\alpha}_{17}$ & $(D_{(-1,0,2)}, x_-)$ & 3 \\ \hline
			$\mathbf{\alpha}_{18}$ & $(D_{(-2,3,0)}, x_-)$ & 3 \\ \hline
			$\mathbf{\alpha}_{19}$ & $(D_{(-2,0,3)}, x_-)$ & 3 \\ \hline
			$\mathbf{\alpha}_{20}$ & $(D_{(-2,1,2)}, x_+x_-)$ & 3 \\ \hline
			$\mathbf{\alpha}_{21}$ & $(D_{(-2,1,2)}, x_-x_+)$ & 3 \\ \hline
			$\mathbf{\alpha}_{22}$ & $(D_{(-2,2,1)}, x_+x_-)$ & 3 \\ \hline
			$\mathbf{\alpha}_{23}$ & $(D_{(-2,2,1)}, x_-x_+)$ & 3 \\ \hline
			$\mathbf{\alpha}_{24}$ & $(D_{(0,1,1)}, x_-)$ & 4 \\ \hline
			$\mathbf{\alpha}_{25}$ & $(D_{(0,2,0)}, x_-x_-)$ & 4 \\ \hline
			$\mathbf{\alpha}_{26}$ & $(D_{(0,0,2)}, x_-x_-)$ & 4 \\ \hline
			$\mathbf{\alpha}_{27}$ & $(D_{(-1,2,1)}, x_-x_-)$ & 4 \\ \hline
			$\mathbf{\alpha}_{28}$ & $(D_{(-1,1,2)}, x_-x_-)$ & 4 \\ \hline
			$\mathbf{\alpha}_{29}$ & $(D_{(-2,1,3)}, x_-x_-)$ & 4 \\ \hline
			$\mathbf{\alpha}_{30}$ & $(D_{(-2,2,2)}, x_-x_-)$ & 4 \\ \hline
			$\mathbf{\alpha}_{31}$ & $(D_{(-2,3,1)}, x_-x_-)$ & 4 \\ \hline
		\end{tabular}
	\end{center}
\end{table}
The flow category, along with its $0$-dimensional moduli spaces, is displayed in Figure \ref{page0}. We use the notation of Section \ref{subsec_part_n_frames}, but also indicate the sign of the point. More precisely, we say $s(P)=0=s(-M)$, $s(-P)=1=s(M)$. The signs are sprinkled as prescribed by the standard sign assignment of Definition \ref{def:signassign}.

\begin{figure}[p]
	\includegraphics[angle = 90, scale=0.75]{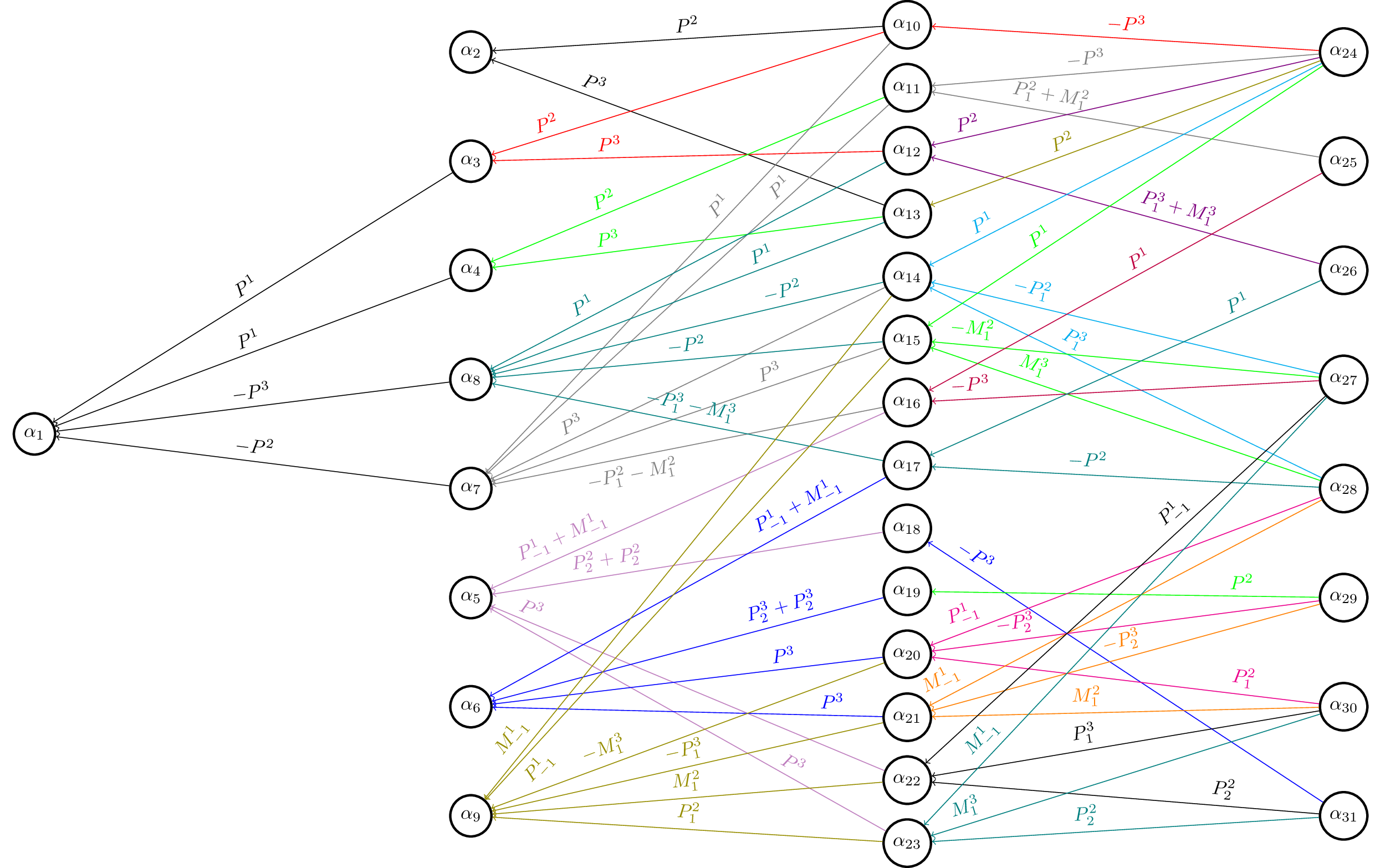}
	\caption{The flow category $\cL(D_{(-2,3,3)})$ of $8_{19} = P(-2,3,3)$.}
	\label{page0}
\end{figure}

Before we analyze the higher dimensional moduli spaces note that the full subcategory with objects $\alpha_{18},\alpha_{19},\alpha_{29},\alpha_{31}$ forms an upward closed subcategory in the sense of \cite[\S 3.4.2]{LipSarKhov}, which is contractible. Using \cite[Lm.3.32]{LipSarKhov} we can pass to the full subcategory without these objects without changing the stable homotopy type. Notice that the resulting category is the same as if we had started with the pretzel link $P(-2,2,2)$ in quantum degree $-3$.

The full subcategory generated by the objects $\alpha_1,\alpha_8$ forms a contractible downward closed subcategory, and after passing to the quotient category, the objects $\alpha_2,\alpha_3,\alpha_{10},\alpha_{12},\alpha_{13}$ and $\alpha_{24}$ form a contractible upward closed subcategory.

The resulting flow category, which we shall denote by $\Cat$, contains $19$ objects and is displayed with its $0$-dimensional moduli spaces in Figure \ref{page1}.
\begin{figure}
	\includegraphics[scale=0.5]{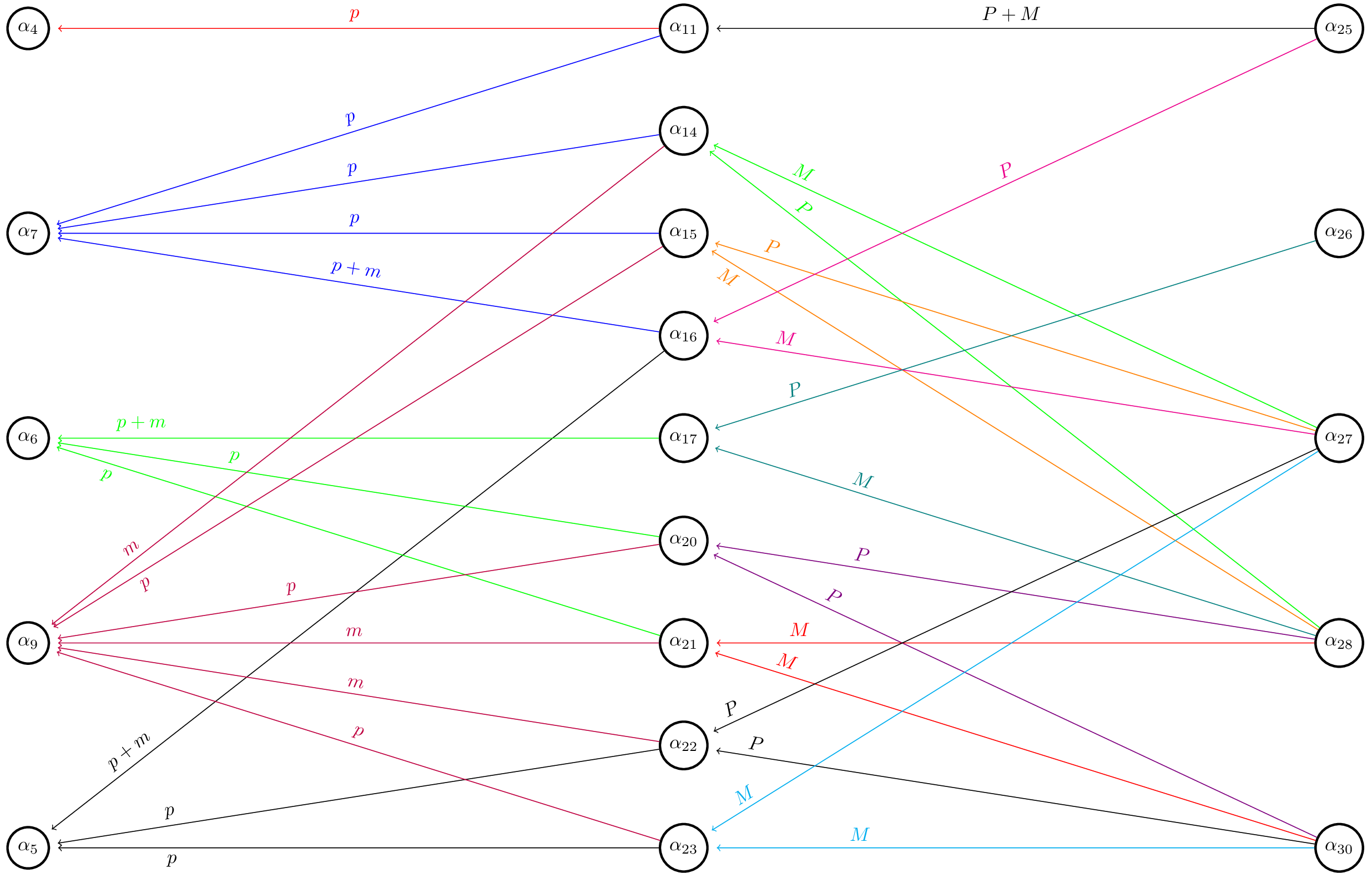}
	\caption{The flow category $\Cat$.}
	\label{page1}
\end{figure}
The $1$-dimensional moduli spaces have frame assignments as prescribed in Definition \ref{def:frameassign} and are listed in Figure \ref{mod_spaces_1}. In order to calculate the framings of each of the intervals, one must refer to the $P_{j}^{i}$, $M_j^i$ notation of the $0$-dimensional moduli spaces in Figure \ref{page0}. However, once these framings have been calculated, this notation is unnecessary and in Figure \ref{page1} we simply refer to a plus (respectively a minus) sign using either a $p$ or $P$ (respectively either a $m$ or $M$). The boundary points of the $1$-dimensional moduli spaces (which are listed in Figure \ref{mod_spaces_1}) are labelled using this notation. The object of the flow category where the boundary breaks is highlighted in blue, and the value of the frame is highlighted in red.
\begin{figure}
	\includegraphics[scale=1]{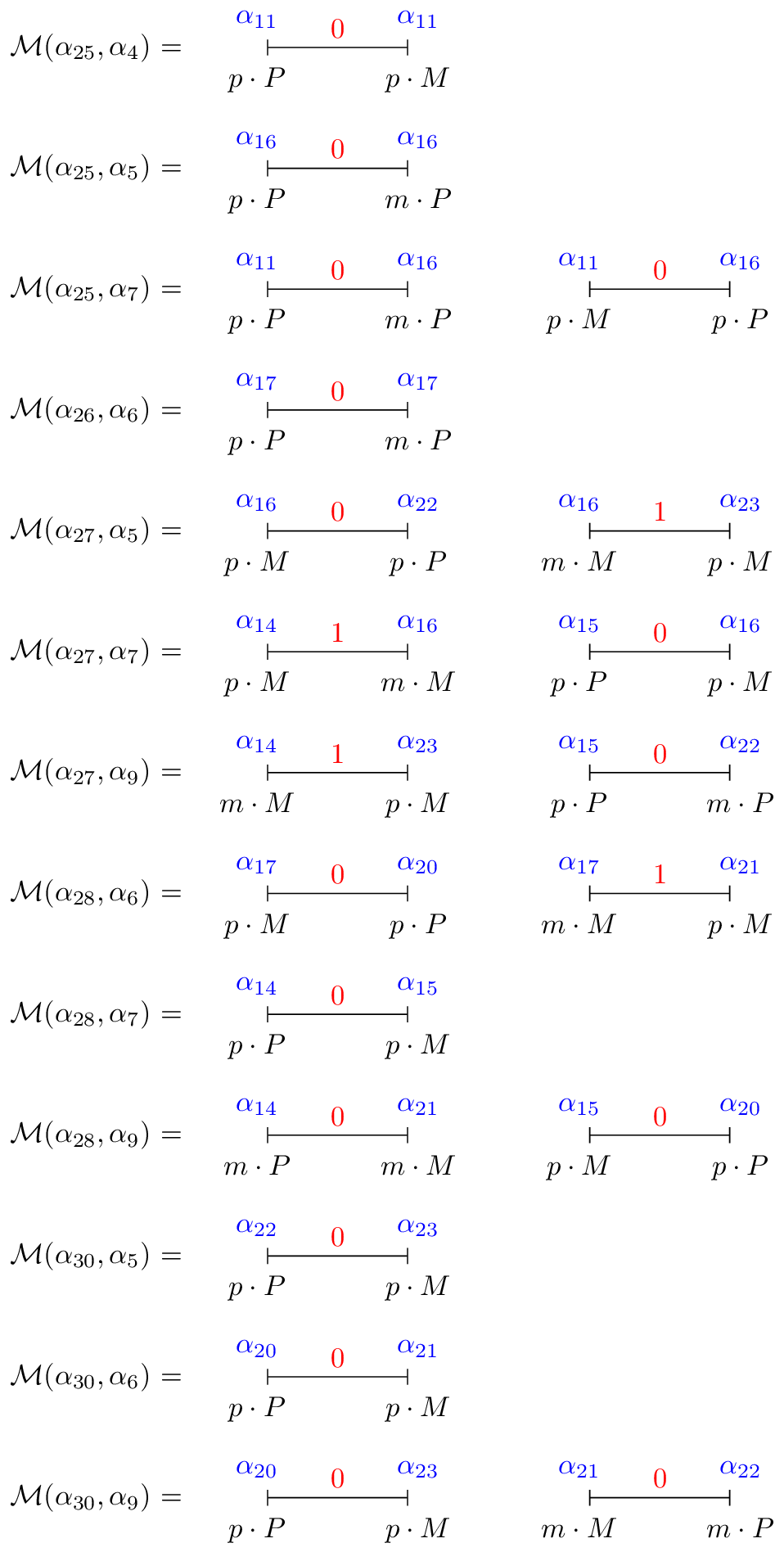}
	\caption{$1$-dimensional moduli spaces for $\Cat$.}
	\label{mod_spaces_1}
\end{figure}

To calculate the Steenrod square notice that the cocycle generating the non-trivial element in $H^2$ is the sum of all five objects of cohomological degree $2$. Notice also that all objects of cohomological degree $4$ are cohomologous, and every such object represents the non-trivial cohomology class in $H^4$.

We need to choose a topological boundary matching and we use the one highlighted in Figure \ref{tbm1}. Notice in particular that there are three pairs of choices for $\eta_{16}$, and we have chosen a boundary-coherent matching of the two points corresponding to $\alpha_5$, one with a positive and the other with a negative (similarly for $\alpha_7$).
\begin{figure}
	\includegraphics[scale=1]{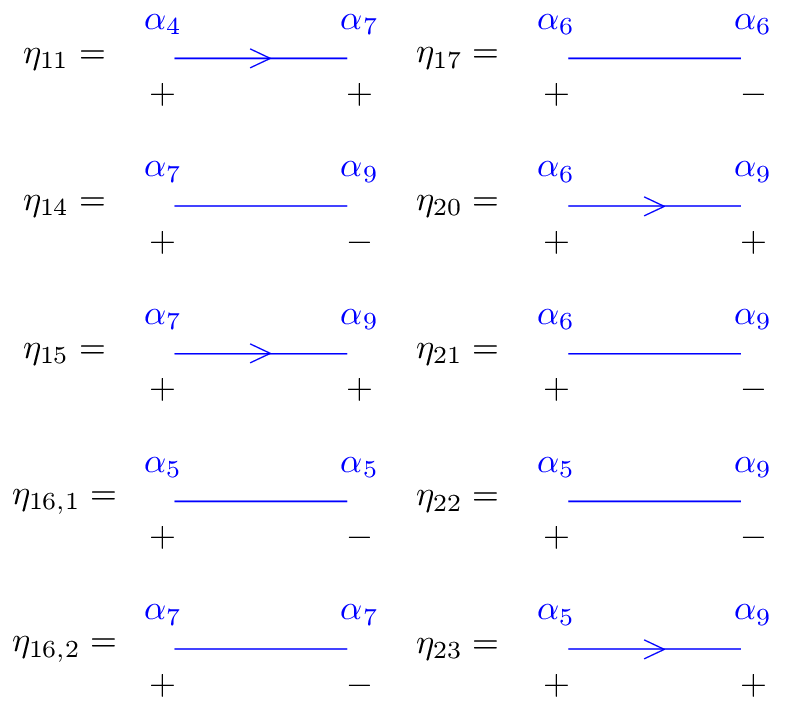}
	\caption{Choice of topological boundary matchings for $\Cat$.}
	\label{tbm1}
\end{figure}
According to these choices of topological boundary matchings, the Steenrod square can be computed using Proposition \ref{prop:add_form_frame} in the following way. Let $u \in H^2(\Cat;\Z/2)$ be the cohomology class represented by the cocycle $c \in C^2(\Cat;\Z/2)$ given by
\[
c = \alpha_4 + \alpha_5 + \alpha_6 + \alpha_7 + \alpha_9 {\rm .}
\]
Following the construction explained in Section \ref{subsec:steen_fl_cat}, $\Sq^2(u)$ is given by the sum of the values assigned to each framed circle $(K,\Phi)$ inside $\partial \mathcal{C}(z)$ for each $z \in \{ \alpha_{25}, \alpha_{26}, \alpha_{27}, \alpha_{28}, \alpha_{30} \}$. The computation of the Steenrod square is illustrated in Figure \ref{computing_sq2} and the individual components of the sum are listed below (where we count arrows in the clockwise direction and sum in the order corresponding to Proposition \ref{prop:add_form_frame}):

The component from $\partial \mathcal{C}(\alpha_{25})$ is $1+0+1 \equiv 0$ (mod $2$) for the outside circle, and $1+0+0 \equiv 1$ (mod $2$) from the inside circle, giving a total of $1$ (mod $2$).

The component from $\partial \mathcal{C}(\alpha_{26})$ is $1+0+0 \equiv 1$ (mod $2$).

The component from $\partial \mathcal{C}(\alpha_{27})$ is $1+3+2 \equiv 0$ (mod $2$).

The component from $\partial \mathcal{C}(\alpha_{28})$ is $1+1+1 \equiv 1$ (mod $2$).

The component from $\partial \mathcal{C}(\alpha_{30})$ is $1+0+1 \equiv 0$ (mod $2$).

Thus giving a non-trivial Steenrod square $\Sq^2(u) = 1 \in H^4(\Cat;\Z/2)$.

\begin{figure}
	\includegraphics[scale=0.85]{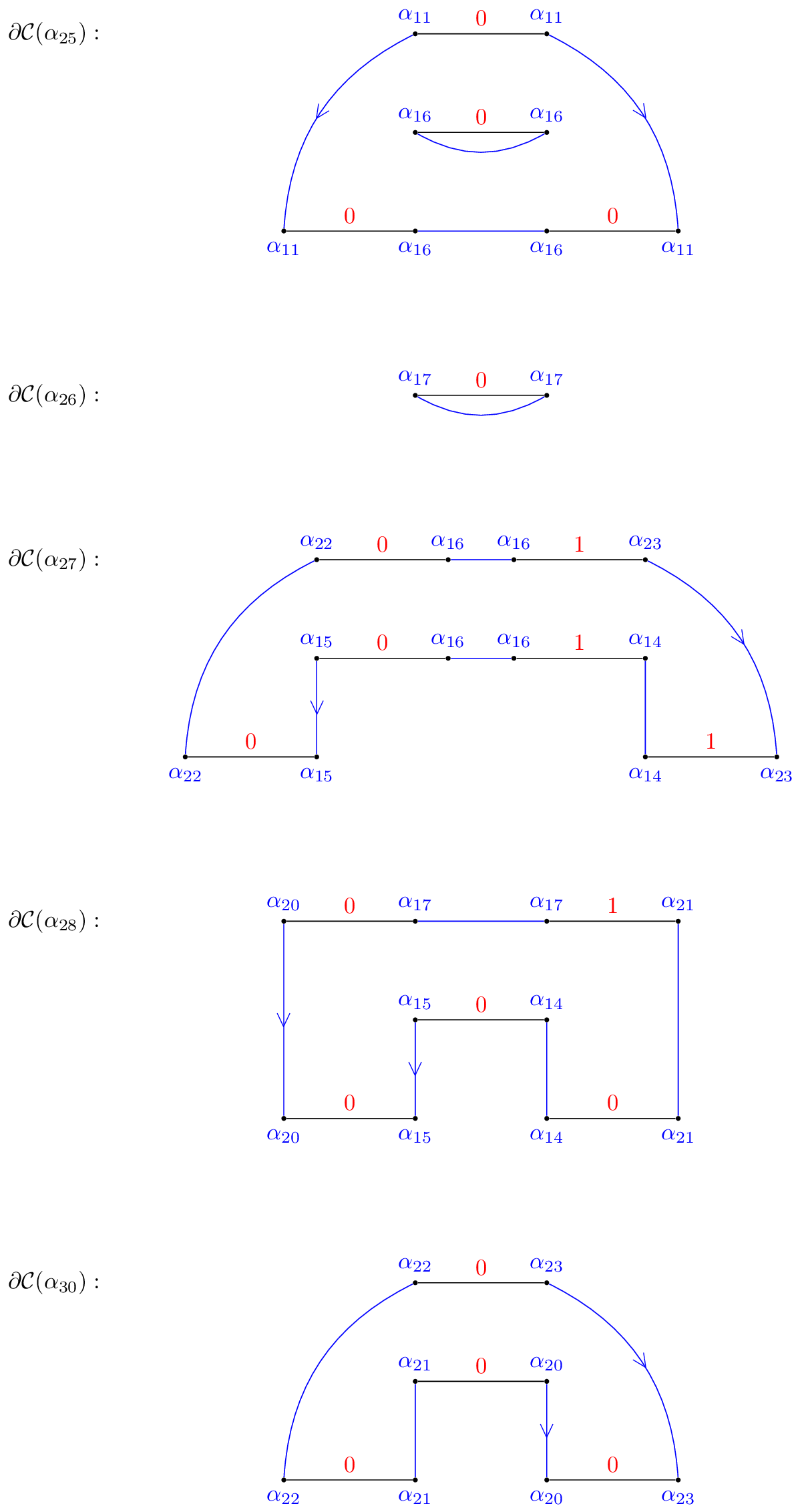}
	\caption{Computation of the Steenrod square for $\Cat$.}
	\label{computing_sq2}
\end{figure}

\subsection{Computer calculations}
\label{subsec:calculations}
The following computer calculations have been performed utilizing the techniques of this paper. The programme together with documentation can be found at the address

{\tt $\,$ http://www.maths.dur.ac.uk/~dma0ds/knotjob.html }

To describe the stable homotopy types that arise, we use a notation similar to the one used by Baues \cite{baues} to describe the \em elementary Chang complexes\em. Let $\eta\colon S^{n+1}\to S^n$ be the generator of $\pi_{n+1}(S^n)$. We assume $n\geq 3$ here, but will de-suspend later to also allow any $n\in \Z$. The CW-complex $X(\eta)$ is obtained from $S^n$ by attaching an $(n+2)$-cell via $\eta$.
We now assume that $p,q\geq 2$ are powers of $2$. Define $X(_p\eta)$ by attaching an $(n+1)$-cell to $S^n$ by a degree $p$ map, and attaching an $(n+2)$-cell via $\eta$.

Also define $X(\eta q)$ by attaching an $(n+2)$-cell to $S^{n+1} \vee S^n$ using a degree $q$ map to attach to $S^{n+1}$, and $\eta$ to attach to $S^n$. Finally let $X(_p\eta q)$ be obtained by attaching an $(n+1)$-cell to $X(\eta q)$ using a degree $p$ map.

To indicate the first non-trivial integral homology group, we use the notation
\[
 X(\eta,n),\, X(_p\eta,n),\, X(\eta q,n), \, X(_p\eta q,n)
\]
for the elementary Chang complexes. Here we allow $n\in \Z$ using appropriate de-suspension. To simplify comparison to \cite{LipSarSq}, note that
\begin{align*}
 X(\eta,2) & \simeq \CP^2 \\
 X(_2\eta,3) & \simeq \RP^5/\RP^2 \\
 X(\eta 2,2) & \simeq \RP^4/\RP^1 \\
 X(_2\eta 2,2) & \simeq \RP^2 \wedge \RP^2.
\end{align*}

In the following examples we write $\X^n_j(D)$ to mean the quantum degree $j$ summand of $\SL_n$ stable homotopy type associated to the diagram $D$.

\begin{example}
 Consider the pretzel link $P(-2,2,2)$, which is the only link with at most $6$ crossings to have a Khovanov stable homotopy type that is not a wedge of Moore spaces (see \cite{LipSarSq}). As this link has an obvious matched diagram, we can obtain $\SL_n$-calculations for $n\geq 3$. The integral $\SL_3$-cohomology has been calculated by Lewark, and can be found at

{\tt $\,$ http://lewark.de/lukas/foamho.html }

The link is listed there as L06n001, and has several quantum degrees of cohomological width $3$, and one, $q=-4$, of cohomological width $4$. Steenrod square calculations show that there is a single non-trivial Steenrod square in quantum degree $-6$. As the cohomology is free abelian in this degree, this demonstrates that there is a $\CP^2$ summand in the wedge decomposition. More precisely, we have
\[
 \X^3_{-6}(P) \simeq S^2 \vee S^2 \vee X(\eta,0).
\]
We note that no $\CP^2$ summand has yet been detected in the Lipshitz-Sarkar stable homotopy type, but this may well be due merely to a lack of computations so far.

The calculation only requires about $30$ objects in the flow category, and we shall give a calculation along the lines of Section \ref{subsec:example} in an forthcoming paper.

In all other quantum degrees the stable homotopy type is a wedge of Moore spaces. This includes $q=-4$, as width $4$ is obtained via $3$-torsion, which does not contribute to more complicated homotopy types by the classification result due to Baues and Hennes \cite{bauhen}.

For $n=4$ there are several non-trivial Steenrod squares, namely $St(0,q,4) = ( 1, 0, 0, 0)$ for $q=-5,-7,-9$. This means there is a $X(\eta,0)$ summand in quantum degrees $q=-5$ and $-7$ for width reasons. For $q=-9$ we cannot conclude this, as there is non-zero cohomology in cohomological degree $4$ as well.
\end{example}

\begin{example}
 Consider the pretzel link $P(2,-3,5)$, which is $10_{125}$ in the Rolfsen knot table. This knot is known to have a matched diagram, and one such diagram $D$ is given in Figure \ref{fig:pret_2-35}. The Khovanov cohomology of this knot is thin, so the stable homotopy type is necessarily a wedge of Moore spaces.
\begin{figure}[ht]
\includegraphics[height=6cm,width=4cm]{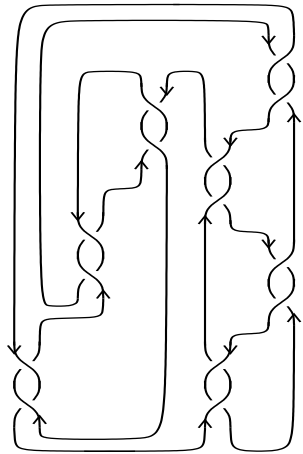}
\caption{A matched diagram of $P(2,-3,5)$.}
\label{fig:pret_2-35}
\end{figure}
The $\SL_3$ Khovanov-Rozansky cohomology has several quantum degrees which have potential non-trivial second Steenrod square, but calculations show that actually none of them are non-trivial.

The $\SL_4$ Khovanov-Rozansky cohomology has non-trivial second Steenrod squares in quantum degrees $q=-1,1,3$. For $q=3$ the cohomological width is $5$, so we cannot identify the stable homotopy type with our current methods. Integral cohomology calculations show
\[
 \X^4_1(D) \simeq S^{-2} \vee S^{-1} \vee S^0 \vee S^0 \vee X(\eta 2,-2).
\]
For $q=-1$ the integral homology $H_{-2}(\X^4_{-1}(D)) = \Z \oplus \Z/4$ so determining the stable homotopy type requires calculation of a certain Bockstein homomorphism which has not been implemented yet. Cohomological width is $3$ in this quantum degree.
\end{example}

\begin{example}
 The knot diagram in Figure \ref{fig:tor_74} was obtained by performing several Reidemeister moves on the torus knot $T_{4,7}$. We do not reproduce these moves here, but the keen reader may use information and programmes on the Knot Atlas to verify at least that this knot has the same Khovanov cohomology as $T_{4,7}$.

\begin{figure}[ht]
\psfrag{-2}{-$2$}
\psfrag{5}{$5$}
\psfrag{3}{$3$}
\psfrag{2}{$2$}
\includegraphics[height=6cm,width=4cm]{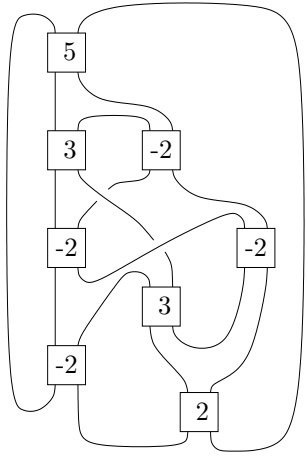}
\caption{A glued diagram of $T_{4,7}$ with ten elementary tangles.}
\label{fig:tor_74}
\end{figure}

Non-trivial second Steenrod squares occur in quantum degrees $q=23,27,29,31,$ $33,35,37$ and $39$. From the Khovanov cohomology listed in the Knot Atlas and the action of the Steenrod algebra we obtain that
\begin{align*}
 \KhSpace_{23}(T_{4,7}) & \simeq X(_2\eta,2) \\
 \KhSpace_{27}(T_{4,7}) & \simeq X(\eta 2,5) \vee S^6 \\
 \KhSpace_{29}(T_{4,7}) & \simeq X(\eta 2,5) \vee S^7 \vee S^8 \\
 \KhSpace_{31}(T_{4,7}) & \simeq S^7 \vee M(\Z/4,8) \vee X(_2\eta,8)\\
 \KhSpace_{33}(T_{4,7}) & \simeq S^9 \vee M(\Z/2,9) \vee X(\eta2,9) \vee S^{10}\\
 \KhSpace_{35}(T_{4,7}) & \simeq X(_2\eta,9) \vee X(_2\eta,10) \vee S^{11}\\
 \KhSpace_{37}(T_{4,7}) & \simeq S^{11} \vee X(_2\eta4,11)\\
 \KhSpace_{39}(T_{4,7}) & \simeq X(\eta 2,12) \vee S^{13}.
\end{align*}
Note that for $q = 37$ there is a summand previously undetected in the Lipshitz-Sarkar stable homotopy type.  For $q=29,31$ and $35$ we need to use the Classification Theorem of Baues and Hennes \cite{bauhen}, see also \cite{baues}. In particular \cite[Cor.11.18]{baues} lists all indecomposable $(n-1)$-connected $(n+3)$-dimensional homotopy types with cyclic integral homology groups. Together with the action of the Steenrod algebra we can identify the stable homotopy type in these cases.
Note that for $q=25$ we cannot identify the stable homotopy type from Steenrod square information alone.
\end{example}

\begin{example}
 A similar calculation can be done for the torus knot $T_{4,5}$. We get non-trivial second Steenrod square in quantum degrees $q=17,21,23,25$. More precisely, the non-trivial part of the Lipshitz-Sarkar Steenrod link invariant $\St(T_{4,5})$, see \cite[Def.4.3]{LipSarSq} is given by
\begin{align*}
 \St(2,17) & = (0,1,0,0) \\
 \St(5,21) & = (0,0,1,0) \\
 \St(5,23) & = (0,0,1,0) \\
 \St(7,25) & = (1,0,0,0).
\end{align*}
Together with the integral cohomology information and \cite[Cor.11.18]{baues} for $q=23$ we get
\begin{align*}
 \KhSpace_{17}(T_{4,5}) & \simeq X(_2\eta,2) \\
 \KhSpace_{21}(T_{4,5}) & \simeq X(\eta 2,5)\vee S^6 \\
 \KhSpace_{23}(T_{4,5}) & \simeq X(\eta 2,5) \vee S^7 \vee S^8 \\
 \KhSpace_{25}(T_{4,5}) & \simeq X(\eta 4,7).
\end{align*}
Note that $q=25$ is a stable homotopy type previously undetected.  For $q=19$ we cannot identify the stable homotopy type from Steenrod square information alone. In all other quantum degrees the stable homotopy type is a wedge of Moore spaces.
\end{example}

\bibliographystyle{myamsalpha}
\bibliography{master_sln}
\end{document}